\documentclass[a4paper,11pt]{amsart}

\usepackage{amsmath, amsthm, amsfonts, amssymb, enumerate}

\usepackage[bookmarks=false, hidelinks]{hyperref}

\title{On ultraproduct approximations and property (T) factors}
\author{Jesse Peterson}
\address{Department of Mathematics, Vanderbilt University, 1326 Stevenson Center, Nashville, TN 37240, USA}
\address{Department of Pure Mathematics, University of Waterloo, 200 University Avenue West, Waterloo, Ontario, N2L 3G1, Canada}
\email{jesse.peterson@uwaterloo.ca}

\newtheorem{thm}{Theorem}[section]
\newtheorem{prop}[thm]{Proposition}
\newtheorem{cor}[thm]{Corollary}
\newtheorem{lem}[thm]{Lemma}
\theoremstyle{definition}
\newtheorem{rem}[thm]{Remark}
\newtheorem{examp}[thm]{Example}
\newtheorem*{note}{Notation}

\newcommand{\Ad}{\operatorname{Ad}}
\newcommand{\Aut}{\operatorname{Aut}}
\newcommand{\SL}{\operatorname{SL}}
\newcommand{\PSL}{\operatorname{PSL}}
\newcommand{\Tr}{\operatorname{Tr}}
\newcommand{\ovt}{\, \overline{\otimes}\,}
\newcommand{\oovt}[1]{\, \overline{\otimes}_{#1}\,}

\begin{document}

\begin{abstract}
We introduce a framework allowing for key aspects of deformation/rigidity theory to be used in the study of continuous model theory of II$_1$ factors. Using this framework, we solve several well-known open problems in the area. For example, we show that the group von Neumann algebras $L(\SL_3(\mathbb Z))$ and $L \mathbb F_2$ are not elementarily equivalent, and we show that the group von Neumann algebra $L\mathbb F_2$ is not pseudomatricial. We also show a Bass-Serre type strong rigidity result in the setting of ultraproducts to provide an infinite family of pairwise non-elementarily equivalent full factors, each of which embeds into an ultraproduct of the hyperfinite II$_1$ factor. Building on previous work of Boutonnet, Chifan and Ioana, we also provide a continuum of pairwise non-elementarily equivalent full factors, which can be taken to be group von Neumann algebras.
\end{abstract}

\maketitle

\section{Introduction}

The foundations for the model theoretic study of II$_1$ factors were laid out by Farah, Hart and Sherman in the series of papers \cite{FaHaSh13, FaHaSh14a, FaHaSh14} where, motivated by a question of Popa on the isomorphism of matrix ultraproducts, they introduced a natural language for which the class of II$_1$ factors becomes axiomatizable. In \cite{FaHaSh14} they introduce the notion of elementary equivalence between two factors, written $M \equiv N$, by requiring that they satisfy the same first order sentences. By the continuous version of the Keisler-Shelah Theorem, this is equivalent to asking that there exist some ultrafilters $\mathcal U$ and $\mathcal V$ on some sets so that we have an isomorphism $M^{\mathcal U} \cong N^{\mathcal V}$. 

Ultrapowers have long played a prominent role in the development of the theory of II$_1$ factors and, in particular, were used to great effect by McDuff when she exhibited the first infinite families of pairwise non-isomorphic II$_1$ factors \cite{Mc69a, Mc69}. Historically, ultraproducts of II$_1$ factors have been used mostly as a tool for studying the central sequence algebra, and it wasn't until the introduction of the model theory of II$_1$ factors that the question of classifying II$_1$ factors up to isomorphism of ultrapowers, that is, up to elementary equivalence, became of importance (see \cite{Go23book} for a survey of the developments in this direction over the last 15 years).

A natural question then arises in the study of model theory of II$_1$ factors, which is to find invariants to distinguish the complete theories. Initially, progress in this direction was slow and mimicked the earlier development of II$_1$ factors. It was shown in \cite{FaHaSh14} that property Gamma and the McDuff property were invariants up to elementary equivalence so that at least three complete theories were known to exist. Much as McDuff's result revealed a rich and complicated structure to the theory of II$_1$ factors, Boutonnet, Chifan and Ioana showed in \cite{BoChIo17} that there is also a rich and complicated structure to the model theory of II$_1$ factors. In their work, they exhibited a continuum of pairwise non-elementarily equivalent II$_1$ factors; in fact, somewhat surprisingly, the factors considered in \cite{BoChIo17} are the same as those considered originally by McDuff. 

Since Boutonnet, Chifan and Ioana considered the same factors of McDuff, the central sequence algebra likewise played a prominent role in their work. The natural question then arose as to what can be done in the setting of full factors where there are no non-trivial central sequences. Much of the difficulty in the setting of full factors is that some of the most powerful tools to distinguish full factors up to isomorphism invoke some form of approximation property, or lack thereof. Examples of this include extensively studied notions such as (relative) property (T) \cite{Co80, CoJo85, BeHaVa08}, the Haagerup property \cite{Ch83, ChCoJoJuVa01}, or Cowling and Haagerup's weak amenability \cite{CoHa89}. Juxtaposing approximation properties with rigidity properties forms the cornerstone of the powerful deformation/rigidity theory, which has been used over the last 25 years to provide some of the most striking rigidity phenomena up to isomorphisms (see, e.g., the surveys \cite{Po07a, Va10, Io18}).

Exploiting these techniques in the study of elementary equivalence is challenging because elementary equivalence does not preserve these approximation properties, as can be seen by the basic fact from \cite{FaHaSh14} that if $N$ is any separable II$_1$ factor and if $M \subset R^{\mathcal U}$ is any separable subfactor, then there exists a II$_1$ factor $\tilde N$ such that $N \equiv \tilde N$ and $M$ embeds into $\tilde N$. Since there is no universal separable factor containing all separable factors that embed into $R^\mathcal U$ \cite{NiPoSa07}, it follows that for every complete theory $\operatorname{Th}(N)$ there are uncountably many pairwise non-isomorphic II$_1$ factors $M$ such that $M \equiv N$. This means that it is also not possible to use separability arguments such as in \cite{Co80, Oz04, NiPoSa07} to exploit rigidity properties such as property (T).  There have, however, been applications of property (T) in other aspects of the model theory of II$_1$ factors, mainly through the use of its spectral gap characterization \cite{Go23b}.

The Connes Embedding Problem asked if every separable II$_1$ factor embedded into $R^{\mathcal U}$. One of the most studied problems in all of operator algebras, a negative solution was recently obtained in \cite{JiNaViWrYu20} using quantum complexity theory. Thus, there is some separable II$_1$ factor $M$ such that $M$ does not embed into $R^{\mathcal U}$, and it follows that $M * L\mathbb Z$ is a full factor that is not elementarily equivalent to any $R^{\mathcal U}$ embeddable full factor such as $L\mathbb F_2$. Exploiting the results in \cite{JiNaViWrYu20}, it was shown by Goldbring and Hart in \cite[Corollary 5.5]{GoHa24} that one may obtain a decreasing sequence of separable II$_1$ factors $(N_n)_n$ so that for each $n \geq 1$, $N_n$ does not embed into $N_{n + 1}^{\mathcal U}$. Since $N_n * L\mathbb Z$ always embeds into $N_n^{\mathcal U}$ \cite{Po95}, it then follows that $\{  N_n * L\mathbb Z \}_n$ gives a family of pairwise non-elementarily equivalent full II$_1$ factors. The factors $N_n$, however, are far from explicit. For example, it is still a major open problem if one can take such factors to be group von Neumann algebras. 

A simpler construction of non-elementarily equivalent full factors was found in \cite{ChIoKE23}, where the authors use a procedure involving inductive limits of amalgamated free products to produce an ``exotic'' II$_1$ factor $N$ that is not elementarily equivalent to the free group factor $L\mathbb F_2$. Their method for showing  that $N$ is not elementarily equivalent to $L\mathbb F_2$ is that their factor has ``sequential commutation'', in the sense that for any two unitaries $u_1, u_2 \in \mathcal U( N^{\mathcal U} )$ satisfying $u_1^2 = u_2^3 = 1$ and $\{ u_1 \}'' \perp \{ u_2 \}''$, there exist Haar unitaries $v_1, v_2 \in \mathcal U( N^{\mathcal U})$ such that $[u_1, v_1] = [v_1, v_2] = [v_2, u_2] =  0$. The free group factor $L\mathbb F_2$ cannot have this property due to entropy constraints from \cite{Vo94, Ju07a, Ha18} (see  \cite{HoIo24}, or also \cite{JeKE26}, where an obstruction is obtained that does not rely upon entropy considerations). The construction of full factors with sequential commutation has since been simplified and refined in \cite{KEP25, GaKEPa25, GaJeKEPa26}, though the simplest process still involves taking successive inductive limits of amalgamated free products. In particular, it is not clear if such factors can be group factors, and it is not clear if this process leads to more than two elementary equivalence classes of full factors. Also, as noted in \cite[p.\ 1246]{ChIoKE23} (see also \cite[Question 4.4]{GoHa23} and \cite[Problems 16, 17]{KE25}), even with these recent advances, it is still open as to whether or not there is more than one elementary equivalence class of full factors that embed into $R^{\mathcal U}$. We are able to address this question here.

\begin{thm}\label{thm:mainfg}
The free group factors $L\mathbb F_n$  are not elementarily equivalent to any property (T) factor. In particular, we have $L\mathbb F_n \not\equiv L(\PSL_m(\mathbb Z))$ for $n \geq 2$ and $m \geq 3$. 
\end{thm}

This result also sheds some light on pseudomatricial factors. A II$_1$ factor $N$ is said to be pseudomatricial if there exists an ultrafilter $\mathcal U$ such that $N \equiv \prod_{\iota \to \mathcal U} \mathbb M_{n_\iota}(\mathbb C)$. Equivalently, $N$ has some ultrapower that is isomorphic to an ultraproduct of matrix algebras. This property was defined more generally in the continuous model theory setting as pseudocompact in \cite{GoLo15}, generalizing the notion of pseudofiniteness from discrete model theory, though this terminology may be somewhat confusing in the setting of von Neumann algebras and so we use the more descriptive terminology pseudomatricial here. Obviously, pseudomatricial II$_1$ factors embed into $R^{\mathcal U}$. It was shown in \cite[Section 5]{FaHaSh14}, using von Neumann's result in \cite{VN42}, that pseudomatricial factors are also full (see also \cite{FaHa11}). This was also earlier observed by Popa, who instead used irreducible representations of property (T) groups. 

Using techniques similar to those in Theorem~\ref{thm:mainfg}, we are able to use irreducible representations of property (T) groups in a manner similar to Popa in order to give the first examples of full II$_1$ factors that embed into $R^{\mathcal U}$ but are not pseudomatricial. In fact, we can show that free group factors themselves are not pseudomatricial, answering a well-known question in von Neumann algebras that has been considered for some time by experts, and has more recently been asked explicitly in \cite[Question 4.6]{GoHa23}, \cite[Question 5.1]{GoPi25}, or \cite[Problems 3, 4]{KE25}.

\begin{thm}\label{thm:freenotpseudomatricial}
The free group factors $L\mathbb F_n$ are not pseudomatricial for $1 < n \leq \infty$.
\end{thm}

Our proof shows that a pseudomatricial factor cannot even existentially embed into an ultrapower of a free group factor, nor can it have the same $\forall \exists$-theory, answering \cite[Question 5.4]{GoPi25}. This should be contrasted with \cite[Remark 4.5]{JeKE26} where it is shown that pseudomatricial factors and free group factors agree on certain $\exists \forall$-sentences. Whether or not a free group factor can existentially embed into a pseudomatricial factor remains an interesting open problem (see \cite[Problem 13]{KE25}).

Theorems~\ref{thm:mainfg}  and \ref{thm:freenotpseudomatricial} are obtained by overcoming some of the previous difficulties that prevented the incorporation of deformation/rigidity techniques in the ultraproduct setting.  On the deformation side, we find a quantitative version of Connes's characterization of amenability from \cite{Co76} that states that if a factor is close to possessing central vectors in the coarse bimodule, then it must be close to being contained in a finite-dimensional matrix subalgebra. In terms of ultraproducts, this means that if we have a separable subfactor $M \subset N^{\mathcal U}$ such that the ultrapower of the coarse bimodule $( L^2N \overline \otimes L^2 N)^{\mathcal U}$ contains $M$-central vectors, then $M$ itself must already be contained in an ultraproduct of matrix algebras $M \subset \prod_{\iota \to \mathcal U} \mathbb M_{n_\iota}(\mathbb C) \subset N^{\mathcal U}$ (see Theorem~\ref{thm:amenimplieshyperfinite} for the precise statement). To prove this result, we have to go through Connes's proof that semi-discreteness implies hyperfiniteness for II$_1$ factors. Connes's proof builds on the work of Effros and Lance in \cite{EfLa77} and involves an intricate and technical analysis of the automorphisms of a semi-discrete II$_1$ factor. Fortunately, Connes's argument has been somewhat simplified over the years, and using in particular the work of Anantharaman-Delaroche on amenable correspondences from \cite{AD95} in lieu of \cite{EfLa77} and using the simplification of Connes's theorem given by Haagerup in \cite{Ha85}, most of the heavy lifting has already been completed so that we may simply put the existing pieces together. 

On the rigidity side, we improve Connes's rigidity argument from \cite{Co80} by showing that in property (T) factors not only is the set of inner-automorphisms open, but, in fact, if a property (T) factor $M$ is only approximately embedded in a factor $N$, then the set of approximate inner-endomorphisms from $M$ to $N$ is still open. In terms of ultrapowers, this means that if $M \subset N^{\mathcal U}$ is a property (T) subfactor and $\alpha \in \operatorname{Aut}(N)$ is in the connected component of the identity, then the ultrapower automorphism $\alpha^{\mathcal U}$ is inner when restricted to $M$. To control the intertwining partial isometries that arise from Connes's argument, it becomes necessary to impose the condition that the relative commutant $M' \cap N^{\mathcal U}$ is a factor, which is the reason for this assumption throughout Section~\ref{sec:embedultra}. Combining these deformation and rigidity arguments with Popa's technique on malleable deformations \cite{Po07a}, we obtain a stronger conclusion than Theorem~\ref{thm:mainfg}. In fact, we obtain that, under the assumption of factorial relative commutant, there is essentially no non-trivial embedding of a property (T) factor into an ultraproduct of $L\mathbb F_2$.

\begin{thm}\label{thm:main}
Suppose $n > 1$ and $M \subset (L\mathbb F_n)^{\mathcal U}$ is a property (T) subfactor such that $M' \cap (L \mathbb F_n)^{\mathcal U}$ is a factor, then there exist finite-dimensional subfactors $A_\iota \subset L\mathbb F_n$ so that 
\[
M \subset \prod_{\iota \to \mathcal U} A_\iota \subset (L\mathbb F_n)^{\mathcal U}.
\]
\end{thm}

We conjecture that Theorem~\ref{thm:main} should hold for all II$_1$ factors with the Haagerup property (see \cite[Problem 1.1]{AIM}). While we are not able to extend this theorem to that level of generality, the theorem provides strong evidence for this conjecture. We also show that Theorem~\ref{thm:main} holds more generally if $L\mathbb F_n$ is replaced by any finite tensor product of II$_1$ factors possessing a weakly coarse malleable deformation in the sense of Popa. For this we recall Popa's notion that $N$ has a weakly coarse malleable deformation if $N$ embed into a larger factor $N \subset \tilde N$ with the property that there exists an automorphism $\alpha \in \Aut(\tilde N)$ in the connected component of the identity such that we have weak containment of $N$-$N$ Hilbert bimodules
\[
L^2 \tilde N \ominus L^2 N \prec L^2 N \ovt L^2 N, \ \ \ \ \ \ L^2 \tilde N_\alpha \prec L^2 N \ovt L^2 N,
\]
where $L^2 \tilde N_\alpha$ denotes the $\alpha$-twisted $N$-$N$ bimodule structure on $L^2 \tilde N$. Examples of factors having a weakly coarse malleable deformation include all amalgamated free products of amenable II$_1$ factors, so that Theorem~\ref{thm:main} holds also in the cases when $L\mathbb F_n$ is replaced by $L( \mathbb Z \wr \mathbb F_n)$, which has a Cartan subalgebra, or for $L(\mathbb Z^2 \rtimes \SL_2(\mathbb Z))$, which has a diffuse subalgebra with relative property (T), or for $L\mathbb F_2 \ovt L\mathbb F_2$, which is not prime.

In particular, free group factors (or even any factor that is elementarily equivalent to a finite tensor product of factors having a weakly coarse malleable deformation) cannot existentially embed into a property (T) factor. This was previously shown for the case of free group factors in \cite[Theorem C]{Ke23} using strongly $1$-boundedness of property (T) factors from \cite{HaJeKE25}.

Using the framework of deformation/rigidity theory in the ultraproduct setting, we prove a strong Bass-Serre rigidity result for asymptotically commutative factors that contain a property (T) type II$_1$ subfactor with factorial relative commutant (see Theorem~\ref{thm:bassserreee}). This extends to the setting of ultraproducts the deformation/rigidity techniques used in \cite{IoPePo08} (cf.\ \cite{Oz06}, \cite{Pe09}). A consequence of this is that we are able to give explicit infinite families of full factors that are pairwise non-elementarily equivalent. We also give examples of factors having a non-elementarily equivalent index-2 subfactor (Theorem~\ref{thm:finotee}). Moreover, this construction is concrete enough that we can ensure that these factors are neither pseudomatricial nor elementarily equivalent to free group factors, and we can ensure that these factors have desirable structural properties, such as being group von Neumann algebras, group-measure space constructions, and/or being embeddable into $R^{\mathcal U}$. For example, the following result is obtained using this strong Bass-Serre rigidity result, answering \cite[Problems 18, 20]{KE25}.

\begin{thm}\label{thm:pairwisenonconj}
Let $M$ be a property (T) II$_1$ factor and let $N = M^{\overline \otimes \mathbb N}$, then the full factors $\{ *^k N \}_{k = 2}^\infty$ are pairwise non-elementarily equivalent where $*^k$ indicates the $k$-fold free product. None of these factors are pseudomatricial, nor are any of these elementarily equivalent to a free group factor. The same statement is true for the super McDuff factors $\{ (*^k N ) \ovt R \}_{k = 2}^\infty$, and also for the non-McDuff property Gamma factors $\{ (*_D^k (N \ovt R)) *_D  R \}_{k = 2}^\infty$, where $D \subset R$ is a Cartan subalgebra.
\end{thm}

The strong Bass-Serre rigidity result for ultraproducts can also be used to show the expected (but nettlesome) result that taking free products does not preserve elementary equivalence (see Theorem~\ref{thm:freeproductee}), answering \cite[Question 3.8]{GoPi25}. Using the strong Bass-Serre rigidity result for ultraproducts, together with \cite[Theorem 1.2]{BoChIo17}, we may also provide a continuum of pairwise non-elementarily equivalent full factors. This realizes an approach suggested by Remi Boutonnet from \cite[Problem 1.3]{AIM}. The same technique can also provide a continuum of pairwise non-elementarily equivalent II$_1$ factors that are super McDuff, or that have property Gamma but are not McDuff. The groups $T_\alpha(\Gamma)$ for $\alpha \in 2^{\mathbb N}$ are obtained using McDuff's functor (see Section~\ref{sec:continuum}).

\begin{thm}\label{thm:uncountablemcduff}
Let $\Gamma$ be a non-trivial i.c.c.\ property (T) group. Given $\alpha \in 2^{\mathbb N}$ we set $M_\alpha = L( T_1( T_\alpha (\Gamma) ))  \overline \otimes L\Gamma$.  The factors 
\[
\left\{ L(  M_\alpha * M_\alpha ), R \ovt ( M_\alpha * M_\alpha ) , R *_D ( R \ovt ( M_\alpha * M_\alpha ) ) \right\}_{\alpha \in 2^{\mathbb N}}
\]
are pairwise non-elementarily equivalent, where $D \subset R$ is a Cartan subalgebra. 
\end{thm}

In Section~\ref{sec:jung}, we consider another important problem in the model theory of II$_1$ factors, which is whether or not an enforceable II$_1$ factor exists. An enforceable factor, should one exist, would be a canonical object in the theory of II$_1$ factors, analogous to the Rado graph in the theory of graphs. In \cite{Go21}, Goldbring defined a property of II$_1$ factors to be enforceable if the property can be forced by a player in a certain two player game used to construct II$_1$ factors (see \cite[Definition 2.1]{Go21}). A II$_1$ factor $S$ is enforceable if the property of being isomorphic to $S$ is an enforceable property. It follows, trivially, that if an enforceable factor exists, then it must be unique up to isomorphism. 

Goldbring showed in \cite{Go21, Go23} that enforceable factors must be locally universal (and hence could not be hyperfinite by \cite{JiNaViWrYu20}), but they nonetheless share many properties in common with a hyperfinite II$_1$ factor $R$. For instance, if $S$ is an enforceable factor, then $S$ is existentially closed, embeds into any other existentially closed II$_1$ factor, and satisfies the Jung property meaning that any two embeddings of $S$ into its ultrapower $S^{\mathcal U}$ are unitarily conjugate. 

It is an open problem as to whether or not a non-hyperfinite factor can have the Jung property. Extending ideas of Jung from \cite{Ju07}, it was shown by Atkinson and Kunnawalkam Elayavalli that this property characterizes hyperfiniteness among those factors that embed into $R^{\mathcal U}$ \cite{AtKE21}. Building on these ideas, we show that, among locally universal II$_1$ factors, the Jung property would also characterize enforceability (see Corollary~\ref{cor:jungenforceable}). Such a strong homogeneity property perhaps suggests that an enforceable II$_1$ factor should not exist. The seemingly tamer property of super McDuffness, which asks simply that $N' \cap N^{\mathcal U}$ be a II$_1$ factor, seems, by contrast, to be quite a reasonable property to expect to be enforceable (it is easy to see that McDuffness itself is an enforceable property). We show that, in fact, enforceability of these two properties is the same. This gives evidence to the fact that an enforceable II$_1$ perhaps does exist.

\begin{thm}\label{thm:enforceable}
The following statements are equivalent:
\begin{enumerate}
\item An enforceable II$_1$ factor exists.
\item The property of being super McDuff is enforceable.
\item There is some super McDuff finitely generic factor. 
\end{enumerate}
Moreover, all super McDuff finitely generic factors are enforceable (and hence are isomorphic). 
\end{thm}

\subsection*{Acknowledgements} We thank Srivatsav Kunnawalkam Elayavalli for encouragement and numerous insightful conversations on this work and on the topic of the model theory of von Neumann algebras in general. We thank Isaac Goldbring for several useful discussions regarding the content of Section~\ref{sec:jung}. We thank Sorin Popa for remarks on some of the ideas used in this article, and we thank David Jekel for insightful comments and conversations. We also thank Adrian Ioana and Teng Zhang for pointing out errors in an earlier draft. The author is supported, in part, by NSF Grant DMS-2400040.

\section{Preliminaries}

For our purposes, an ultrafilter on a nonempty set $I$ is simply a point in the Stone-\v{C}ech compactification $\mathcal U \in \beta I$. It is a principal ultrafilter if $\mathcal U \in I \subset \beta I$. Given a function $a: I \to K$ with $K$ a compact Hausdorff space, we let $\lim_{\iota \to \mathcal U} a_\iota$ denote the value at $\mathcal U$ of the continuous extension of $a$ to $\beta I$.  We say that some property holds for $\mathcal U$-almost every $\iota$ if the set of $\iota \in I$ where it holds is contained in $\mathcal U$, i.e., if $A \subset I$ denotes the set of points where the property holds, then we have $\lim_{\iota \to \mathcal U} 1_A(\iota) = 1$. 

If $(M_\iota, \tau_\iota)_{\iota \in I}$ is a family of tracial von Neumann algebras, then we let $\ell^\infty(M_\iota)$ denote the $C^*$-algebra of bounded functions $(a_\iota)_\iota$ on $I$ such $a_\iota \in M_\iota$ for each $\iota \in I$. The set 
\[
\mathcal I = \{ (a_\iota)_\iota \in \ell^\infty(M_\iota) \mid \lim_{\iota \to \mathcal U} \| a_\iota \|_2 = 0 \}
\] 
is a closed self-adjoint ideal in $\ell^\infty(M_\iota)$ and the quotient 
\[
\prod_{\iota \to \mathcal U} (M_\iota, \tau_\iota) := \ell^\infty(M_\iota)/ \mathcal I
\] 
is a $W^*$-algebra with a normal faithful trace $\tau$ satisfying $\tau( ( a_\iota)_\iota + \mathcal I ) = \lim_{\iota \to \mathcal U} \tau_\iota( a_\iota)$. If each $M_\iota$ is a finite factor with $\lim_{\iota \to \mathcal U} \dim M_\iota = \infty$, then this gives a II$_1$ factor. 

In order to simplify notation, when the traces are all understood from the context (e.g., if each $M_\iota$ is a factor) we will use the same notation $\tau$ for all of the traces and we will denote the ultraproduct as $\prod_{\iota \to \mathcal U} M_\iota$. If $(a_\iota)_\iota \in \ell^\infty(M_\iota)$, then we denote by $(a_\iota)^{\mathcal U}$ its image in $\prod_{\iota \to \mathcal U} M_\iota$. Similarly, if we have a family of trace-preserving unital completely positive (u.c.p.)\ maps $\{ \phi_\iota \}_\iota$ where $\phi_\iota: N_\iota \to M_\iota$, we use the notation $(\phi_\iota)^{\mathcal U}: \prod_{\iota \to \mathcal U} N_\iota \to \prod_{\iota \to \mathcal U} M_\iota$ for the corresponding ultraproduct of maps, which satisfy $(\phi_\iota)^{\mathcal U} ( (a_\iota)^{\mathcal U} ) = ( \phi_\iota(a_\iota) )^{\mathcal U}$, for $(a_\iota)^{\mathcal U} \in \prod_{\iota \to \mathcal U} M_\iota$.

If $a \in \prod_{\iota \to \mathcal U} M_\iota$, then $a$ has a lift $a = (a_\iota)^{\mathcal U}$ where $\| a_\iota \| \leq \| a \|$ for each $\iota \in I$. We will freely use the fact that if $a$ is self-adjoint, or unitary, or a partial isometry, or a projection, then we may choose a lift so that each $a_\iota$ has the corresponding property.

For most of the results here, it is already sufficiently interesting to consider only the case $I = \mathbb N$. The proofs change very little, however, by considering an arbitrary set $I$, and the advantage of this more general setting is that, unconditional of the continuum hypothesis, we have a connection to the model theory of tracial von Neumann algebras by the Keisler-Shelah Theorem, from which it follows (see \cite{FaHaSh14}) that that two tracial von Neumann algebras $M$ and $N$ are elementarily equivalent if and only if there exist ultrafilters $\mathcal U$ and $\mathcal V$ on some sets so that $M^{\mathcal U}$ and $N^{\mathcal V}$ are isomorphic. For our purpose, we will take this theorem as a definition of elementary equivalence and we write $M \equiv N$ when this holds. 

An ultrafilter $\mathcal U$ on $I$ is countably incomplete if there exists a decreasing sequence of sets $I_n \subset I$ so that $\mathcal U(I_n) = 1$ and $\mathcal U( \cap_n I_n ) = 0$. If $I$ is countable, then every nonprincipal ultrafilter on $I$ is countably incomplete. In general, the existence of nonprincipal countably complete ultrafilters is a large cardinal axiom that is independent of ZFC. 

In working with elementary equivalence, there is not much lost in only considering those ultrafilters that are countably incomplete. Every complete theory of a tracial von Neumann algebra is the complete theory of a separable tracial von Neumann algebra, and if $M$ is a separable tracial von Neumann algebra and $\mathcal U$ is a countably complete ultrafilter, then the diagonal embedding $M \ni x \mapsto (x)^{\mathcal U} \in M^{\mathcal U}$ is an isomorphism \cite[Lemma 2.2]{BoChIo17}. More generally, the proof of \cite[Lemma 2.2]{BoChIo17} can be adapted to show that if $(N_\iota)_\iota$ is a family of separable tracial von Neumann algebras and if $\mathcal U$ is countably complete, then there is a separable tracial von Neumann algebra $N_0$ such that for $\mathcal U$-almost every $\iota$ we have an isomorphism $\theta_\iota: N_0 \to N_\iota$, and the corresponding embedding $N_0 \ni x \mapsto (\theta_\iota(x) )^{\mathcal U} \in  \prod_{\iota \to \mathcal U} N_\iota$ is an isomorphism. Because of this, we will often assume that our ultrafilter is countably incomplete, and we leave it to the reader to note in each case that, in light of the stated results above, this is justified. 

\begin{note}
Throughout this article, $\mathcal U$ (resp.\ $\mathcal V$) will be used to denote a nonprincipal ultrafilter on some set $I$ (resp.\ $K$). We will use $\iota$ (resp.\ $\kappa$) to denote the elements of $I$ (resp.\ $K$).
\end{note}

\section{Weak containment and ultraproducts of correspondences}\label{sec:weakcontain}

If $M$ and $N$ are von Neumann algebras, a Hilbert $M$-$N$ bimodule (or $M$-$N$ correspondence) consists of a Hilbert space $\mathcal H$, together with commuting $*$-representations $\pi_L: M \to \mathbb B(\mathcal H)$ and $\pi_R: N^{\rm op} \to \mathbb B(\mathcal H)$ (see, e.g., \cite{Po86, AD95}). Throughout this article we will only consider normal bimodules, meaning that the maps $\pi_L$ and $\pi_R$ will be assumed to be normal maps. Given $x \in M$, $y \in N$ and $\xi \in \mathcal H$, we write $x \xi y$ for the vector $\pi_L(x) \pi_R(y^{{\rm op}}) \xi$.

The $C^*$-algebra $M \otimes_{\rm bin} N^{\rm op}$ is the $C^*$-algebra generated by all normal Hilbert $M$-$N$ bimodules. We may then think of a normal Hilbert bimodule $\mathcal H$ as a representation of $M \otimes_{\rm bin} N^{\rm op}$ with the property that it is normal when restricted to $M$ or $N^{\rm op}$. In particular, if $\mathcal H$ and $\mathcal K$ are two Hilbert $M$-$N$ bimodules with corresponding representations $\pi$ and $\rho$, respectively, we say that the Hilbert bimodule $\mathcal H$ is weakly contained $\mathcal K$ and write $\mathcal K \prec \mathcal K$ if we have weak containment of the corresponding representations $\pi \prec \rho$, i.e., if $\| \pi(z) \| \leq \| \rho(z) \|$ for each $z \in M \otimes_{\rm bin} N^{\rm op}$. The trivial bimodule is $L^2 N$, the coarse bimodule is $L^2 N \ovt L^2N$, and a von Neumann algebra is amenable (or semi-discrete) if we have weak containment $L^2N \prec L^2 N \ovt L^2 N$. 

If $M$ and $N$ are tracial von Neumann algebras and $\mathcal H$ is a Hilbert $M$-$N$ bimodule, then a vector $\xi \in \mathcal H$ is left-bounded if the operator $L_\xi: L^2 N \to \mathcal H$ given by $L_\xi(y) = \xi y$ is bounded. Right-bounded vectors are those such that the operator $R_\xi: L^2 M \to \mathcal H$ given by $R_\xi(x) = x \xi$ is bounded. A vector is left-right-bounded if it is both left-bounded and right-bounded. As an example, for the trivial bimodule $L^2M$ the notions of left and right-boundedness coincide and describe the image of $M$ in $L^2 M$, and both norms agree with the uniform norm on $M$.

Ultraproducts of Hilbert bimodules for tracial von Neumann algebras $M$ and $N$ were studied in \cite{GoHaSi19} and we now recall this construction. If $(\mathcal H_\iota)_\iota$ is a family of normal Hilbert $M_\iota$-$N_\iota$ bimodules for tracial von Neumann algebras $M_\iota$ and $N_\iota$, we let $\prod_{\iota \to \mathcal U}^{\rm Hilb} \mathcal H_\iota$ denote the ultraproduct of Hilbert spaces $(\mathcal H_\iota)_\iota$, which is again a Hilbert space with inner-product $\langle (\xi_\iota)^{\mathcal U}, (\eta_\iota)^{\mathcal U} \rangle = \lim_{\iota \to \mathcal U} \langle \xi_\iota, \eta_\iota \rangle$. A vector $\xi  \in \prod_{\iota \to \mathcal U}^{\rm Hilb} \mathcal H_\iota$ is left-right-bounded if it has a representation $\xi = (\xi_\iota)^{\mathcal U}$ such that $\lim_{\iota \to \mathcal U} \max \{ \| L_{\xi_\iota} \|,  \| R_{\xi_\iota} \| \} < \infty$. The correspondence ultraproduct is the closure in $\prod_{\iota \to \mathcal U}^{\rm Hilb} \mathcal H_\iota$ of the space of left-right-bounded vectors. We denote the correspondence ultraproduct as $\prod_{\iota \to \mathcal U} \mathcal H_\iota$. 

The correspondence ultraproduct $\prod_{\iota \to \mathcal U} \mathcal H_\iota$ is naturally a Hilbert $\left( \prod_{\iota \to \mathcal U} M_\iota \right)$-$\left( \prod_{\iota \to \mathcal U} N_\iota \right)$ bimodule, with bimodule structure satisfying $( x_\iota)^{\mathcal U} \cdot (\xi_\iota)^{\mathcal U} \cdot (y_\iota)^{\mathcal U} = ( x_\iota \xi_\iota y_\iota)^{\mathcal U}$ for $(x_\iota)^{\mathcal U} \in \prod_{\iota \to \mathcal U} M_\iota$, $(y_\iota)^{\mathcal U} \in \prod_{\iota \to \mathcal U} N_\iota$, and $(\xi_\iota)^{\mathcal U} \in \prod_{\iota \to \mathcal U} \mathcal H_\iota$ such that $\lim_{\iota \to \mathcal U} \max \{ \| L_{\xi_\iota} \|, \| R_{\xi_\iota} \| \}$. Note that, in this case, we have 
\[
\| ( x_\iota )^{\mathcal U} \cdot (\xi_\iota)^{\mathcal U} \cdot (y_\iota)^{\mathcal U}  \| 
= \lim_{\iota \to \mathcal U} \| x_\iota \xi_\iota y_\iota \| 
\leq  \lim_{\iota \to \mathcal U} \| x_\iota \|_2 \| R_{\xi_\iota} \| \| y_\iota \|
= \| (x_\iota)^{\mathcal U} \|_2  \| (y_\iota)^{\mathcal U} \| \lim_{\iota \to \mathcal U} \| R_{\xi_\iota} \|
\]
and similarly we also have
\[
\| ( x_\iota )^{\mathcal U} \cdot (\xi_\iota)^{\mathcal U} \cdot (y_\iota)^{\mathcal U}  \| 
\leq \| (x_\iota)^{\mathcal U} \|  \| (y_\iota)^{\mathcal U} \|_2 \lim_{\iota \to \mathcal U} \| L_{\xi_\iota} \|.
\]
It therefore follows that the bimodule structure given above is well-defined and also shows that $\prod_{\iota \to \mathcal U} \mathcal H_\iota$ is a normal $\left( \prod_{\iota \to \mathcal U} M_\iota \right)$-$\left( \prod_{\iota \to \mathcal U} N_\iota \right)$ bimodule.  For example, if $M_\iota$ are tracial von Neumann algebras, then it follows that the correspondence ultraproduct $\prod_{\iota \to \mathcal U} L^2 M_\iota \subset \prod_{\iota \to \mathcal U}^{\rm Hilb} L^2 M_\iota$ can be identified with $L^2 \left( \prod_{\iota \to \mathcal U} M_\iota \right)$ as Hilbert $\prod_{\iota \to \mathcal U} M_\iota$ bimodules. 

Conversely, if $\xi = (\xi_\iota)^{\mathcal U} \in \prod_{\iota \to \mathcal U} \mathcal H_\iota$ is $\prod_{\iota \to \mathcal U} N_\iota$-left bounded and if $\kappa > \| L_\xi \|$, then we may let $\zeta_\iota \in L^1N$ denote the vector functional given by $\langle \xi_\iota \cdot, \xi_\iota \rangle$ and we may set $p_\iota = 1_{[\kappa, \infty)}( \zeta_\iota )$. Since $\| L_{\xi} \| > \kappa$ it then follows that $\lim_{\iota \to \mathcal U} \tau(p_\iota) = 0$ and hence replacing $(\xi_\iota)_\iota$ with $(\xi_\iota p_\iota)_\iota$ we have that $\xi$ is represented by a family of uniformly left-bounded vectors, i.e., $\sup_\iota \| L_{\xi_\iota} \| < \infty$. Similarly if $\xi$ is a left-right-bounded vector, then we may choose a representation $\xi = (\xi_\iota)^{\mathcal U}$ so that that family $(\xi_\iota)_\iota$ is uniformly left-right bounded. 

Thus, left-bounded vectors in $\prod_{\iota \to \mathcal U} \mathcal H_\iota$ are precisely those vectors that are represented by a family that is uniformly bounded in the left-norm. We also note that if $(\xi_\iota)_\iota$ and $(\eta_\iota)_\iota$ are such uniformly left-norm bounded families, then we have 
\begin{equation}\label{eq:leftboundedip}
\langle (\xi_\iota)^{\mathcal U}, (\eta_\iota)^{\mathcal U} \rangle_{\prod_{\iota \to \mathcal U} N_\iota} = ( \langle \xi_\iota, \eta_\iota \rangle_{N_\iota} )^{\mathcal U}.
\end{equation}

We leave the following lemma to the reader.

\begin{lem}\label{lem:ultradirectsum}
Let $\mathcal H_\iota$ and $\mathcal K_\iota$ be normal $M_\iota$-$N_\iota$ Hilbert bimodules.  We have a canonical isomorphism $\prod_{\iota \to \mathcal U} ( \mathcal H_\iota \oplus \mathcal K_{\iota} ) \cong \left( \prod_{\iota \to \mathcal U} \mathcal H_\iota \right) \oplus \left( \prod_{\iota \to \mathcal U} \mathcal K_\iota \right)$.
\end{lem}

\begin{lem}\label{lem:ultrafusion}
Let $(M_\iota)_\iota$, $(N_\iota)_\iota$, and $(P_\iota)_\iota$ be families of tracial von Neumann algebras and suppose $(\mathcal H_\iota)_\iota$ and $(\mathcal K_\iota)_\iota$ are families of $M_\iota$-$N_\iota$ and $N_\iota$-$P_\iota$ normal Hilbert bimodules, respectively. We then have a canonical inclusion of $\prod_{\iota \to \mathcal U} M_\iota$-$\prod_{\iota \to \mathcal U} P_\iota$ Hilbert bimodules
\[
\left( \prod_{\iota \to \mathcal U} \mathcal H_\iota \right) \oovt{\left( \prod_{\iota \to \mathcal U} N_\iota \right)} \left( \prod_{\iota \to \mathcal U} \mathcal K_\iota \right)
\subset \prod_{\iota \to \mathcal U} \mathcal H_\iota \oovt{N_\iota} \mathcal K_\iota.
\]
\end{lem}
\begin{proof}
If $\xi \in \prod_{\iota \to \mathcal U} \mathcal H_\iota$ and $\eta \in \prod_{\iota \to \mathcal U} \mathcal K_\iota$ are left-right-bounded, then we may choose representations $\xi = (\xi_\iota)^{\mathcal U}$ and $\eta = (\eta_\iota)^{\mathcal U}$ where each of these families is uniformly left-right-bounded. Thus, the family $(\xi_\iota \otimes_{N_\iota} \eta_\iota )^{\mathcal U}$ is also uniformly left-right-bounded and so defines a vector in $\prod_{\iota \to \mathcal U} \mathcal H_\iota \oovt{N_\iota} \mathcal K_\iota$. 

Using (\ref{eq:leftboundedip}) we see that the association $\xi \otimes_{\prod_{\iota \to \mathcal U} N_\iota} \eta \mapsto (\xi_\iota \otimes_{N_\iota} \eta_\iota )^{\mathcal U} \in \prod_{\iota \to \mathcal U} \mathcal H_\iota \oovt{N_\iota} \mathcal K_\iota$ gives a well-defined isometric $\prod_{\iota \to \mathcal U} M_\iota$-$\prod_{\iota \to \mathcal U} N_\iota$ bimodular map on the dense span of vectors of this type. Hence, extends to a canonical embedding as stated above.
\end{proof}

\begin{prop}\label{prop:weakcontainmentultra}
Let $\mathcal H_\iota$, $\mathcal K_\iota$ be normal $M_\iota$-$N_\iota$ Hilbert bimodules. If for $\mathcal U$-almost every $\iota$ we have $\mathcal H_\iota \prec \mathcal K_\iota$, then as $(\prod_{\iota \to \mathcal U} M_\iota)$-$(\prod_{\iota \to \mathcal U} N_\iota)$ bimodules we have weak containment
\[
 \prod_{\iota \to \mathcal U} \mathcal H_\iota \prec  \prod_{\iota \to \mathcal U} \mathcal K_\iota.
\]
\end{prop}
\begin{proof}
We let $\pi_\iota: M_\iota \otimes_{\rm bin} N_\iota^{\rm op} \to \mathbb B(\mathcal H_\iota)$ and $\rho_\iota: M_\iota \otimes_{\rm bin} N_\iota^{\rm op} \to \mathbb B(\mathcal K_\iota)$ denote the associated representations. We also denote the associated representations of the ultraproducts by 
\[
\pi^{\mathcal U}: \left(\prod_{\iota \to \mathcal U} M_\iota \right) \otimes_{\rm bin} \left(\prod_{\iota \to \mathcal U} N_\iota \right)^{\rm op} \to \mathbb B\left(\prod_{\iota \to \mathcal U} \mathcal H_\iota \right)
\]
and
\[
\rho^{\mathcal U}: \left(\prod_{\iota \to \mathcal U} M_\iota \right) \otimes_{\rm bin} \left(\prod_{\iota \to \mathcal U} N_\iota \right)^{\rm op} \to \mathbb B\left(\prod_{\iota \to \mathcal U} \mathcal K_\iota \right),
\]
respectively.

If we let $A \subset \ell^\infty(M_\iota \otimes_{\rm bin} N_\iota^{\rm op})$ denote the $C^*$-algebra generated by $\ell^\infty(M_\iota)$ and $\ell^\infty(N_\iota^{\rm op})$, then $A$ naturally acts on $\prod_{\iota \to \mathcal U} \mathcal K_\iota$, and if we denote by $q: A \to \mathbb B( \prod_{\iota \to \mathcal U} \mathcal K_\iota)$ the corresponding representation, then we see that the image of $q$ is $\rho^{\mathcal U}\left( \left(\prod_{\iota \to \mathcal U} M_\iota \right) \otimes_{\rm bin} \left(\prod_{\iota \to \mathcal U} N_\iota \right)^{\rm op}  \right)$. Thus, if $z \in (\prod_{\iota \to \mathcal U} M_\iota) \otimes_{\rm bin} (\prod_{\iota \to \mathcal U} N_\iota)^{\rm op}$, then there exist $z_\iota \in M_\iota \otimes_{\rm bin} N_\iota^{\rm op}$ with $\| z_\iota \| \leq \| \rho^{\mathcal U}(z) \|$ so that $(z_\iota)_\iota \in A$ with $q( (z_\iota)_\iota) = \rho^{\mathcal U}(z)$.

Fix $\xi \in \prod_{\iota \to \mathcal U} \mathcal H_\iota$ with $\| \xi \| \leq 1$ and take $\xi_\iota \in \mathcal H_\iota$ with $\| \xi_\iota \| \leq 1$ so that $\xi = (\xi_\iota)^{\mathcal U}$. We then have 
\[
\| \pi^{\mathcal U}(z) \xi \| = \lim_{\iota \to \mathcal U} \| \pi_\iota(z_\iota) \xi_\iota \| \leq \lim_{\iota \to \mathcal U} \| \rho_\iota(z_\iota) \| \leq \| \rho^{\mathcal U}(z) \|.
\] 
Hence $\pi^{\mathcal U} \prec \rho^{\mathcal U}$. 
\end{proof}

\begin{thm}[Haagerup \cite{Ha85}]\label{thm:Haagerup}
For each $n \in \mathbb N$ and $\varepsilon > 0$ there exists $\delta = \delta(n, \varepsilon) > 0$ so that if $M$ is a II$_1$ factor, $u_1, u_2, \ldots, u_n \in \mathcal U(M)$, and there exists u.c.p.\ maps $\phi: M \to A$ and $\psi: A \to M$ for $A$ some finite-dimensional von Neumann algebra such that $\| \psi \circ \phi(u_k) - u_k \|_2 < \delta$ for each $1 \leq k \leq n$, then there exists a finite-dimensional von Neumann subalgebra $B \subset M$ such that $\| u_k - \mathbb E_B(u_k) \|_2 < \varepsilon$. 
\end{thm}
\begin{proof}
This follows from \cite{Ha85} by noticing that the proofs in \cite{Ha85} give quantitative estimates. Indeed, by the proofs Lemma 3.4 and Proposition 3.5 in \cite{Ha85} if $\varepsilon > 0$ is fixed, and $n \in \mathbb N$, then if we take $\delta, \eta > 0$ such that $(2\eta + 4 \eta^{1/2} )^{1/2} + \eta + 2 \eta^{1/2} \leq \delta$ and if $u_1, \ldots, u_n \in \mathcal U(M)$ are such that there exist $\phi: M \to \mathbb M_p(\mathbb C)$, and $\psi: \mathbb M_p(\mathbb C) \to M$ are u.c.p.\ with $\| \psi \circ \phi(u_k) - u_k \|_2 < \delta$ for $1 \leq k \leq n$, then there exists $q \in \mathbb N$ and u.c.p.\ maps $\phi': M \to \mathbb M_q(\mathbb C)$ and $\psi': \mathbb M_q(\mathbb C) \to M$ with $\tau \circ \psi' = \operatorname{tr}_q$ and $\operatorname{tr}_q \circ \phi' = \tau$ such that $\| \psi' \circ \phi'(u_k) - u_k \|_2 < \delta$ for $1 \leq k \leq n$. The proof of Proposition 3.5 in \cite{Ha85} then shows that there exist unitaries $v_1, \ldots, v_n$ in a finite-dimensional von Neumann subalgebra $B \subset M$ so that $(u_1, \ldots, u_n)$ and $(v_1, \ldots, v_n)$ are $\delta$-related in the terminology of \cite{Ha85}. However, by Theorem 4.2 in \cite{Ha85}, there exists a universal $\delta = \delta(n, \varepsilon)$, which is independent of the factor $M$, so that if these $n$-tuples are $\delta$-related, then there is a unitary $w \in \mathcal U(M)$ so that $\| w u_k - v_k w \|_2 < \varepsilon$ for $1 \leq k \leq n$. Hence, $\| u_k - \mathbb E_{w^* B w}(u_k) \|_2 < \varepsilon$ for $1 \leq k \leq n$. 
\end{proof}

\begin{thm}\label{thm:amenimplieshyperfinite}
Suppose $M_\iota$ are II$_1$ factors and $M \subset \prod_{\iota \to \mathcal U} M_\iota$ is a von Neumann subalgebra. The following conditions are equivalent. 
\begin{enumerate}
\item\label{item:main1} For any separable von Neumann subalgebra $M_0 \subset M$, there exist separable hyperfinite von Neumann subalgebras $R_\iota \subset M_\iota$ so that 
\[
M_0 \subset \prod_{\iota \to \mathcal U} R_\iota \subset \prod_{\iota \to \mathcal U} M_\iota.
\] 
\item\label{item:main2} We have weak containment of $M$-$(\prod_{\iota \to \mathcal U} M_\iota)$ bimodules 
\[
L^2 \left( \prod_{\iota \to \mathcal U} M_\iota \right) \prec \prod_{\iota \to \mathcal U} (L^2 M_\iota \ovt L^2 M_\iota).
\] 
\item\label{item:main3} We have weak containment of $M$-$M$ bimodules 
\[
L^2 M \prec \prod_{\iota \to \mathcal U} (L^2 M_\iota \ovt L^2 M_\iota).
\] 
\item\label{item:main4} For every $F \subset M$ finite and $\varepsilon > 0$, there exist u.c.p.\ maps $\phi_\iota: M_\iota \to \mathbb M_{k(\iota)}(\mathbb C)$ and $\psi_\iota: \mathbb M_{k(\iota)}(\mathbb C) \to M_\iota$ so that $\| ( \psi_\iota \circ \phi_\iota)^{\mathcal U} (x) - x \|_2 < \varepsilon$ for all $x \in F$. 
\end{enumerate}
\end{thm}
\begin{proof}
If $R_\iota \subset M_\iota$ are hyperfinite, then as $R_\iota$-$M_\iota$ bimodules we have $L^2 M_\iota \prec L^2 M_\iota \ovt L^2 M_\iota$ and hence by Proposition~\ref{prop:weakcontainmentultra}, we have weak containment of normal Hilbert $(\prod_{\iota \to \mathcal U} R_\iota)$-$(\prod_{\iota \to \mathcal U} M_\iota)$ bimodules $L^2 (\prod_{\iota \to \mathcal U} M_\iota ) = \prod_{\iota \to \mathcal U} (L^2M_\iota) \prec \prod_{\iota \to \mathcal U} (L^2 M_\iota \ovt L^2 M_\iota)$. Thus if (\ref{item:main1}) holds, then restricting to a separable von Neumann subalgebra $M_0 \subset M$ we see that (\ref{item:main2}) holds for $M_0$. But weak containment is a local property, and since $M_0 \subset M$ was an arbitrary separable von Neumann subalgebra we have that (\ref{item:main2}) holds also for $M$. The implication (\ref{item:main2}) $\implies$ (\ref{item:main3}) is trivial. 

Our proof of (\ref{item:main3}) $\implies$ (\ref{item:main4}) follows the proof of Theorem 2.2 from \cite{AD95}. Suppose we have weak containment of $M$-$M$ bimodules $L^2 M \prec \prod_{\iota \to \mathcal U} (L^2 M_\iota \ovt L^2 M_\iota)$. Then there exists a net $\alpha_k: M \to M$ of completely positive maps, each of which is a finite sum of coefficients of $\prod_{\iota \to \mathcal U} (L^2 M_\iota \ovt L^2 M_\iota)$ such that we have ultraweak operator topology convergence $\alpha_k(x) \to x$ for $x \in M$. Moreover, we may assume $\alpha_k(1) \leq 1$ by Lemma 1.6 in \cite{ADHa90}. 
Thus, if $F \subset \mathcal U(M)$ is finite and $\varepsilon > 0$ we may find $K_0 \in \mathbb N$ and $\zeta_k = (\zeta_{k, \iota})_\iota \in \prod_{\iota \to \mathcal U} (L^2 M_\iota \ovt L^2 M_\iota)$ for $1 \leq k \leq K_0$ so that $\| \sum_{k = 1}^{K_0} \theta_k(x) - x \|_2 < \varepsilon$ for all $x \in F$, where $\theta_k: M \to M$ is the c.p.\ map associated to $\zeta_k$ satisfying  
\[
\langle x \zeta_k y, \zeta_k \rangle = \tau( \theta_k(x) y), \ \ \ \ \ x, y \in M.
\]

By approximation, we may assume that $\zeta_{k, \iota} = \sum_{j = 1}^{J_{k, \iota}} \xi_{k, j}^\iota \otimes b_{k, j}^\iota \in (L^2 M_\iota) \odot M_\iota$ with $\{ \xi_{k, j}^\iota \}_{j = 1}^{J_{k, \iota}}$ orthonormal. Since we may assume $\sum_{k = 1}^{K_0} \theta_k(1) \leq 1$, we may multiply on the right by some projection in $M_\iota$ and likewise assume that $\sum_{k = 1}^{K_0} \| \zeta_{k, \iota} y \|^2 \leq \| y \|_2^2$ for $y \in M_\iota$. We let $\phi_{k, \iota}: M_\iota \to \mathbb M_{J_{k, \iota}}(\mathbb C)$ and $\psi_{k, \iota}: \mathbb M_{J_{k, \iota}}(\mathbb C) \to M_\iota$ be the c.p.\ maps defined by $\phi_{k, \iota}(x) = ( \langle x \xi_{k, j}^\iota, \xi_{k, \ell}^\iota \rangle )_{j, \ell}$ and $\psi_{k, \iota}( (\alpha_{j \ell})_{j  \ell}) = \sum_{j, \ell = 1}^{J_{k, \iota}} \alpha_{j \ell} b_{k, j}^\iota (b_{k, \ell}^\iota)^*$. Setting $\phi_\iota = \oplus_{k = 1}^{K_0} \phi_{k,\iota}$ and $\psi_\iota  = \sum_{k = 1}^{K_0} \psi_{k, \iota}\circ \Ad (P_{k, \iota})$ where $P_{k, \iota}: \mathbb C^{\sum_{j = 1}^{K_0} J_{j, \iota}} \to \mathbb C^{J_{k, \iota}}$ is the corresponding orthogonal projection, we then have that $\psi_\iota \circ \phi_\iota$ is the completely positive map corresponding to the vector $\oplus_{k = 1}^{K_0} \zeta_{k, \iota}$, and so we have $\mathbb E_{M} \circ (\psi_\iota \circ \phi_\iota)^{\mathcal U}(x) = \sum_{k = 1}^{K_0} \theta_k(x)$ for $x \in F$.

This then shows that (\ref{item:main4}) holds except that the maps $\phi_\iota$ and $\psi_\iota$ may not be unital. 
However, since $\{ \xi_{k, j}^\iota \}_{j = 1}^{J_{k, \iota}}$ are orthonormal we see that $\phi_\iota$ is unital and we then have $\psi_\iota(1) = \psi_\iota \circ \phi_\iota(1) \leq 1$. If we fix states $\eta_\iota$ in the matrix algebras, then replacing $\psi_\iota$ with $\psi_\iota'(T) = \eta_\iota(T) (1 - \psi_\iota(1)) +  \psi_\iota(T)$ we see that the maps may be taken to be unital.

The implication (\ref{item:main4}) $\implies$ (\ref{item:main1}) follows directly from Haagerup's result above (Theorem~\ref{thm:Haagerup}). Indeed, if $M_0 \subset M$ is separable, let $\{ u_n \}_{n = 1}^\infty \subset \mathcal U(M_0)$ enumerate a countable set of unitaries that generate $M_0$ as a von Neumann algebra. For each unitary $u_n$, we choose a representation $u_n = (u_{n, \iota})^{\mathcal U}$ where $u_{n, \iota}$ is unitary. Set
\[
I_N = \{ \iota \mid \| \mathbb E_{A_{N, \iota}}(u_{n, \iota}) - u_{n, \iota} \|_2 < 1/N, 1 \leq n \leq N, {\rm \ for \ some \ } A_{N, \iota} \subset M_\iota {\rm \ finite}-{\rm dimensional} \}.
\]
From (\ref{item:main4}) and Theorem~\ref{thm:Haagerup} we have $\mathcal U(I_N) = 1$ for each $N$. If $\mathcal U( \cap_N I_N) = 0$, then we may set $A_\iota = A_{N, \iota}$ where $N$ is the maximal number such that $\iota \in I_N$ (set $A_\iota = \mathbb C$ if $\iota \not\in I_1$ or if $\iota \in \cap_{N \in \mathbb N} I_N$, both of which are $\mathcal U$-null). We then have $M_0 \subset \prod_{\iota \to \mathcal U} A_\iota$. Otherwise, if $\mathcal U( \cap_N I_N) = 1$, then set $R_\iota = \{ u_{n, \iota} \}_n''$, which is approximately finite-dimensional and hence hyperfinite for $\mathcal U$-almost every $\iota$.
\end{proof}

\begin{rem}
If $\mathcal U$ is countably incomplete, then (\ref{item:main1}) is equivalent to the statement
\begin{itemize}
\item For any separable von Neumann subalgebra $M_0 \subset M$, there exist finite-dimensional subfactors $A_\iota \subset M_\iota$ so that 
\[
M_0 \subset \prod_{\iota \to \mathcal U} A_{\iota} \subset \prod_{\iota \to \mathcal U} M_\iota. 
\]
\end{itemize}
\end{rem}

We'll say that a subalgebra $M \subset  \prod_{\iota \to \mathcal U} M_\iota$ is matricially approximable with respect to $(M_\iota)_{\iota \in I}$ (or simply with respect to $N$ if $M_\iota = N$ for each $\iota \in I$) if the equivalent conditions in Theorem~\ref{thm:amenimplieshyperfinite} are satisfied. We'll say that $M$ is matricially inapproximable with respect to $(M_\iota)_{\iota \in I}$ otherwise.

Sakai initiated the study of the flip automorphism on $M \ovt M$ in \cite{Sa75}. Connes extended this study by characterizing the hyperfinite II$_1$ factor $R$ as the unique separable II$_1$ factor such that the flip automorphism on $R \ovt R$ is approximately inner \cite{Co76}. Theorem~\ref{thm:amenimplieshyperfinite} allows us to obtain the following ultraproduct analog of Connes's characterization.

\begin{cor}\label{cor:approxconj}
Suppose $M_\iota^j$ are II$_1$ factors for $j \in \{ 1, 2 \}$ and $\iota \in I$. Suppose $M$ is a separable tracial von Neumann algebra and for $j \in \{ 1, 2 \}$ we have an embedding $\theta_j: M \to \prod_{\iota \to \mathcal U} M_\iota^j$. If for every $F \subset M$ finite and $\varepsilon > 0$ there exists $u \in \mathcal U( \prod_{\iota \to \mathcal U} ( M_\iota^1 \ovt M_\iota^2) )$ such that 
\[
\| u (\theta_1(b) \otimes 1) u^* - 1 \otimes \theta_2(b) \|_2 < \varepsilon
\] 
for all $b \in F$, then for each $j \in \{ 1, 2 \}$, $\theta_j(M)$ is matricially approximable with respect to $(M_\iota^j)_\iota$. 
\end{cor}
\begin{proof}
We take $F$, $\varepsilon> 0$ and $u$ as above. We take $u_\iota \in \mathcal U(M_\iota^1 \ovt M_\iota^2)$ so that $u = (u_\iota)^{\mathcal U}$. We set $v_\iota = (u_\iota^* \otimes 1) (1 \otimes \sigma(u_\iota) ) \in \mathcal U( M_\iota^1 \ovt M_\iota^2 \ovt M_\iota^1 )$ where $\sigma: M_\iota^1 \ovt M_\iota^2 \to M_\iota^2 \ovt M_\iota^1$ is the flip. Setting $v = (v_\iota)^{\mathcal U}$, for $x \in F$ we have $\| ( \theta_1(x) \otimes 1 \otimes 1) v - v (1 \otimes 1 \otimes \theta_1(x) ) \|_2 \leq 2 \varepsilon$. Theorem~\ref{thm:amenimplieshyperfinite} then shows that $\theta_1(M)$ is matricially approximable with respect to $(N_\iota^1)_\iota$. The case for $\theta_2(M)$ is similar. 
\end{proof}

\subsection{An obiter dictum on the Jung property and enforceability}\label{sec:jung}

From the previous corollary we may apply a result from \cite{AtKE21} to give a criterion for hyperfiniteness in a manner similar to \cite{Ju07} or \cite{AtKE21} but without the initial assumption that $M$ embeds into a matricial ultraproduct.

\begin{cor}
Suppose $M_\iota^j$ are II$_1$ factors for $j \in \{ 0, 1 \}$ and $\iota \in I$. Suppose $M$ is a separable tracial von Neumann algebra and we have trace-preserving embeddings $\theta_j: M \to \prod_{\iota \to \mathcal U} M_\iota^j$. Then, $M$ is hyperfinite if and only if any two embeddings of $M$ into $\prod_{\iota \to \mathcal U} (M_\iota^1 \ovt M_\iota^2)$ are unitarily conjugate. 
\end{cor}
\begin{proof}
Since the embeddings $\theta_1$ and $\theta_2$ are unitarily conjugate, it follows from Corollary~\ref{cor:approxconj} that $M$ embeds into $R^{\mathcal U}$. With this extra condition, the result is then given by Corollary 4.8 in \cite{AtKE21}.
\end{proof}

Following \cite{AtGoKE22}, we'll say that a separable II$_1$ factor $M$ has the Jung property if for some countably incomplete ultrafilter $\mathcal U$ on a set $I$ we have that any two embeddings of $M$ into $M^{\mathcal U}$ are conjugate by a unitary in $M^{\mathcal U}$. (Standard arguments show that if this holds for some countably incomplete ultrafilter, then it holds for all countably incomplete ultrafilters). It is an open problem as to whether $R$ is the unique separable II$_1$ factor with this property \cite[Question 3.3.12]{AtGoKE22}, \cite[Question 6.13]{GoHa23}. The result \cite[Corollary 4.8]{AtKE21} cited above can be rephrased to state that if $M$ embeds into $R^{\mathcal U}$ and also has the Jung property, then $M \cong R$. It also follows easily from Connes's characterizations of the hyperfinite II$_1$ factor that if $M$ has the Jung property and $M$ is elementarily equivalent to $M \ovt M$, then $M \cong R$ (see also the remark after Theorem 2 in \cite{Go23}). Using the previous corollary instead of Connes's theorem, we get the following strengthening of this remark.

\begin{cor}
Suppose $M$ has the Jung property and $M \equiv N_1 \ovt N_2$ for some II$_1$ factors $N_1$ and $N_2$ such that $M$ embeds into $N_j^{\mathcal U}$ for $j \in \{ 1, 2 \}$, then $M \cong R$. 
\end{cor}

A trace-preserving embedding of tracial von Neumann algebras $\theta: N \to P$ is existential if there exists an ultrafilter $\mathcal U$ and an embedding $\alpha: P \to N^{\mathcal U}$ so that $\alpha_{\theta(N)}$ gives the diagonal embedding of $N \subset N^{\mathcal U}$. A factor $N$ is existentially closed if every embedding of $N$ into any other tracial von Neumann algebra is existential. 

The special case of the previous corollary when $N_1$ and $N_2$ are existentially closed can be shown to follow more directly from Connes's characterization of amenability by utilizing the following strengthening of Atkinson and Kunnawalkam Elayavalli's result from \cite{AtKE21}, which is of independent interest.

\begin{thm}\label{thm:conjinclusion}
Suppose $\mathcal U$ is countably incomplete, $N \subset \tilde N$ is an inclusion of tracial von Neumann algebras, and $M \subset N^{\mathcal U}$ is a separable von Neumann algebra such that for any embedding $\rho: M \to N^{\mathcal U} \subset \tilde N^{\mathcal U}$ there exists a unitary $u \in \tilde N^{\mathcal U}$ with $\rho(x) = uxu^*$ for all $x \in M$. Then $M$ embeds into $\tilde N$. Moreover, if the inclusion $N \subset \tilde N$ is existential, then $M$ embeds into $N$.  
\end{thm}
\begin{proof}
We let $C$ denote the universal $C^*$-algebra generated by a sequence of contractions $\{ x_n \}_{n = 1}^\infty$. We fix a $*$-homomorphism $\theta: C \to M \subset N^{\mathcal U}$ with weakly dense image and we choose $*$-homomorphisms $\theta_\iota: C \to N$ so that $\theta(x_n) = (\theta_\iota(x_n))^{\mathcal U}$ for each $n \in \mathbb N$. We also let $d$ be a metric on the state space of $C$ that induces the weak$^*$-topology and we choose a decreasing sequence of sets $I_n \subset I$ so that $\mathcal U(I_n) = 1$ and $\cap_n I_n = \emptyset$.

We claim that for each $n \in \mathbb N$, there exist $m \geq n$ and $\iota_{n} \in I_m \setminus I_{m-1}$ having the property that whenever $\rho: C \to N$ is $*$-homomorphism satisfying $d( \tau_N \circ \rho, \tau_M \circ \theta) < 1/m$, then there exists a unitary $u \in \mathcal U(\tilde N)$ so that $\| u \rho(x_k) u^* - \theta_{\iota_{n}}(x_k) \|_2 \leq 1/n$ for all $k \leq n$. Otherwise, we could find some $n_0 \in \mathbb N$ and choose for each $m \geq n_0$ and $\iota \in I_m \setminus I_{m-1}$ a $*$-homomorphism $\rho_\iota: C \to N$ satisfying $d( \tau_N \circ \rho_\iota, \tau_M \circ \theta) < 1/m$ and having the property that 
\[
\inf_{u \in \mathcal U(\tilde N)} \max_{k \leq n_0} \{ \| u \rho_\iota(x_k) u^* - \theta_\iota(x_k) \|_2 \} \geq 1/n_0.
\]
But then, setting $\rho(x) = (\rho_\iota(x))^{\mathcal U}$ we see that $\rho$ gives a well-defined  $*$-homomorphism from $C$ into  $N^{\mathcal U}$ that satisfies $\tau_{N^{\mathcal U}} \circ \rho = \tau_M \circ \theta$ so that this defines an embedding of $M$ into $N^{\mathcal U}$. Moreover, we would have 
\[
\inf_{u \in \mathcal U(\tilde N^{\mathcal U})} \max_{k \leq n_0} \{ \| u \rho(x_k) u^* - x_k \|_2 \} \geq 1/n_0,
\] 
contradicting the hypothesis.

Having established this claim, we inductively define sequences $m_n \in \mathbb N$, and $\iota_n \in I_{m_n} \setminus I_{m_n - 1}$ by setting $m_0 = 1$, and choosing $m_n \geq \max \{2^n, m_{n-1} \}$ and $\iota_n \in  I_{m_n} \setminus I_{m_n - 1}$ so that the above claim holds for the integer $\max \{ 2^{n}, m_{n-1} \}$. Note that we have $d( \tau_N \circ \theta_{\iota_n}, \tau_M \circ \theta) \to 0$. 

Since $\lim_{\iota \to \mathcal U} d(\tau_N \circ \theta_\iota, \tau_M \circ \theta) = 0$, we see that for each $n \in \mathbb N$ and $\mathcal U$-almost every $\iota \in I$ there exists $u_\iota \in \mathcal U(\tilde N)$ so that $\|  u_\iota \theta_\iota(x_k) u_\iota^* -  \theta_{\iota_n}(x_k) \|_2 \leq 2^{-n}$ for $k \leq n$. By the triangle inequality we may then find $v_{n} \in \mathcal U(\tilde N)$ so that for $k \leq n$ we have
\begin{equation}\label{eq:approxconj}
\| v_n \theta_{\iota_{n + 1}}(x_k) v_{n}^* - \theta_{\iota_{n}}(x_k) \|_2 \leq 2^{-n + 1}.
\end{equation}

We now consider the sequence of $*$-homomorphisms $\alpha_n: C \to \tilde N$ given by 
\[
\alpha_n(x) = v_1 \cdots v_{n-1} \theta_{\iota_n}(x) v_{n-1}^* v_{n-2}^* \cdots v_1^*.
\] 
From the discussion above we have $\| \alpha_{n + 1}(x_k) - \alpha_n(x_k) \|_2 \leq 2^{-n + 1}$ for $k \leq n$, and hence it follows that for each $k \in \mathbb N$ the sequence $\{ \alpha_n(x_k) \}_n$ is Cauchy in $\| \cdot \|_2$ and hence converges to some $a_k \in \tilde N$. The $*$-homomorphism $\alpha: C \to N$ defined by $\alpha(x_k) = a_k$ then satisfies $\tau_N \circ \alpha = \tau_M \circ \theta$ and so defines a trace-preserving embedding of $M$ into $\tilde N$. 

If the inclusion $N \subset \tilde N$ is existential, then we may find the $v_n$ from (\ref{eq:approxconj}) inside of $N^{\mathcal V}$ for some ultrafilter $\mathcal V$, and it is then easy to see that writing $v_n = (v_{n, \kappa})^{\mathcal V}$ we have that for $\mathcal V$-almost every $\kappa$, (\ref{eq:approxconj}) is satisfied for the unitary $v_{n, \kappa}$. Thus, in this case we may assume $v_n \in N$ so that $\alpha$ maps into $N \subset \tilde N$.
\end{proof}

\begin{rem}
If in the previous theorem one starts with the weaker hypothesis that any two embeddings of $M$ into $N^{\mathcal U}$ are u.c.p.-equivalent in the sense of \cite{AtKE21} (by taking conditional expectations we then no longer need the ambient algebra $\tilde N$), then the argument can be adapted to show that $M$ then has an embedding into $N^{\mathcal U}$ with a u.c.p.\ lift (and hence every embedding of $M$ in to $N^{\mathcal U}$ must have a u.c.p.\ lift). 
\end{rem}

\begin{cor}\label{cor:embedsubalg}
If a II$_1$ factor $M$ has the Jung property and if $N$ is a separable II$_1$ factor such that $M$ embeds into an ultrapower of $N$ and $N$ existentially embeds into an ultrapower of $M$ (e.g., if $M \equiv N$ and if $N$ is existentially closed in its universal theory), then $M$ embeds into $N$. 
\end{cor}
\begin{proof}
Suppose we have $M \subset N^{\mathcal V}$ and we have an existential embedding $N \subset M^{\mathcal U}$. We may assume that $\mathcal V$ is countably incomplete.  By the Jung property, any two embeddings of $M$ into $N^{\mathcal V} \subset (M^{\mathcal U})^{\mathcal V}$ are conjugate by a unitary in $(M^{\mathcal U})^{\mathcal V} \cong M^{\mathcal U \times \mathcal V}$ and so by Theorem~\ref{thm:conjinclusion} we have an embedding of $M$ into $N$.  
\end{proof}

A II$_1$ factor $M$ is locally universal if any separable II$_1$ factor embeds into an ultrapower of $M$. Note that by considering the embeddings $M \subset M \ovt N$ we see that any existentially closed II$_1$ factor is locally universal and McDuff. A II$_1$ factor $M$ is a prime model of its theory if $M$ embeds elementarily into any factor $N$ such that $N \equiv M$ \cite[Section 3.2]{Go21}. 

A II$_1$ factor $M$ is enforceable if it is existentially closed and embeds into any existentially closed factor (this is not the original definition, but is equivalent to it by \cite[Corollary 6.26]{Go21}). Compactness arguments can be used to show that locally universal, and even existentially closed II$_1$ factors, exist (in fact there are continuum many that are pairwise non-isomorphic by \cite[Corollary 1.3]{FaGoHaSh16} \cite[Theorem E]{IoTa24}), but it is open as to whether or not an enforceable II$_1$ factor exists. If such a factor does exist, then it was shown by Goldbring  \cite{Go21} that it is unique up to isomorphism and is not only locally universal but also has the Jung property \cite[Theorem 2]{Go23}. Here we show the converse of Goldbring's result. 

\begin{cor}[cf.\ {\cite[Proposition 2.15]{FaHaRoTi17}}]\label{cor:jungenforceable}
A separable II$_1$ factor is enforceable if and only if it is locally universal and has the Jung property. If an enforceable II$_1$ factor $M$ exists, then it is a prime model of its theory. 
\end{cor}
\begin{proof}
Suppose $M$ is locally universal and has the Jung property. It follows easily that $M$ is existentially closed. If $N$ is any existentially closed II$_1$ factor, then $M$ and $N$ embed into each other's ultrapowers. Hence, $M$ embeds into $N$ by the previous corollary. The fact that $M$ is a prime model of its theory then follows from \cite[Corollary 3.16]{Go21}.
\end{proof}

Note that if $M$ is existentially closed and self-tracially stable (see \cite[Definition 3.3.13]{AtGoKE22} for this definition), then $M$ is locally universal and has the Jung property by \cite{Go23}, thus the above corollary generalizes one direction of Theorem 8 in \cite{Go23}. Further properties of an enforceable II$_1$ factor, should one exist, can be found in \cite{Go21, Go22}.

A II$_1$ factor $M$ is super McDuff if $M' \cap M^{\mathcal U}$ is a II$_1$ factor for every ultrafilter. It is known that being McDuff (even being existentially closed) is an enforceable property \cite{Go21}, however it is open if the property of being super McDuff is enforceable (see \cite[Section 5]{GoJeKEPi25}). We now show that this is equivalent to the existence of an enforceable factor, strengthening \cite[Proposition 6.3.9]{AtGoKE22}.

\begin{proof}[Proof of Theorem~\ref{thm:enforceable}]
Suppose first that $M$ is an enforceable II$_1$ factor, then $M$ is McDuff and has the Jung property. Let $\mathcal U$ be a countably incomplete ultrafilter and let $p \in \mathcal P(M' \cap M^{\mathcal U})$ be a projection with trace $1/2$. Take $u \in \mathcal U(M^{\mathcal U})$ so that $pu = up^\perp$. Since $M$ is McDuff, there exists an isomorphism $\theta: p M^{\mathcal U} p \to M^{\mathcal U}$. By the Jung property there exists $v \in \mathcal U(M^{\mathcal U})$ so that for $w \in \mathcal U(M)$ we have $\theta(p w ) = v \theta( u w p^\perp u^* ) v^*$, and hence $p w \theta^{-1}(v) u w^* = \theta^{-1}(v) u p^\perp$ for all $w \in \mathcal U(M)$. Since $E_{M' \cap M^{\mathcal U}}( \theta^{-1}(v) u)$ is in the strong closure of convex combinations of the form $w \theta^{-1}(v) u w^*$ with $w \in \mathcal U(M)$, we then see that $p E_{M' \cap M^{\mathcal U}}( \theta^{-1}(v) u) = \theta^{-1}(v) u p^\perp$, from which we deduce $E_{M' \cap M^{\mathcal U}}( u^* \theta^{-1}(v^*) ) p E_{M' \cap M^{\mathcal U}}( \theta^{-1}(v) u) = p^\perp$, which shows that $p \sim p^\perp$ in $M' \cap M^{\mathcal U}$. Hence, we see that $\mathcal Z( M' \cap M^{\mathcal U} )$ cannot contain a projection of trace $1/2$ and so cannot be diffuse, which forces $\mathcal Z( M' \cap M^{\mathcal U} ) = \mathbb C$. Hence, $M$ is super McDuff and thus the property of being super McDuff is enforceable.

If being super McDuff is enforceable, then some finitely generic factor is super McDuff. Thus, to finish the proof we may show that every super McDuff finitely generic factor is enforceable. Finitely generic factors are locally universal and have the generalized Jung property \cite{Go21}. From \cite[Proposition 6.3.5]{AtGoKE22}, it then follows that if $M$ is finitely generic and super McDuff, then $M$ has the Jung property. Since $M$ is also locally universal it then follows from the previous corollary that $M$ is an enforceable factor. 
\end{proof}

In \cite{Go20} Goldbring showed that every property (T) factor embeds into an existentially closed II$_1$ factor with factorial relative commutant and hence also embeds with factorial relative commutant into the ultrapowers of this existentially closed II$_1$ factor. It was also observed in \cite[Section 5]{GoJeKEPi25} that if this property were to hold for an ultrapower of a finitely generic factor $N$, then it would follow that $N$ is super McDuff, and hence the previous result would show that an enforceable II$_1$ factor exists. We thus have that either an enforceable II$_1$ factor exists, or else the finitely generic factors cannot be elementarily equivalent to the infinitely generic factors. (This can also be seen using the above results, together with \cite[Theorem 4.12]{GoJeKEPi25}.) It is an open problem if there exist existentially closed II$_1$ factors that are not elementarily equivalent \cite[Question 5.9]{GoHa23}. It is also an open problem whether or not every factor $N$ satisfies the property that any property (T) factor $M$ that embeds into $N^{\mathcal U}$ must also do so in a way that has factorial asymptotic relative commutant. For the case $N = R$, Popa has conjectured that this holds even without the property (T) assumption.

\section{Embeddings into ultraproducts of amenable von Neumann algebras}

 In order to exploit the equivalences given by Theorem~\ref{thm:amenimplieshyperfinite}, it will be useful to collect some techniques to ensure that subalgebras $M \subset \prod_{\iota \to \mathcal U} M_\iota$ are matricially inapproximable with respect to $(M_\iota)_{\iota \in I}$. The purpose of this section is to provide some examples of this. 

First note that if $M \subset \prod_{\iota \to \mathcal U} M_\iota$ is a separable von Neumann subalgebra such that  $M$ does not embed into $R^{\mathcal U}$, then $M$ is matricially inapproximable with respect to $(M_\iota)_{\iota \in I}$. Examples of the former type are extremely difficult to produce \cite{JiNaViWrYu20}, though we note that many of the applications in Section~\ref{sec:embedultra} will become easier to prove if we work with a class of von Neumann algebras that do not embed into $R^{\mathcal U}$. 

The following is also a simple observation, and it is much easier to find examples where this applies. 

\begin{prop}\label{prop:amenmatriciallyapprox}
Suppose $N$ is a II$_1$ factor and $M \subset N \subset N^{\mathcal U}$ is a von Neumann subalgebra, then $M$ is matricially inapproximable with respect to $N$ if and only $M$ is nonamenable. 
\end{prop}

The previous proposition can easily be proved directly.  We note, however, that using ideas from \cite[Theorem 3.7]{AtKE21}, we may prove the following more general result. 

\begin{thm}\label{thm:ucplift}
Suppose $N_\iota$ are II$_1$ factors and $M \subset \prod_{\iota \to \mathcal U} N_\iota$ is a separable von Neumann subalgebra such that $M$ has a unital complete positive lift to $\ell^\infty(I, M_\iota)$ (equivalently, there exist unital completely positive maps $\phi_\iota: M \to N_\iota$ such that $x = (\phi_\iota(x))^{\mathcal U}$ for all $x \in M$), then $M$ is matricially approximable with respect to $(N_\iota)_\iota$ if and only if $M$ is amenable. 
\end{thm}
\begin{proof}
We only need to prove the forward implication. Suppose that $\phi: M \to \ell^\infty( I, N_\iota)$ gives a u.c.p.\ lift and suppose there exist injective von Neumann subalgebras $R_\iota \subset M_\iota$ so that $M \subset \prod_{\iota \to \mathcal U} R_\iota \subset \prod_{\iota \to \mathcal U} N_\iota$. Consider the normal conditional expectation $\mathbb E: \ell^\infty( I, N_\iota) \to \ell^\infty(I, R_\iota)$ given by $\mathbb E((x_\iota)_\iota) = ( \mathbb E_{R_\iota}(x_\iota) )_\iota$ and let $q: \ell^\infty( I, N_\iota) \to \prod_{\iota \to \mathcal U} N_\iota$ be the canonical quotient map. Since $M \subset \prod_{\iota \to \mathcal U} R_\iota$ it follows that $q \circ \mathbb E \circ \phi$ gives the identity map. Since the identity map has a u.c.p.\ factorization through the injective von Neumann algebra $\ell^\infty(I, R_\iota)$, we then have that $M$ is injective. 
\end{proof}

Another class of examples of matricially inapproximable subalgebras comes from the observation that if $\mathcal U$ is countably incomplete and $M \subset  \prod_{\iota \to \mathcal U} M_\iota$ is separable and matricially approximable with respect to $(N_\iota)_\iota$, then there are matrix algebras $\mathbb M_{n_\iota}(\mathbb C) \subset M_\iota$ so that $M \subset \prod_{\iota \to \mathcal U} \mathbb M_{n_\iota}(\mathbb C) \subset \prod_{\iota \to \mathcal U} M_\iota$ and hence  
\[
M' \cap \prod_{\iota \to \mathcal U} M_\iota \supset \left(\prod_{\iota \to \mathcal U} \mathbb M_{n_\iota}(\mathbb C) \right)' \cap \left( \prod_{\iota \to \mathcal U} M_\iota \right) = \prod_{\iota \to \mathcal U} (\mathbb M_{n_\iota}(\mathbb C)' \cap M_\iota).
\] 
We therefore have the following result.

\begin{prop}
Suppose $\mathcal U$ is countably incomplete and $M \subset  \prod_{\iota \to \mathcal U} M_\iota$ is a separable irreducible subfactor, then $M$ is matricially inapproximable with respect to $(M_\iota)_\iota$. 
\end{prop}

\subsection{Stability under compressions}

\begin{lem}\label{lem:subprojection}
Let $M$ be a finite von Neumann algebra and $p \in \mathcal P(M)$ a nonzero projection, then there exist projections $q \in \mathcal P(M)$ with $\mathbb E_{\mathcal Z(M)}(q) \in \mathbb C$, and $r \in \mathcal P(\mathcal Z(M))$ so that $qr \not= 0$ and $qr \leq p$. 
\end{lem}
\begin{proof}
By \cite[Section 1.2]{Po14} we may find a separable von Neumann subalgebra $M_0 \subset M$ such that $p \in M_0$ and $\mathbb E_{\mathcal Z(M_0)}(x)  = \mathbb E_{\mathcal Z(M)}(x)$ for $x \in M_0$, so we will assume $M$ is separable. We then consider a direct integral decomposition into factors $M = \int_X^\oplus M_x \, d\mu(x)$ where $(X, \mu)$ is a standard probability space and we identify $L^\infty(X, \mu)$ with $\mathcal Z(M)$, and $(M_x, \tau_x)_{x \in X}$ is a measurable field of finite factors. 

We may then decompose $p$ as $p = \int p_x \, d\mu(x)$, where $p_x \in \mathcal P(M_x)$. We fix some $c > 0$ so that the set $E = \{ x \in X \mid \tau_x(p_x) \geq c \}$ has positive measure. We then choose a measurable field of projections $(q_x)_{x \in X}$ so that $\tau(q_x) = c$ and $q_x \leq p_x$ for $x \in E$. Setting $q = \int q_x \, d\mu(x)$ and $r = 1_E$, we then have $\mathbb E_{\mathcal Z(M)}(q) = c$, $qr \not= 0$, and $qr \leq p$. 
\end{proof}

\begin{prop}\label{prop:cutdownapprox}
Suppose $\mathcal U$ is countably incomplete, $N_\iota$ are II$_1$ factors and $M \subset \prod_{\iota \to \mathcal U} N_\iota$ is a von Neumann subalgebra. Suppose also that $p \in \mathcal P(M)$ with $p = (p_\iota)^{\mathcal U}$ for $p_\iota \in \mathcal P(N_\iota)$, and $q \in \mathcal P(M' \cap \prod_{\iota \to \mathcal U} N_\iota)$ with $q = (q_\iota)^{\mathcal U}$ for $q_\iota \in \mathcal P(N_\iota)$. If $M$ is matricially approximable with respect to $(N_\iota)_\iota$, then $pMp$ is matricially approximable with respect to $(p_\iota N_\iota p_\iota)_\iota$ and $Mq$ is matricially approximable with respect to $(q_\iota N_\iota q_\iota)_\iota$. The converse of the assertions hold if $p$ and $q$ have central support $1$, respectively. 
\end{prop}
\begin{proof}
It suffices to consider the case when $M$ is separable. If $M$ is matricially approximable with respect to $(N_\iota)_\iota$, then there exist finite-dimensional factors $A_\iota \subset N_\iota$ so that $M \subset \prod_{\iota \to \mathcal U} A_\iota \subset \prod_{\iota \to \mathcal U} N_\iota$. If $p \in \mathcal P(M)$ is as above, then we may assume that $p_\iota \in \mathcal P(A_\iota)$ so that we then have $pMp \subset \prod_{\iota \to \mathcal U} p_\iota A_\iota p_\iota \subset \prod_{\iota \to \mathcal U} p_\iota N_\iota p_\iota$, showing that $pMp$ is matricially approximable with respect to $(p_\iota N_\iota p_\iota)_\iota$.

Conversely, if $pMp$ is matricially approximable with respect to $(p_\iota N_\iota p_\iota)_\iota$ and if $p$ has central support $1$, then we may find a sequence of partial isometries $(v_n)_n \subset M$ so that $v_n^* v_n \leq p$ and $\sum_n v_n v_n^* = 1$. If $A_\iota \subset p_\iota N_\iota p_\iota$ are finite-dimensional von Neumann subalgebras such that $pMp \subset \prod_{\iota \to \mathcal U} A_\iota$, then for any $n \in \mathbb N$ we may choose representations $v_n = (v_{n, \iota})^{\mathcal U}$ with $v_{n, \iota}$ partial isometries in $N_\iota$ so that $v_{n, \iota}^* v_{n, \iota} \leq p_\iota$ and $r_\iota = \sum_{n = 1}^N v_{n, \iota} v_{n, \iota}^* \in \mathcal P(N_\iota)$. Setting $r = (r_\iota)^{\mathcal U}$ we then have 
\[
r M r \subset \prod_{\iota \to \mathcal U} \left( \oplus_{n =1}^N v_{n, \iota} A_\iota v_{n, \iota} \right)  \subset r \left(\prod_{\iota \to \mathcal U} N_\iota \right)r.
\] 
Hence, we see that $rMr$ is matricially approximable with respect to $(r_\iota N_\iota r_\iota)$. Since $\mathcal U$ is countably incomplete we see that $M$ is matricially approximable with respect to $(N_\iota)_\iota$.

We now suppose again that $M$ is matricially approximable with respect to $(N_\iota)_\iota$ so that $M \subset \prod_{\iota \to \mathcal U} A_\iota \subset \prod_{\iota \to \mathcal U} N_\iota$ where $A_\iota$ are finite-dimensional factors. Note that 
\[
M' \cap \prod_{\iota \to \mathcal U} N_\iota 
\supset \left( \prod_{\iota \to \mathcal U} A_\iota \right)' \cap \left( \prod_{\iota \to \mathcal U} N_\iota \right) 
= \prod_{\iota \to \mathcal U} (A_\iota' \cap N_\iota),
\] 
and $\prod_{\iota \to \mathcal U} (A_\iota' \cap N_\iota)$ is a factor. 

Suppose first $q \in \mathcal P(M' \cap \prod_{\iota \to \mathcal U} N_\iota)$ has scalar central trace. Then we may find a unitary $u \in \mathcal U( M' \cap \prod_{\iota \to \mathcal U} N_\iota)$ so that $u^* q u \in \prod_{\iota \to \mathcal U} (A_\iota' \cap N_\iota)$ and taking a representation $u^*q u = (r_\iota)^{\mathcal U}$ with $r_\iota \in \mathcal P(A_\iota' \cap N_\iota)$ we then have 
\[
M u^* q u \subset \prod_{\iota \to \mathcal U} A_\iota r_\iota \subset u^*q u \left( \prod_{\iota \to \mathcal U} N_\iota \right) u^* q u,
\] 
and hence $M q \subset u ( \prod_{\iota \to \mathcal U} A_\iota r_\iota ) u^* \subset q  ( \prod_{\iota \to \mathcal U} N_\iota )  q$ showing that $Mq$ is matricially approximable with respect to $(q_\iota N_\iota q_\iota)_\iota$. 

The next case we consider is when $q \in \mathcal Z(M' \cap \prod_{\iota \to \mathcal U} N_\iota)$. Since $ \prod_{\iota \to \mathcal U} (A_\iota' \cap N_\iota) \subset M' \cap \prod_{\iota \to \mathcal U} N_\iota$ we then have
\[
\mathcal Z\left(M' \cap \prod_{\iota \to \mathcal U} N_\iota \right) \subset M' \cap \prod_{\iota \to \mathcal U} A_\iota
\]
so that taking a representation $q = (q_\iota)^{\mathcal U}$ with $q_\iota \in \mathcal P(A_\iota)$ we have
\[
Mq \subset \prod_{\iota \to \mathcal U} q_\iota A_\iota q_\iota \subset q \left(\prod_{\iota \to \mathcal U} N_\iota \right)q
\]
again showing that $Mq$ is matricially approximable with respect to $(q_\iota N_\iota q_\iota)_\iota$.

For the general case $q \in \mathcal P(M' \cap \prod_{\iota \to \mathcal U} N_\iota)$, we consider a maximal family of nonzero pairwise orthogonal projections $\{ q_k \}_k \subset \mathcal P(M' \cap \prod_{\iota \to \mathcal U} N_\iota)$ such that choosing a representation $q_k = (q_{k ,\iota})^{\mathcal U}$ we have that $Mq_k$ is matricially approximable with respect to $(q_{k, \iota} N_\iota q_{k, \iota})_\iota$. If $q - \sum_k q_k \not= 0$, then by Lemma~\ref{lem:subprojection} there is a projection $r_1 \in \mathcal P(M' \cap \prod_{\iota \to \mathcal U} N_\iota)$ with scalar central trace and a projection $r_2 \in \mathcal P(\mathcal Z( M' \cap \prod_{\iota \to \mathcal U} N_\iota ))$ so that $r_1 r_2 \not= 0$ and $r_1 r_2 \leq q - \sum_k q_k$. However, the argument above would then contradict the maximality of the family $\{ q_k \}_k$ and so we conclude that $q = \sum_k q_k$. As explained above, this then shows that $Mq$ is matricially approximable with respect to $(q_\iota N_\iota q_\iota)_\iota$.

The converse assertion that when $Mq$ is matricially approximable with respect to $(q_\iota N_\iota q_\iota)_\iota$ and $q \in \mathcal P(M' \cap \prod_{\iota \to \mathcal U} N_\iota)$ has central support $1$ follows just as in the second paragraph above. 
\end{proof}

\subsection{Spectral gap and back-and-forth methods for matricial inapproximability}
For additional examples where we do not have matricial approximations with respect to $(N_\iota)_\iota$, we recall that a von Neumann subalgebra $M \subset N$ has weak spectral gap (see, e.g., \cite{Po08, Po12}) if for all $\varepsilon > 0$, there are $u_1, \ldots, u_n \in \mathcal U(M)$ and $\delta > 0$ so that for all $x \in (N)_1$ we have
\[
\sum_{k = 1}^n \| [x, u_k ] \|_2 \leq \delta \implies \| x - \mathbb E_{M' \cap N}(x) \|_2 \leq \varepsilon \| x \|_2.
\]
Equivalently, $M \subset N$ has weak spectral gap if for any ultrafilter $\mathcal U$ we have $M' \cap N^{\mathcal U} = (M' \cap N)^{\mathcal U}$. A simple observation (see \cite[Lemma 3.3.2]{Po95}), which will be used below, is that if $M$ has property (T), then any inclusion $M \subset N$ will have weak spectral gap.

The following theorem (and its proof) will be generalized in Theorem~\ref{thm:genspectralgapapprox}, but we include here the statement and proof in this simpler case for clarity of exposition. 

\begin{thm}\label{thm:backandforth}
Suppose $\mathcal U$ is countably incomplete, $N_1$ and $N_2$ are II$_1$ factors, $M \subset N_1$ is a separable diffuse von Neumann subalgebra with weak spectral gap, and $\theta: N_1^{\mathcal U} \to N_2^{\mathcal V}$ is an isomorphism. Then, either $\theta(M)$ is matricially inapproximable with respect to $N_2$, or else  $\theta^{-1}(P)$ is matricially inapproximable with respect to $N_1$ for every nonamenable von Neumann subalgebra $P \subset N_2$. 
\end{thm}
\begin{proof}
We suppose $\theta(M)$ is matricially approximable with respect to $N_2$. Then there exist finite-dimensional subfactors $B_\kappa \subset N_2$ so that $\theta(M) \subset \prod_{\kappa \to \mathcal V} B_\kappa \subset N_2^{\mathcal V}$. Since $M \subset N_1$ has weak spectral gap we have 
\[
 M^{\mathcal U}  
\subset  ( (M' \cap N_1)' \cap N_1)^{\mathcal U}
= ( (M' \cap N_1)^{\mathcal U} )' \cap N_1^{\mathcal U} 
= (M' \cap N_1^{\mathcal U})' \cap N_1^{\mathcal U}.
\]
Hence, 
\[
\theta( M^{\mathcal U}) 
\subset (\theta(M)' \cap N_2^{\mathcal V})' \cap N_2^{\mathcal V}
\subset \left( \left(\prod_{\kappa \to \mathcal V} B_\kappa \right)' \cap N_2^{\mathcal V} \right)' \cap N_2^{\mathcal V}
= \prod_{\kappa \to \mathcal V} B_\kappa.
\]

Let $P \subset N_2$ be a separable von Neumann subalgebra and suppose $\theta^{-1}(P)$ is matricially approximable with respect to $N_1$. Let $A_\iota \subset N_1$ be finite-dimensional subfactors such that $\theta^{-1}(P) \subset \prod_{\iota \to \mathcal U} A_\iota \subset N_1^{\mathcal U}$. Since $M$ is type II$_1$ and $A_\iota$ is finite-dimensional, there exists a unitary $u_\iota \in \mathcal U(N_1)$ so that $u_\iota A_\iota u_\iota^* \subset M$. Set $u = (u_\iota)^{\mathcal U} \in \mathcal U(N_1^{\mathcal U})$, so that $u \theta^{-1}(P) u^* \subset M^{\mathcal U} \subset N_1^{\mathcal U}$. 

From above, we then have $P  \subset \theta(u)^* \left(\prod_{\kappa \to \mathcal V} B_\kappa \right) \theta(u)$, and so $P$ is amenable by Proposition~\ref{prop:amenmatriciallyapprox}.
\end{proof}

Generalizing the weak spectral gap property, if $M$ is a separable tracial von Neumann algebra, $(N_\iota)_\iota$ is a family of tracial von Neumann algebras, and $\alpha_\iota: M \to N_\iota$ gives a family of embeddings, then we say that the family $(\alpha_\iota)_\iota$ has uniform weak spectral gap if for each $\varepsilon > 0$ there are $u_1, \ldots, u_n \in \mathcal U(M)$ and $\delta > 0$ so that for for any $\iota \in I$ and  $x_\iota \in (N_\iota)_1$ we have
\[
\sum_{k = 1}^n \| [ x_\iota, \alpha_\iota(u_k) ] \|_2 < \delta \implies \| x_\iota - \mathbb E_{\alpha_\iota(M)' \cap N}(x_\iota) \|_2 \leq \| x_\iota \|_2.
\]
Note that if $M$ is a property (T) factor, then any family of embeddings $\alpha_\iota: M \to N_\iota$ will have uniform weak spectral gap. Another case where this happens is if $M \subset N = N_\iota$ has weak spectral gap and $\alpha_\iota$ extend to automorphisms on $N_\iota$.

We then have the following technical generalization of Theorem~\ref{thm:backandforth}.

\begin{thm}\label{thm:genspectralgapapprox}
Suppose $(N_\iota)_{\iota \in I}$ and $(P_\kappa)_{\kappa \in K}$ are families of finite factors, $M$ is a separable II$_1$ factor and  $\alpha_\iota: M \to N_\iota$ are embeddings with associated embedding $\alpha = (\alpha_\iota)^{\mathcal U}: M \to \prod_{\iota \to \mathcal U} N_\iota$. Suppose $(\alpha_\iota)_\iota$ has uniform weak spectral gap, and suppose also that we have an isomorphism $\theta: \prod_{\iota \to \mathcal U} N_\iota \to \prod_{\kappa \in \mathcal V} P_\kappa$. 

If $\theta(\alpha(M))$ is matricially approximable with respect to  $( P_\kappa )_\kappa$ and if $Q \subset \prod_{\kappa \in \mathcal V} P_\kappa$ is a separable von Neumann algebra such that $\theta^{-1}(Q)$ is matricially approximable with respect to $(N_\iota)_\iota$, then $Q$ is matricially approximable with respect to $(P_\kappa)_\kappa$. If, moreover, $Q$ has a u.c.p.\ lift to $\ell^\infty(K, P_\kappa)$, then $Q$ is amenable. 
\end{thm}
\begin{proof}
If $\theta(\alpha(M)) \subset \prod_{\kappa \to \mathcal V} B_\kappa \subset \prod_{\kappa \to \mathcal V} P_\kappa$ with $B_\kappa \subset P_\kappa$ finite-dimensional subfactors, then just as in the proof of Theorem~\ref{thm:backandforth} we see that
\[
\theta\left(\prod_{\iota \to \mathcal U} \alpha_\iota(M)\right) \subset \prod_{\kappa \to \mathcal V} B_\kappa \subset \prod_{\kappa \to \mathcal V} P_\kappa.
\]
If $\theta^{-1}(Q) \subset \prod_{\iota \to \mathcal U} A_\iota \subset \prod_{\iota \to \mathcal U} N_\iota$ with $A_\iota \subset N_\iota$ finite-dimensional subfactors, then we may find $u_\iota \in \mathcal U(N_\iota)$ so that $u_\iota A_\iota u_\iota^* \subset \alpha_\iota(M)$, and so setting $u = (u_\iota)^{\mathcal U}$ we have $u \left( \prod_{\iota \to \mathcal U} A_\iota \right) u^* \subset  \prod_{\iota \to \mathcal U} \alpha_\iota(M)$ and hence 
\[
\theta(u) Q \theta(u^*) \subset \prod_{\kappa \to \mathcal V} B_\kappa \subset \prod_{\kappa \to \mathcal V} P_\kappa.
\]

The last assertion of the theorem then follows from Theorem~\ref{thm:ucplift}.
\end{proof}

\section{Connes's rigidity argument in the ultraproduct setting}

The first rigidity result for II$_1$ factors was obtained by Connes in \cite{Co80} where he shows that i.c.c. (infinite non-trivial conjugacy classes) property (T) groups give rise to II$_1$ factors with countable fundamental and outer automorphism groups. This galvanizing result of Connes has been adapted to become one of the precepts of deformation/rigidity theory as can be seen in \cite{Po86, Po06b, Oz04, IoPePo08, Ta23}. 

In this section, we adapt Connes's rigidity argument to the setting where the property (T) factor embeds into an ultraproduct. For matricial ultraproducts, this situation has been previously considered in \cite{JuSh07, HaJeKE25}. In our setting we need a strengthening of these arguments, which shows that not only are asymptotic embeddings of property (T) factors in some sense ``discrete'', but that they have ``discrete clusters'' where within each cluster they are pairwise approximate conjugate in a uniform way.

The essential idea in Connes's argument is that from a trace-preserving automorphism we may twist the trivial bimodule by this automorphism to obtain a new bimodule. Automorphisms that are close to the identity on a Kazhdan set then must have an intertwiner for this new bimodule that is close in $\| \cdot \|_2$ to the identity operator. For factors, any intertwiner must be a scalar multiple of a unitary, which shows that the inner-automorphism group must be open. In the setting where we deal with property (T) subfactors, the following well-known proposition is therefore useful.

\begin{prop}\label{prop:intertwinerinner}
Suppose $N$ is a II$_1$ factor and $M_1, M_2 \subset N$ are subfactors such that $M_1' \cap N$ and $M_2' \cap N$ are factors. If $\alpha : M_1 \to M_2$ is a trace-preserving isomorphism such that there exists a nonzero operator $v \in N$ satisfying $\alpha(x) v = v x$ for all $x \in M_1$, then there exists $u \in \mathcal U(N)$ such that $\alpha(x) = u x u^*$ for all $x \in M_1$. 
\end{prop}
\begin{proof}
Replacing $v$ with $v | v |^{-1}$, we may assume that $v$ is a nonzero partial isometry with $v^* v \in M_1' \cap N$ and $vv^* \in M_2' \cap N$. We let $\mathcal F$ denote the set of all such intertwining partial isometries and define a partial ordering on $\mathcal F$ by $v \prec w$ if $v^*v \leq w^*w$, $vv^* \leq ww^*$, and $wv^*v = v$. By Zorn's Lemma we may find $u \in \mathcal F$ that is maximal with respect to this partial ordering. 

Set $p = u^*u \in M_1' \cap N$ and $q = u u^* \in M_2' \cap N$. If $u \not\in \mathcal U(N)$, then as $M_1' \cap N$ and $M_2' \cap N$ are factors we may choose partial isometries $w_1 \in M_1' \cap N$ and $w_2 \in M_2' \cap N$ so that $w_1w_1^* \leq p$, $w_1^*w_1 \leq p^\perp$, $w_2w_2^* \leq q$,$w_2^*w_2 \leq q^\perp$ and $w_2w_2^* u = u w_1 w_1^*  \not= 0$. Considering $u + w_2^* u w_1$ would then contradict maximality of $u \in \mathcal F$. Thus, we conclude that $u \in \mathcal U(N)$. 
\end{proof}

\begin{thm}\label{thm:propTultra}
Suppose $(N_\iota)_\iota$ are II$_1$ factors and $M \subset \prod_{\iota \to \mathcal U} N_\iota$ is a subfactor with property (T) such that $M' \cap \prod_{\iota \to \mathcal U} N_\iota$ is a factor. Suppose $\alpha_\iota \in \Aut(N_\iota)$ is in the connected component of the identity and let $\alpha = (\alpha_\iota)^{\mathcal U} \in \Aut(\prod_{\iota \to \mathcal U} N_\iota)$, then there exists $u \in \mathcal U(\prod_{\iota \to \mathcal U} N_\iota)$ so that $\alpha(x) = ux u^*$ for all $x \in M$. 
\end{thm}
\begin{proof}
We fix a decreasing sequence $I_n \subset I$ so that $\cap_n I_n = \emptyset$ and $\mathcal U(I_n) = 1$ for each $n \in \mathbb N$.  Fix $c > 0$ and $S \subset M$ finite so that if $P$ denotes the projection onto the space of $M$ central vectors for some normal Hilbert bimodule, then $\| P^\perp \xi \| \leq c ( \sum_{x \in S} \| [x, \xi] \|^2 )^{1/2}$. Let $\theta_\iota: S \to N_\iota$ so that $\| \theta_\iota(x) \| \leq \| x \|$ and $(\theta_\iota(x) )^{\mathcal U} = x$ for each $x \in S$.

For each $\iota \in I$, let $\rho_\iota: \Aut(N_\iota) \to [0, \infty)$ be the continuous function defined by 
\[
\rho_\iota(\beta) = \inf_{u \in \mathcal U(N_\iota)} ( \sum_{x \in S} \| u \beta( \theta_\iota(x) ) u^* - \theta_\iota(x) \|_2^2 )^{1/2},
\] 
and let 
\[
\mathcal G_\iota = \{ \beta \in  \Aut(N_\iota) \mid \rho_\iota(\beta) \leq 1/2c \}.
\]
Note that $\operatorname{Inn}(N_\iota) \subset \mathcal G_\iota$ and $\mathcal G_\iota$ is a closed set. We claim that for $\mathcal U$-almost every $\iota \in I$, the set $\mathcal G_\iota$ is also open. 

To see this, for $\beta_1, \beta_2  \in \Aut(N_\iota)$ set $d_{\iota}(\beta_1, \beta_2) = ( \sum_{x \in S} \| \beta_1( \theta_\iota(x) ) - \beta_2( \theta_\iota(x) ) \|_2^2 )^{1/2}$. This is a continuous pseudometric on $\Aut(N_\iota)$, and by the triangle inequality we have $\rho_\iota(\beta_1) \leq \rho_\iota(\beta_2) + d_{ \iota}(\beta_1, \beta_2)$.

If $\mathcal G_\iota$ were not open for $\mathcal U$-almost every $\iota \in I$, then we could choose $\alpha_\iota \in \mathcal G_\iota$, and $\beta_\iota \in \Aut(N_\iota) \setminus \mathcal G_\iota$ so that 
\[
\lim_{\iota \to \mathcal U} d_{S, \iota}(\alpha_\iota, \beta_\iota) = 0.
\]
We then choose for $\iota \in I_n \setminus I_{n + 1}$ a unitary $u_\iota \in \mathcal U(N_\iota)$ so that $ ( \sum_{x \in S} \| u_\iota \alpha_\iota( \theta_\iota(x) ) u_\iota^* - \theta_\iota(x) \|_2^2 )^{1/2} \leq 1/2c + 1/n$. We define the automorphism $\tilde \alpha \in \Aut(\prod_{\iota \to \mathcal U} N_\iota)$ by $\tilde \alpha( ( x_\iota)^{\mathcal U}) = (u_\iota \alpha_\iota( x_\iota) u_\iota^* )^{\mathcal U}$. By construction we have
\[
( \sum_{x \in S} \| \tilde \alpha( x ) - x \|_2^2 )^{1/2} \leq  1/2 c.
\]
There then exists $w \in \prod_{\iota \to \mathcal U} N_\iota$ with $\| w - 1 \|_2 \leq 1/2$ so that $ \tilde \alpha(x) w = w x$ for $x \in M$. Proposition~\ref{prop:intertwinerinner} then provides a unitary $v \in \mathcal U(\prod_{\iota \to \mathcal U} N_\iota)$ so that $\tilde \alpha(x) v = v x$ for all $x \in M$. 
Choosing a representation $v = (v_\iota)^{\mathcal U}$ with $v_\iota \in \mathcal U(N)$ we then have
\[
\lim_{\iota \to \mathcal U} \| v_\iota^* u_\iota \alpha_\iota( \theta_\iota(x) ) u_\iota^*  v_\iota - \theta_\iota(x) \|_2 = 0
\]
for each $x \in S$. Thus, we have $\lim_{\iota \to \mathcal U} \rho_\iota( \alpha_\iota)  = 0$, and hence $\lim_{\iota \to \mathcal U} \rho_\iota(\beta_\iota) = 0$, contradicting the fact that $\beta_\iota \not \in \mathcal G_\iota$ for $\mathcal U$-almost every $\iota$. 

We therefore conclude that $\mathcal G_\iota$ is a clopen neighborhood of $\operatorname{Inn}(N)$ for $\mathcal U$-almost every $\iota \in I$. If $\alpha_\iota \in \Aut(N)$ is in the connected component of the identity, it then follows that $\alpha_\iota \in \mathcal G_\iota$ for $\mathcal U$-almost every $\iota$. But then, repeating the argument above, it follows that there exists $u \in \mathcal U(\prod_{\iota \to \mathcal U} N_\iota)$ so that $\alpha(x) = u xu^*$ for all $x \in M$. 
\end{proof}

\section{Asymptotic relative commutants and intertwining in free products}

Here we adapt some of the techniques on intertwining and controlling relative commutants from \cite{IoPePo08} to the setting of ultraproducts of  free products. This will allow us to adapt Popa's ``local to global" intertwining arguments (e.g., as in \ \cite[Appendix]{Po06b}) for subalgebras of ultraproducts of free products. This, in turn, will allow us to extend previous deformation/rigidity arguments using ``free malleable deformations'' from \cite{Po06, Po07, IoPePo08}.

The following proposition is well-known and follows from analyzing the Fock space representation used in the definition of amalgamated free products \cite{Vo85, Po93}.

\begin{prop}\label{prop:coarsebimoduletensorafp}
Suppose $N = P *_B Q$ where $(P, \tau_P)$ and $(Q, \tau_Q)$ are tracial von Neuman algebras containing a common amenable von Neumann subalgebra $(B, \tau)$. Let $\mathcal H$ denote the Hilbert $N$-$N$ sub-bimodule of $ L^2\langle N, P \rangle  \otimes_N L^2 \langle N, P \rangle $ that is generated by the vector $e_P \otimes_N e_P$. We then have weak containment of $N$-$N$ bimodules
\[
( L^2\langle N, P \rangle  \otimes_N L^2 \langle N, P \rangle   ) \ominus \mathcal H \prec L^2 N \overline \otimes L^2 N.
\]
\end{prop}

\begin{prop}\label{prop:coarsebimoduleafp}
Let $(N, \tau)$ be a tracial von Neumann algebra and let $B \subset N$ be an amenable von Neumann subalgebra. If $(Q, \tau_Q)$ is any tracial von Neumann algebra containing $(B, \tau)$, then as an $N$-$N$-bimodule we have $L^2(N *_B Q) \ominus L^2 N  \prec L^2 N  \ovt L^2 N $.
\end{prop}

The following result will be used in the sequel to control asymptotic relative commutants in free products. This will allow us to exploit property (T) in asymptotic deformation/rigidity arguments. This is an ultraproduct counterpart to similar relative commutants as computed in \cite{Po83} or \cite[Section 1]{IoPePo08}.

\begin{prop}\label{prop:relcommutant}
Suppose $N_\iota$ are II$_1$ factors with amenable von Neumann subalgebras $B_\iota \subset N_\iota$. Suppose $p_\iota \in \mathcal P(N_\iota)$ are such that $p = (p_\iota)^{\mathcal U}$ is nonzero and $M \subset \prod_{\iota \to \mathcal U} p_\iota N_\iota p_\iota$ is a subfactor such that $M$ is matricially inapproximable with respect to $(p_\iota N_\iota p_\iota )_\iota$. 
Then, for any family of tracial von Neumann algebras $\{ (Q_\iota, \tau_\iota) \}_\iota$ such that we have tracial embeddings $B_\iota \subset Q_\iota$ we have 
\[
M' \cap p \left( \prod_{\iota \to \mathcal U} N_\iota *_{B_\iota} Q_\iota \right) p \subset p \left( \prod_{\iota \to \mathcal U} N_\iota \right) p.
\]
\end{prop}
\begin{proof}
By Proposition~\ref{prop:coarsebimoduleafp}, the $N_\iota$-$N_\iota$ bimodule $L^2(N_\iota *_{B_\iota} Q_\iota) \ominus L^2(N_\iota)$ is weakly contained in the coarse bimodule. Thus, if $M' \cap p ( \prod_{\iota \to \mathcal U} N_\iota *_{B_\iota} Q_\iota ) p \not\subset  p (\prod_{\iota \to \mathcal U} N_\iota )p$, then since $M$ is a factor we would have weak containment of $M$-bimodules $L^2M \prec \prod_{\iota \to \mathcal U} (L^2 (p_\iota N_\iota p_\iota)  \ovt L^2 (p_\iota N_\iota p_\iota) )$. Theorem~\ref{thm:amenimplieshyperfinite} then gives the result.  
\end{proof}

\begin{rem}\label{rem:normalizer}
The previous result can be extended to show that, under the same hypotheses, there does not even exist a nonzero element $x \in \prod_{\iota \to \mathcal U} N_\iota *_{B_\iota} Q_\iota$ with $\mathbb E_{ \prod_{\iota \to \mathcal U} N_\iota}(x) = 0$ such that $M x \subset x  \prod_{\iota \to \mathcal U} N_\iota$. Indeed, if given such an $x$ we let $\tilde x$ be the corresponding element in a copy of $\prod_{\iota \to \mathcal U} N_\iota *_{B_\iota} Q_\iota$ so that we may then identify $x \tilde x^*$ with an element in 
\[
M' \cap \left( \prod_{\iota \to \mathcal U} (N_\iota *_{B_\iota}  Q_\iota ) *_{N_\iota} ( N_\iota *_{B_\iota} Q_\iota) \right) 
\cong
M' \cap \left( \prod_{\iota \to \mathcal U} N_\iota *_{B_\iota} ( Q_\iota *_{B_\iota} Q_\iota) \right) \subset \prod_{\iota \to \mathcal U} N_\iota.
\]
It then follows from Proposition~\ref{prop:relcommutant} that $x \tilde x^* = \mathbb E_{ \prod_{\iota \to \mathcal U} N_\iota } (x \tilde x^*)$, but since $\mathbb E_{ \prod_{\iota \to \mathcal U} N_\iota}(x) = 0$ we then have $x \tilde x^* = 0$, and so
\[
 \mathbb E_{ \prod_{\iota \to \mathcal U} N_\iota }  (x \mathbb E_{ \prod_{\iota \to \mathcal U} N_\iota }  (x^* x) x^*)
= \mathbb E_{ \prod_{\iota \to \mathcal U} N_\iota }  (x \tilde x^* \tilde x x^*)
= 0,
\]
from which it follows that $x = 0$.
\end{rem}

The following result will allow us to weaken the hypothesis of Theorem~\ref{thm:genspectralgapapprox} in the free product setting.

\begin{prop}\label{prop:ultrabackandforth}
Suppose $N_\iota$ are II$_1$ factors and $M_\iota \subset N_\iota$ are type II$_1$ von Neumann subalgebras. Suppose $P$ and $Q$ are tracial von Neumann algebras containing a common amenable tracial von Neumann algebra $B$. Suppose $\theta: \prod_{\iota \to \mathcal U} N_\iota \to (P *_B Q)^{\mathcal V}$ is an isomorphism such that $\theta \left( \prod_{\iota \to \mathcal U} M_\iota \right) \subset P^{\mathcal V}$. 

If $Q_0 \subset Q^{\mathcal V}$ is a separable von Neumann algebra such that $\theta^{-1}(Q_0)$ is matricially approximable with respect to $(N_\iota)_\iota$, then $Q_0$ is matricially approximable with respect to $P *_B Q$. If, moreover, $Q$ has a u.c.p.\ lift to $\ell^\infty(I, Q_\kappa)$, then $Q$ is amenable. 
\end{prop}
\begin{proof}
If $\theta^{-1} (Q_0) \subset \prod_{\iota \to \mathcal U} A_\iota \subset \prod_{\iota \to \mathcal U} N_\iota$ with $A_\iota \subset N_\iota$ finite-dimensional factors, then as $M_\iota$ is type II$_1$ there exists a unitary $u \in \mathcal U( \prod_{\iota \to \mathcal U} N_\iota)$ so that $u^* \theta^{-1}(Q_0) u \subset \prod_{\iota \to \mathcal U} M_\iota$ and hence $\theta(u^*) Q_0 \theta(u) \subset \theta \left( \prod_{\iota \to \mathcal U} M_\iota \right) \subset P^{\mathcal V}$. This then gives weak containment of Hilbert $Q_0$-$Q_0$ bimodules $L^2 Q_0 \prec L^2(( P *_B Q )^{\mathcal V}) \oovt{P^{\mathcal V}} L^2(( P *_B Q)^{\mathcal V})$. Since $B$ is amenable, we have that $L^2(P *_B Q)$ is a weakly coarse $P$-$Q$ bimodule, and then applying Lemma~\ref{lem:ultrafusion}, Proposition~\ref{prop:weakcontainmentultra} and Theorem~\ref{thm:amenimplieshyperfinite} we obtain the result.

The last assertion of the theorem then follows from Theorem~\ref{thm:ucplift}.
\end{proof}

\begin{rem}
In the proof of the previous theorem, after concluding that $\theta(u^*) Q_0 \theta(u) \subset \theta \left( \prod_{\iota \to \mathcal U} M_\iota \right) \subset P^{\mathcal V}$, we could also use Remark~\ref{rem:normalizer} to conclude that $Q_0$ is matricially approximable with respect to $P *_B Q$. 
\end{rem}

We now fix $(P, \tau_P)$ and $(Q, \tau_Q)$ tracial von Neumann algebras with a common tracial von Neumann subalgebra $(B, \tau_B)$ and set $N = P *_B Q$. We set $\tilde N = N *_B ( B \ovt L\mathbb F_2) $, and we let $\alpha \in \Aut(\tilde N)$ be the automorphism given by $\alpha = \Ad(u_a) * \Ad(u_b)$, where $\mathbb F_2 = \langle a, b \rangle$ and this free product automorphism is applied to the canonical decomposition 
\[
\tilde N \cong (P *_B (B \otimes L\langle a \rangle) ) *_B ( Q *_B( B \otimes L\langle b \rangle ) )
\] 
as in \cite[Section 2.2]{IoPePo08}.  

We consider the twisted $N$-$N$-bimodule structure on $L^2\tilde N$ given by $x \cdot \xi \cdot y = x \xi \alpha(y)$ for $x, y \in N$ and $\xi \in L^2 \tilde N$. To distinguish this from the usual bimodule structure we write $L^2 \tilde N_\alpha$ instead of $L^2 \tilde N$. We have a canonical embedding of $N$-$N$ Hilbert bimodules $L^2\langle N, P \rangle  \subset L^2\tilde N_\alpha$ taking $e_P$ to the vector $\widehat{u_{a^{-1}}}$. We similarly have an embedding $L^2\langle N, Q \rangle  \subset L^2\tilde N_\alpha$ taking $e_Q$ to the vector $\widehat{u_{b^{-1}}}$.

\begin{prop}\label{prop:afpwkcontain}
Using the notation above, we have weak containment of $N$-$N$ bimodules 
\[
L^2\tilde N_\alpha \ominus \left( L^2\langle N, P \rangle  \oplus L^2\langle N, Q \rangle  \right) \prec L^2 N \ovt L^2 N.
\]  
\end{prop}
\begin{proof}
An $N$-$N$ bimodule spanning set of vectors for $L^2 \tilde N$ is given by $\hat{1}$ and vectors of the form $u_{w_1} a_1 u_{w_2} a_2 \cdots a_{n -1} u_{w_n} a_n$ where $w_i \in \mathbb F_2 \setminus \{ e \}$ and $a_i \in N$ with $E_B(a_i) = 0$, together with vectors of the form $u_{w_1} a_1 u_{w_2} a_2 \cdots a_{n -1} u_{w_n}$ where $w_i \in \mathbb F_2 \setminus \{ e \}$ and $a_i \in N$ with $E_B(a_i) = 0$. In all of these cases except the two exemptions above, if $v$ is a vector of this form it follows that we have $\tau(x v \alpha(y) v^*) = \tau(E_B(x) E_B(y) ) \| v \|_2^2$, and hence this gives rise to a bimodule isomorphic to $L^2 \langle N, B \rangle$, which is weakly contained in $L^2N \ovt L^2N$ since $B$ is amenable. 
\end{proof}

\begin{lem}\label{lem:reltensorprojection}
Let $N$ be a tracial von Neumann algebra and $P \subset N$ a von Neumann subalgebra. Consider the $N$-bimodular isometric embedding $\pi: L^2\langle N, P \rangle \to L^2 \langle N, P \rangle \otimes_N L^2 \langle N, P \rangle$ satisfying $\pi( x e_P y ) = ( x e_P ) \otimes_N ( e_P y)$ for $x, y \in N$. If $A, B \in L^2 \langle N, P \rangle$, with $A$ left-bounded, then $\| A B \|_2 \leq \| L_A \| \| B \|_2 < \infty$ and we have $\pi^*( A \otimes_N B ) = AB$. 
\end{lem}
\begin{proof}
The span of vectors of the form $x e_P y$, $x, y \in N$ is dense in $L^2 \langle N, P \rangle$ and since this also forms an $N$-$N$ bimodule it must also be dense in the space of left-bounded vectors with respect to the norm $\| L_A \|$. Thus, it suffices to verify $\pi^*( A \otimes_N B) = AB$ for $A = x_1 e_P y_1$ and $B = x_2 e_P y_2$ with $x_1, x_2, y_1, y_2 \in N$. (The inequality on the $\| \cdot \|_2$ would then follow from general facts on relative tensor products of bimodules.) In this case the result follows easily from the computation that for $a, b \in N$ we have
\begin{align*}
\langle A \otimes_N B, (a e_P) \otimes_N (e_P b) \rangle
&= \langle B, \langle A, a e_P \rangle_N e_P b \rangle \\
&= \langle x_2 e_P y_2, y_1^* \mathbb E_P( x_1^* a) e_P b \rangle \\
&= \tau( \mathbb E_P(a^* x_1) \mathbb E_P( y_1 x_2 ) \mathbb E_P( y_2 b^*) ) \\
& = \langle AB, a e_P b \rangle. \qedhere
\end{align*}
\end{proof}

\begin{lem}\label{lem:diffusetensordiag}
For every $\varepsilon > 0$ there exists $\delta > 0$ so that if $N$ is a tracial von Neumann algebra, $P \subset N$ is a von Neumann subalgebra, and $A, B \in L^2\langle N, P \rangle$ with $\| x A \|_2, \| A x \|_2, \| x B \|_2 \leq \| x \|_2$ for $x \in N$ and $\sup_{a, b \in (N)_1 } | \langle A, a e_P b \rangle | < \delta$, then $\| A B \|_2 < \varepsilon$. 
\end{lem}
\begin{proof}
Let $P \subset N$ be an inclusion of tracial von Neumann algebras and suppose $A, B \in L^2\langle N, P \rangle$ satisfy the inequalities above. For $\kappa > 0$ we let $P_\kappa = 1_{[\kappa, \infty)}(| A |^2)$, which is a finite trace projection in $\langle N, P \rangle$ satisfying $\Tr(P_\kappa) \leq \kappa^{-1} \Tr( | A |^2 ) \leq \kappa^{-1}$. We let $\mathcal K \subset L^2 N$ denote the range of $P_\kappa$. This is a right Hilbert $P$ module and we may choose a Pimsner-Popa basis $\{ y_n \}_n \subset L^2 N$. Moreover, since $\Tr(P_\kappa)$ is finite, this basis is at most countable and we can choose such a basis so that setting $p_n = \mathbb E_P(y_n^* y_n) \in \mathcal P(P)$ we have $p_n \geq p_{n + 1}$ for each $n$. 

Since $\| L_A \| \leq 1$, for $x \in N$ we have
\[
\| x \|_2^2 = \| x^* \|_2^2 \geq \| A x^* \|_2^2 \geq \kappa \Tr( x P_\kappa x^* ) = \kappa \Tr( P_\kappa x^* x P_\kappa ) = \kappa \sum_n \| x y_n \|_2^2. 
\]
We therefore conclude that $\sum_n y_n^* y_n \in N \subset L^2N$ and we have the uniform bound $\| \sum_n y_n^* y_n \| \leq \kappa^{-1}$. In particular, there can be no more than $\kappa^{-1}$ elements in $\{ y_n \}_n$, since $p_n \geq p_{n + 1}$ for each $n$. 

We consider the polar decomposition $A = V | A |$ and we set $z_n = V(y_n)$. Since $\| R_A \| \leq 1$, a similar computation as above shows that for $x \in N$ we have $\| x \|_2^2 \geq \| x A \|_2^2 \geq \kappa \sum_n \| x z_n \|_2^2$. Hence, we also have $\sum_n z_n^* z_n \in N \subset L^2N$ and $\| \sum_n z_n^* z_n \|_2^2 \leq \kappa^{-1}$. 

If $\delta > 0$ is such that $\sup_{a, b \in (N)_1 } | \langle A, a e_P b \rangle | < \delta$, we then conclude that 
\[
\tau( p_1 ) = \| y_1 \|_2^2 \leq \kappa^{-{1/2} } \langle | A | y_1, y_1 \rangle = \kappa^{-1/2} \Tr (A y_1 e_P z_1^*) \leq  \kappa^{-3/2} \delta.
\]
Taking the supremum of the left supports of $y_n$, then gives a projection $p \in \mathcal P(N)$ with $\tau(p) \leq \kappa^{-1} \tau(p_1) \leq \kappa^{-5/2} \delta$ such that $p^{\perp} y_n = 0$ for each $n$, and hence $p^{\perp} P_\kappa = 0$. Thus, $\| p^\perp | A |^2 p^\perp \| \leq \kappa$ so that
\[
\| A B \|_2 \leq \| A p B \|_2 + \| A p^\perp B \|_2
\leq \| p B \|_2 + \kappa \| B \|_2
\leq \| p \|_2 + \delta = \kappa^{-3/4} \delta^{1/2} + \kappa. 
\]
Since $\delta$ and $\kappa$ are arbitrary, this then gives the result. 
\end{proof}

The following is analogous to \cite[Theorem 5.1]{IoPePo08}.

\begin{thm}\label{thm:afpintertwining}
Suppose we have families of II$_1$ factors $(P_\iota, \tau_{P_\iota} )$ and $(Q_\iota, \tau_{Q_\iota})$ and amenable von Neumann algebras $(B_\iota, \tau_{B_\iota})$ with tracial inclusions $B_\iota \subset P_\iota$ and $B_\iota \subset Q_\iota$. Let $\tilde N_\iota$ and $\alpha_\iota \in \Aut(\tilde N_\iota)$ be as described above. 

Suppose $M \subset \prod_{\iota \to \mathcal U}  N_\iota$ is a von Neumann subalgebra with $M' \cap \prod_{\iota \to \mathcal U}  N_\iota$ a factor such that we have an inclusion of $M$-$M$ bimodules $L^2 M \subset L^2 (\prod_{\iota \to \mathcal U} \tilde N_\iota)_{\alpha}$ with the action on the right given by multiplication by $\alpha(M)$ where $\alpha = (\alpha_\iota)^{\mathcal U}$. Then one of the following conditions holds:
\begin{enumerate}
\item\label{item:afpintA}  $M$ is matricially approximable with respect to $(N_\iota)_\iota$.
\item\label{item:afpintB}  There exists a unitary $u \in \mathcal U(N)$ so that  $uMu^* \subset \prod_{\iota \to \mathcal U}  P_\iota$.
\item\label{item:afpintC} There exists a unitary $u \in \mathcal U(N)$ so that $u Mu^* \subset \prod_{\iota \to \mathcal U}  Q_\iota$. 
\end{enumerate}
\end{thm}
\begin{proof}
Set 
\[
P = \prod_{\iota \to \mathcal U}  P_\iota, \ \ \ \ \ \ 
Q = \prod_{\iota \to \mathcal U}  Q_\iota, \ \ \ \ \ \ 
N = \prod_{\iota \to \mathcal U}  N_\iota, \ \ \ \ \ \ 
\tilde N = \prod_{\iota \to \mathcal U} \tilde N_\iota. \ \ \ \ \ 
\] 

Suppose $M$ is matricially inapproximable with respect to $(N_\iota)_\iota$ and $a$ is a nonzero $M$-central vector in $L^2(\tilde N)_\alpha$. By cutting with a spectral projection of $| a |$ we may assume that $y \in \tilde N$ and so we may write $a = (a_\iota)^{\mathcal U}$ with $\| a_\iota \| \leq \| a \|$. By Lemma~\ref{lem:ultradirectsum} and Proposition~\ref{prop:afpwkcontain}, we have an isomorphism of $M$-$M$ bimodules 
\[
L^2(\tilde N)_\alpha \cong \left( \prod_{\iota \to \mathcal U}  L^2\langle N_\iota, P_\iota \rangle \right) \oplus \left( \prod_{\iota \to \mathcal U} L^2\langle N_\iota, Q_\iota \rangle \right) \oplus \prod_{\iota \to \mathcal U} \mathcal H_\iota 
\] 
where $\mathcal H_\iota$ is an $N_\iota$-$N_\iota$ bimodule that is weakly contained in $L^2N_\iota \ovt L^2N_\iota$. Propositions~\ref{prop:weakcontainmentultra}, \ref{prop:relcommutant} and Theorem~\ref{thm:amenimplieshyperfinite} then show that the projection of $a$ onto the $\prod_{\iota \to \mathcal U} \mathcal H_\iota$ is zero, and thus the projection of $a$ onto either $\prod_{\iota \to \mathcal U} L^2\langle N_\iota, P_\iota \rangle$ or $\prod_{\iota \to \mathcal U} L^2\langle N_\iota, Q_\iota \rangle$ must be non-zero.

Suppose that the former conclusion holds. Let $A$ denote the projection of $a$ onto $\prod_{\iota \to \mathcal U}  L^2\langle N_\iota, P_\iota \rangle$ and note that $A$ is $N$-left-right-bounded by $\| y \|$. Hence, we may write $A = (A_\iota)^{\mathcal U}$ where $A_\iota \in L^2\langle N_\iota, P_\iota \rangle$ satisfy $\| x A_\iota \|_2, \| A_\iota x \|_2 \leq ( \| y \|  + 1 ) \| x \|_2$ for each $x \in N_\iota$.  If $A \perp L^2\langle N, P \rangle \subset \prod_{\iota \to \mathcal U} L^2\langle N_\iota, P_\iota \rangle$, then we must have $\lim_{\iota \to \mathcal U} \sup_{a, b \in (N_\iota)_1} | \langle A, a e_{P_\iota} b \rangle | = 0$ and it then follows from Lemmas~\ref{lem:reltensorprojection} and \ref{lem:diffusetensordiag} that $( A_\iota \otimes_N A_\iota^* )^{\mathcal U}$ gives a non-zero $M$-central vector in $\prod_{\iota \to \mathcal U} ( ( L^2\langle N_\iota, P_\iota \rangle  \otimes_N  L^2\langle N_\iota, P_\iota \rangle  ) \ominus \mathcal H_\iota )$, where $\mathcal H_\iota \subset L^2\langle N_\iota, P_\iota \rangle$ is the Hilbert $N_\iota$-$N_\iota$ sub-bimodule generated by the vector $e_{P_\iota} \otimes_{N_\iota} e_{P_\iota}$. However, as $M$ is matricially inapproximable with respect to $(N_\iota)_\iota$, this cannot occur by Proposition~\ref{prop:coarsebimoduletensorafp} and Theorem~\ref{thm:amenimplieshyperfinite}. It follows therefore that $A \not\perp L^2 \langle N, P \rangle$ and thus taking the projection of $A$ onto $L^2\langle N, P \rangle$ gives a non-zero $M$-central vector in $L^2 \langle N, P \rangle$, which shows $M \prec_N P$.

By Popa's Intertwining Theorem there then exist $e \in \mathcal P(M)$, $p \in \mathcal P(P)$, a normal unital $*$-homomorphism $\theta: eMe \to pPp$ and a vector $v \in eNp$ so that $xv = v \theta(x)$ for $x \in eMe$. Take representations $e =(e_\iota)^{\mathcal U}$ and $p = (p_\iota)^{\mathcal U}$ with $e_\iota \in \mathcal P(N_\iota)$, and $p_\iota \in \mathcal P(P_\iota)$. We also take representations $v = (v_\iota)^{\mathcal U}$ where $v_\iota \in e_\iota N_\iota p_\iota$ are partial isometries. Since $M$ is matricially inapproximable with respect to $(N_\iota)_\iota$ and since $M' \cap N$ is a factor, it follows from Proposition~\ref{prop:cutdownapprox} that $vv^* eMe$ is matricially inapproximable with respect to $(v_\iota v_\iota^* e_\iota N e_\iota v_\iota v_\iota^*)_\iota$. Conjugating by $v^*$ and again using Proposition~\ref{prop:cutdownapprox} then shows that $\theta(e M e)$ is matricially inapproximable with respect to $(p_\iota N p_\iota)_\iota$. Since $v^*v \in \theta(M)' \cap pNp$, we have from Proposition~\ref{prop:relcommutant} that $v^*v \in p P p$. The argument in \cite[Lemma F.18]{BrOz08} may then be applied to construct a unitary $u \in \mathcal U(N)$ such that $u M u^* \subset P$. 

If instead the projection of $a$ onto $\prod_{\iota \to \mathcal U} L^2\langle N_\iota, Q_\iota \rangle $ is non-zero, then we similarly deduce that $M \prec_N Q$, and from this we would similarly obtain a unitary $u \in \mathcal U(N)$ such that $u M u^* \subset Q$. 
\end{proof}

Note that if both (\ref{item:afpintB}) and (\ref{item:afpintC}) hold above, then by Proposition~\ref{prop:relcommutant} we have that (\ref{item:afpintA}) must also hold. 
Also, note that if in the previous theorem we have that $\lim_{\iota \to \mathcal U} \dim(P_\iota)$ and $\lim_{\iota \to \mathcal U} \dim(Q_\iota)$ are both finite, then neither (\ref{item:afpintB}) or (\ref{item:afpintC}) can hold and so, in this case, (\ref{item:afpintA}) is the only possible conclusion.

\section{Embeddings of property (T) factors into ultraproducts}\label{sec:embedultra}

We are now in position to analyze embeddings of property (T) factors into free products. Since the case of embeddings into free group factors is a bit more direct, we first give an argument for this case, and then we consider the general free product situation. 

\subsection{Property (T) subfactors and weakly coarse malleable deformations.}

Following Popa \cite{Po07}, we'll say that a II$_1$ factor $N$ has a weakly coarse malleable deformation if there exists a II$_1$ superfactor $N \subset \tilde N$ and there exists an automorphism $\alpha \in \Aut(\tilde N)$ that is in the connected component of the identity such that we have weak containment of the following $N$-$N$ Hilbert bimodules:
\begin{equation}\label{eq:wkcontain}
L^2 \tilde N \ominus L^2 N \prec L^2 N \ovt L^2 N, \ \ \ \ \ \ \ \
L^2 \tilde N_{\alpha} \prec L^2 N \ovt L^2 N. 
\end{equation}

\begin{lem}\label{lem:coarseautomorphism}
Suppose $\Sigma < \Gamma$ and $\alpha \in \Aut(\Gamma)$ is such that $t \alpha(\Sigma) t^{-1} \cap \Sigma$ is amenable for each $t \in \Gamma$. Let $M = L\Sigma$ and $N = L\Gamma$ and let $\tilde \alpha \in \Aut(L\Gamma)$ be the corresponding automorphism given by $\tilde \alpha(\lambda_t) = \lambda_{\alpha(t)}$. Then the $M$-$M$ bimodule $L^2 N$ given by $x \cdot \xi \cdot y = x \xi \tilde \alpha(y)$ is weakly contained in the coarse bimodule. 
\end{lem}
\begin{proof}
If $x \in \Gamma$ and $s, t \in \Sigma$, then $\langle \lambda_s \widehat{ \lambda_x } \tilde \alpha(\lambda_t^*), \widehat{\lambda_x} \rangle$ is zero unless $s = x^{-1} \alpha(t) x$, so that $s \in \Sigma \cap x \alpha(\Sigma) x^{-1}$. Thus, we see that the $L\Sigma$-bimodule spanned by $\widehat{\lambda_x}$ is isomorphic to a bimodule induced from a $L (\Sigma \cap x \alpha(\Sigma) x^{-1} )$-$L\Sigma$ bimodule. Since $ \Sigma \cap x \alpha(\Sigma) x^{-1}$ is amenable, this bimodule is weakly contained in the coarse bimodule. Since such vectors span a dense subspace of $L^2 N$, the result follows. 
\end{proof}

\begin{examp}
Suppose $\mathbb F_n$ is the free group on $n$ generators $\{ a_k \}_k$. We identify $\mathbb F_{2n}$ as the free group with generators $\{ a_k, b_k \}_k$ so that we have a canonical embedding $\mathbb F_n < \mathbb F_{2n}$. We let $\alpha_0 \in \Aut(\mathbb F_{2n})$ be the automorphism satisfying $\alpha_0(a_k) = a_k b_k$, $\alpha_0(b_k) = b_k$, $k \geq 1$. We let $\alpha \in \Aut(L\mathbb F_{2n})$ be the corresponding automorphism. This automorphism is in the connected component of the identity \cite[Lemma 2.1]{Po07}, and it follows that for any $t \in \mathbb F_{2n}$ we have $t \mathbb F_n t^{-1} \cap \alpha(\mathbb F_n) = \{ e \}$ and so by Lemma~\ref{lem:coarseautomorphism} the $L\mathbb F_n$-$L\mathbb F_n$ bimodule $L^2( L\mathbb F_{2n})$ given by $x \cdot \xi \cdot y = x \xi \alpha(y)$ is weakly contained in the coarse bimodule, so that we have the weak containments in (\ref{eq:wkcontain}). 
\end{examp}

\begin{examp}
Generalizing the previous example, suppose $(N_n)_{n}$ is a sequence (finite or infinite) of separable amenable tracial von Neumann algebras and we have trace-preserving embeddings $B \subset N_n$ where $B$ is a tracial amenable von Neumann algebra. We set $N = N_1 *_{B} N_2 *_{B} \cdots$ and suppose $N$ is a II$_1$ factor. We let $M_n =  (B \ovt L\mathbb Z) *_{B} N_n$ and take $\tilde N = M_1 *_{B} M_2 *_{B} \cdots$. 

We let $a_n$ be the generator of the copy of $\mathbb Z$ in $M_n$, and we let $\alpha \in \Aut(\tilde N)$ be given by the freeproduct automorphism $\alpha = \Ad(\lambda(a_1)) * \Ad(\lambda(a_2)) * \cdots$. By \cite[Lemma 2.4]{IoPePo08} it follows that $\alpha$ is in the connected component of the identity. Since each $N_n$ is amenable, it follows from Proposition~\ref{prop:coarsebimoduleafp} that the bimodules appearing in (\ref{eq:wkcontain}) are weakly contained in the coarse bimodule. 
\end{examp}

\begin{examp}
If $M_1$ and $M_2$ have weakly coarse malleable deformations, then, taking free products of the deformations, it follows that $M_1 * M_2$ also has a weakly coarse malleable deformation. 
\end{examp}

Theorem~\ref{thm:main} is a consequence of the following more general result. 

\begin{thm}\label{thm:weaklycoarsemalleable}
Suppose $N_\iota$ are II$_1$ factors having a weakly coarse malleable deformation. If $M \subset \prod_{\iota \to \mathcal U} N_\iota$ is a property (T) subfactor with factorial relative commutant, then $M$ is matricially approximable with respect to $(N_\iota)_\iota$.
\end{thm}
\begin{proof}
We take superfactors $N_\iota \subset \tilde N_\iota$ and automorphisms $\alpha_\iota \in \Aut( \tilde N_\iota)$ in the connected component of the identity such that we have weak containment of bimodules as in (\ref{eq:wkcontain}). We let $\alpha = (\alpha_\iota)^{\mathcal U} \in \Aut\left( \prod_{\iota \to \mathcal U} \tilde N_\iota \right)$.

If $M' \cap \prod_{\iota \to \mathcal U} \tilde N_\iota \not= M' \cap \prod_{\iota \to \mathcal U} N_\iota$, then by Proposition~\ref{prop:weakcontainmentultra} we have inclusions of $M$-$M$ Hilbert bimodules 
\[
L^2 M \subset \prod_{\iota \to \mathcal U} ( L^2(\tilde N_\iota) \ominus L^2(N_\iota) ) \prec \prod_{\iota \to \mathcal U} (L^2 N_\iota \ovt L^2 N_\iota).
\]

Otherwise, if $M' \cap \prod_{\iota \to \mathcal U} \tilde N_\iota = M' \cap \prod_{\iota \to \mathcal U} N_\iota$ is a factor, then by Theorem~\ref{thm:propTultra} there exists $u \in \mathcal U( \prod_{\iota \to \mathcal U} \tilde N_\iota )$ so that $\alpha(x) = ux u^*$ for $x \in M \subset \prod_{\iota \to \mathcal U} N_\iota$. From (\ref{eq:wkcontain}) and Proposition~\ref{prop:weakcontainmentultra} we then also have weak containment of $M$-$M$ bimodules $L^2M \prec \prod_{\iota \to \mathcal U} (L^2 N_\iota \ovt L^2 N_\iota)$.

Thus, in either case we have $L^2M \prec \prod_{\iota \to \mathcal U} (L^2 N_\iota \ovt L^2 N_\iota)$, and so the result then follows from Theorem~\ref{thm:amenimplieshyperfinite}.  
\end{proof}

\begin{rem}\label{rem:interp}
Theorem~\ref{thm:main} also applies in the setting of interpolated free group factors $L\mathbb F_t$ for $t > 1$. Indeed, we can just consider $M \subset (L \mathbb F_t)^{\mathcal U} \subset (L \mathbb F_t * L\mathbb F_\infty)^{\mathcal U} \cong (L\mathbb F_\infty)^{\mathcal U}$. If $M' \cap (L \mathbb F_t)^{\mathcal U}$ is a factor, then by Proposition~\ref{prop:relcommutant} we reduce to the case when $M' \cap (L \mathbb F_\infty)^{\mathcal U} = M'  \cap (L \mathbb F_t)^{\mathcal U}$ is also a factor so that we may apply Theorem~\ref{thm:main} to conclude that $M$ is matricially approximable with respect to $L \mathbb F_\infty$, say $M \subset \prod_{\iota \to \mathcal U} A_\iota \subset (L \mathbb F_\infty)^{\mathcal U}$ where $A_\iota \subset L \mathbb F_\infty$ are finite-dimensional subfactors. 

But then we may take $u \in \mathcal U ( (L \mathbb F_\infty)^{\mathcal U} )$ so that $u \left( \prod_{\iota \to \mathcal U} A_\iota  \right) u^* \subset (L \mathbb F_t)^{\mathcal U}$. By Remark~\ref{rem:normalizer} we would have that either the conclusion of the theorem holds or else $u \in \mathcal U( (L\mathbb F_t )^{\mathcal U} )$, in which case the conclusion of the theorem also holds. 

A similar argument can be used even if the parameter $t$ is allowed to vary with $\iota \in I$ (even if $\lim_{\iota \to \mathcal U} t_\iota = 1$). In particular, it follows that if $M$ is a property (T) factor and $M \subset (L \mathbb F_t)^{\mathcal U}$, then the center of $M' \cap (L\mathbb F_t)^{\mathcal U}$ cannot have a minimal projection, since otherwise we could restrict to the corresponding corner, which would be of the form $\prod_{\iota \to \mathcal U} L \mathbb F_{t_\iota}$, for $t_\iota > 1$. 
\end{rem}

\begin{proof}[Proof of Theorem~\ref{thm:mainfg}]
If $M$ is a property (T) factor and $\theta: M^{\mathcal U} \to (L\mathbb F_n)^{\mathcal V}$ were an isomorphism, then as $M$ is full we have $\theta(M)' \cap (L\mathbb F_n)^{\mathcal V} = \mathbb C$ and so the hypotheses of Theorem~\ref{thm:main} are satisfied. There must then exist finite-dimensional subfactors $A_\iota \subset L\mathbb F_n$ so that $\theta(M) \subset \prod_{\iota \to \mathcal V} A_\iota \subset (L\mathbb F_n)^{\mathcal V}$. But then $\theta(M)' \cap (L\mathbb F_n)^{\mathcal V}$ would contain the non-separable subfactor $\prod_{\iota \to \mathcal V} (A_\iota' \cap L\mathbb F_n)$, giving a contradiction. 
\end{proof}

\subsection{Property (T) subfactors and tensor products with weakly coarse malleable deformations}

We show here how to generalize Theorem~\ref{thm:weaklycoarsemalleable} to account for tensor product of factors with weakly coarse malleable deformations. The techniques here are analogous to those found in \cite{OzPo04, Pe09, ChSi13}.

\begin{thm}\label{thm:weaklycoarsemalleabletensor}
Suppose $1 \leq k < \infty$ and $N_\iota^1, N_\iota^2, \ldots, N_\iota^k$ are II$_1$ factors, each of which has a weakly coarse malleable deformation. Set $N_\iota = N_\iota^1 \ovt N_\iota^2 \ovt \cdots \ovt N_\iota^k$. If $M \subset \prod_{\iota \to \mathcal U} N_\iota$ is a property (T) subfactor with factorial relative commutant, then $M$ is matricially approximable with respect to $(N_\iota)_\iota$.
\end{thm}
\begin{proof}
For $1 \leq j \leq k$, we suppose $N_\iota^j \subset \tilde N_\iota^j$ and $\alpha_\iota^j \in \Aut(\tilde N_\iota^j)$ are automorphisms in the connected component of the identity so that we have weak containment of $N_\iota^j$-$N_\iota^j$ bimodules 
\[
L^2 \tilde N_\iota^j \ominus L^2 N_\iota^j \prec L^2 N_\iota^j \ovt L^2 N_\iota^j, \ \ \ \ \ L^2 ( \tilde N_\iota^j)_{\alpha_\iota^j} \prec L^2 N_\iota^j \ovt L^2 N_\iota^j.
\]
We use the following notation: 
\[
Q_\iota^j = N^1_\iota \ovt \cdots \ovt N^{j - 1}_\iota \ovt \tilde N_\iota^j \ovt N_\iota^{j + 1} \ovt \cdots \ovt N_\iota^k,
\]
\[
Q_\iota(j) = \tilde N^1_\iota \ovt \cdots \ovt \tilde N^{j - 1}_\iota \ovt  N_\iota^j \ovt \tilde N_\iota^{j + 1} \ovt \cdots \ovt \tilde N_\iota^k, 
\]
\[
\tilde N_\iota = \tilde N_\iota^1 \ovt \cdots \ovt \tilde N_\iota^k.
\] 

Let $E \subset \{ 1, 2, \ldots, k \}$ denote the set of indices $j$ such that we have 
\begin{equation}\label{eq:relcom}
M' \cap \prod_{\iota \in \mathcal U} Q_\iota^j
\not= M' \cap \prod_{\iota \to \mathcal U} N_\iota.
\end{equation}
By applying a rearrangement, we will assume that $E = \{ 1, 2, \ldots, k_0 \}$ for some $0 \leq k_0 \leq k$. For each $1 \leq j \leq k_0$ there exists a nonzero element $u_j \in M' \cap \prod_{\iota \to \mathcal U} Q_\iota^j \subset \prod_{\iota \to \mathcal U} \tilde N_\iota$ with $\mathbb E_{\prod_{\iota \to \mathcal U} N_\iota}(u_j) = 0$. Since $M' \cap \prod_{\iota \to \mathcal U} N_\iota$ is a factor, we may multiply each $u_j$ on the left and right by elements in $M' \cap \prod_{\iota \to \mathcal U} N_\iota$ so that we may assume $u_1u_2 \cdots u_{k_0} \not= 0$ (take this product to be $1$ if $k_0 = 0$).

For $k_0 < j \leq k$ we set $\alpha^j = ( {\rm id} \otimes \alpha_\iota^j \otimes {\rm id} )^{\mathcal U} \in \Aut( \prod_{\iota \to \mathcal U} Q_\iota^j)$ and we see that Theorem~\ref{thm:propTultra} applies, giving a unitary $u_j \in \mathcal U( \prod_{\iota \to \mathcal U} Q_\iota^j )$ that satisfies $x u = u \alpha^j(x)$ for $x \in M$. Setting 
\[
u = u_1 u_2 \cdots u_{k_0 + 1} \alpha_{k_0 + 1}( u_{k_0 + 2} \alpha_{k_0 + 2}( \cdots \alpha_{k - 1}( u_k ) \cdots )) \in \prod_{\iota \to \mathcal U} \tilde N_\iota,
\] 
we then have that $u \not= 0$, and if $x \in M$, then $x u = u \alpha_{k_0+ 1} \circ \alpha_{k_0 + 2} \circ \cdots \circ \alpha_k(x)$.

If $1 \leq j \leq k_0$, then from the definition of $u$ above we see that $u$ may be written as $u = a u_j b$ where $a, b \in \prod_{\iota \to \mathcal U} Q_\iota(j)$. We therefore have 
\[
\mathbb E_{\prod_{\iota \to \mathcal U} Q_\iota(j)}(u) 
= a \mathbb E_{\prod_{\iota \to \mathcal U} Q_\iota(j)}(u_j) b 
= a \mathbb E_{\prod_{\iota \to \mathcal U} Q_\iota(j)} \circ \mathbb E_{\prod_{\iota \to \mathcal U} Q_\iota^j}(u_j) b
= a \mathbb E_{\prod_{\iota \to \mathcal U} N_\iota}(u_j) b
 = 0.
\]

Thus, we may represent $u$ as $u = (u_\iota)^{\mathcal U}$ where $\mathbb E_{Q_\iota(j)}(u) = 0$ for each $1 \leq j \leq k_0$, i.e., 
\[
u_\iota \in \mathcal H_\iota := (L^2 \tilde N_1 \ominus L^2 N_1) \ovt (L^2 \tilde N_2 \ominus L^2 N_2) \ovt \cdots \ovt (L^2 \tilde N_{k_0} \ominus L^2 N_{k_0}) \ovt L^2 \tilde N_{k_0 + 1} \ovt \cdots \ovt L^2 \tilde N_k.
\]
The twisted $N_\iota$-$N_\iota$ bimodule structure on each $\mathcal H_\iota$ given by  $a \xi b = a \xi \alpha_{k_0 + 1} \circ \alpha_{k_0 + 2} \circ \cdots \circ \alpha_k( b )$ is weakly contained in the coarse bimodule, and as $u$ represents an $M$-central element in the $\prod_{\iota \to \mathcal U} N_\iota$-$\prod_{\iota \to \mathcal U} N_\iota$ Hilbert bimodule $\prod_{\iota \to \mathcal U} \mathcal H_\iota$, it then follows from Theorem~\ref{thm:amenimplieshyperfinite} that $M$ is matricially approximable with respect to $(N_\iota)_{\iota \in I}$.
\end{proof}

\subsection{Non-pseudomatricial full factors}

The following proposition was noted by Popa (see the remark before Corollary 6.11 in \cite{Br11}) but has not yet appeared in the literature in this generality. We include a proof for completeness. The roots of the argument go back to \cite{Wa91}, \cite{Vo90}, or perhaps even \cite{Wa75}.

\begin{prop}\label{prop:irredpseudocpt}
Let $\mathcal U$ be an ultrafilter on a set $I$ and suppose $n_\iota \in \mathbb N$ are such that $\lim_{\iota \to \mathcal U} n_\iota = \infty$, then there exists an irreducible embedding of $L( \SL_3(\mathbb Z) )$ into $\prod_{\iota \to \mathcal U} \mathbb M_{n_\iota}(\mathbb C)$. 
\end{prop}
\begin{proof}
Given $n \in \mathbb N$, we let $k_n$ denote the largest integer $k$ so that $k \leq n$ and such that there exists an irreducible representation $\rho_k: \SL_3(\mathbb Z) \to U(k)$. We then define the representation $\pi_n: \SL_3(\mathbb Z) \to U(n)$ by $\pi_n(t) = \rho_k \oplus \operatorname{id}_{n-k}$. We may then consider the representation $\pi: \SL_3(\mathbb Z) \to \mathcal U( \prod_{\iota \to \mathcal U} \mathbb M_{n_\iota}(\mathbb C) )$ given by $\pi = (\pi_{n_\iota})^{\mathcal U}$. 

For each prime $p$ we have a canonical surjection $\SL_3(\mathbb Z) \to \SL_3(\mathbb Z/p\mathbb Z)$, and hence any irreducible unitary representation of $\SL_3(\mathbb Z/p\mathbb Z)$ gives an irreducible unitary representation of $\SL_3(\mathbb Z)$. In particular, there is the Steinberg Representation of $\SL_3(\mathbb Z/p\mathbb Z)$, which is an irreducible representation of dimension $p^3$. By the Prime Number Theorem, it then follows that given any $\varepsilon > 0$, and large enough $n$, there exists a finite-dimensional irreducible representation of $\SL_3(\mathbb Z)$ into $U(k)$ where $(1-\varepsilon)n < k \leq n$. From property (T), we then conclude that 
\[
\pi( \SL_3(\mathbb Z) )' \cap \prod_{\iota \to \mathcal U} \mathbb M_{n_\iota} (\mathbb C) = \prod_{\iota \to \mathcal U} ( \pi_{n_\iota}(\SL_3(\mathbb Z) )' \cap \mathbb M_{n_\iota}(\mathbb C) ) = \mathbb C.
\]
Hence, $\pi( \SL_3(\mathbb Z))''$ is an irreducible subfactor of $\prod_{\iota \to \mathcal U} \mathbb M_{n_\iota} (\mathbb C)$ with property (T) \cite{CoJo85}. By \cite{Be07} we have that, in fact, $\pi( \SL_3(\mathbb Z))'' \cong L( \SL_3(\mathbb Z) )$. 
\end{proof}

\begin{proof}[Proof of Theorem~\ref{thm:freenotpseudomatricial}]
This follows from Proposition~\ref{prop:irredpseudocpt} and Theorem~\ref{thm:main}.
\end{proof}

\begin{rem}
From Remark~\ref{rem:interp}, we see that even if $t_\iota > 1$ with $\lim_{\iota \to \mathcal U} t_\iota = 1$, then $\prod_{\iota \to \mathcal U} L \mathbb F_t$ cannot be pseudomatricial. 
\end{rem}

In fact, fixing the finitely presented group $\Gamma = \SL_3(\mathbb Z)$, the statement that ``every homomorphism of $\Gamma$ into the unitary group of an ultraproduct must have diffuse relative commutant'' is easily expressible as an $\forall \exists$-sentence (we thank David Jekel for showing this to us). Thus, we obtain the following result, which answers \cite[Question 5.4]{GoPi25}. 

\begin{thm}
If $\lim_{\iota \to \mathcal U} n_\iota = \infty$, then 
\[
\operatorname{Th}_{\forall \exists} \left( \prod_{\iota \to \mathcal U} \mathbb M_{n_\iota}(\mathbb C) \right) \not\leq \operatorname{Th}_{\forall \exists}(L \mathbb F_2).
\]
Consequently, by \cite[Lemma 5.3]{GoPi25} we have
\[
\operatorname{Th}_{\forall \exists} \left( \left( \prod_{\iota \to \mathcal U} \mathbb M_{n_\iota}(\mathbb C) \right) * L\mathbb Z \right) \not= \operatorname{Th}_{\forall \exists} \left( \prod_{\iota \to \mathcal U} (\mathbb M_{n_\iota}(\mathbb C) * L \mathbb Z ) \right).
\]
\end{thm}

\subsection{A Bass-Serre type strong rigidity for ultraproducts}

If $M$ is a tracial von Neumann algebra with $B \subset M$ a von Neumann subalgebra, then we have a canonical inclusion of $M^{\mathcal U}$ Hilbert bimodules $L^2\langle M^{\mathcal U}, B^{\mathcal U} \rangle \subset (L^2 \langle M, B \rangle )^{\mathcal U}$ taking $(x_\iota)^{\mathcal U} e_{B^{\mathcal U}} (y_\iota)^{\mathcal U}$ to $( x_\iota e_B y_\iota )^{\mathcal U}$, where $(x_\iota)^{\mathcal U}, (y_\iota)^{\mathcal U} \in M^{\mathcal U}$.

\begin{thm}\label{thm:ultrafreeembedding}
Suppose we have families of II$_1$ factors $(P_\iota, \tau_{P_\iota} )$, $(Q_\iota, \tau_{Q_\iota})$, and amenable von Neumann algebras $(B_\iota, \tau_{B_\iota})$ with tracial inclusions $B_\iota \subset P_\iota$ and $B_\iota \subset Q_\iota$. 

Suppose $M \subset \prod_{\iota \to \mathcal U}  N_\iota$ is a property (T) subfactor such that $M' \cap \prod_{\iota \to \mathcal U}  N_\iota$ a factor, then one of the following conditions holds:
\begin{enumerate}
\item\label{item:afpintnewA}  $M$ is matricially approximable with respect to $(N_\iota)_\iota$.
\item\label{item:afpintnewB}  There exists a unitary $u \in \mathcal U(N)$ so that  $uMu^* \subset \prod_{\iota \to \mathcal U}  P_\iota$.
\item\label{item:afpintnewC} There exists a unitary $u \in \mathcal U(N)$ so that $u Mu^* \subset \prod_{\iota \to \mathcal U}  Q_\iota$. 
\end{enumerate}
\end{thm}
\begin{proof}
Just combine Theorems~\ref{thm:propTultra} and \ref{thm:afpintertwining}. 
\end{proof}

A II$_1$ factor $M$ is asymptotically commutative \cite{Sa68} if for any $\varepsilon > 0$ and finite set $F \subset M$ there exists an automorphism $\alpha \in \Aut(M)$ so that $\sum_{x, y \in F} \| [ x, \alpha(y) ] \|_2 < \varepsilon$. An example of such a factor is one of the form $ M^{{ {\overline \otimes} {\mathbb N} }} $ for $M$ any finite factor. Another example of such a factor is one of the form $M^{{ {\overline \otimes} {\mathbb N} }} \rtimes S_{\mathbb N}$, where $M$ is a finite factor and $S_{\mathbb N}$ is the group of finite permutations on $\mathbb N$, which acts on $M^{{ {\overline \otimes} {\mathbb N} }}$ in the obvious way. This latter class is even inner asymptotically commutative \cite{ZM69}, meaning that we can take the automorphisms above to be inner. Note that if $M$ is a nonamenable II$_1$ factor, then $ M^{{ {\overline \otimes} {\mathbb N} }}$ is not inner asymptotically commutative so, in this case, $ M^{{ {\overline \otimes} {\mathbb N} }} \not\equiv  N^{{ {\overline \otimes} {\mathbb N} }} \rtimes S_{\mathbb N}$ by \cite[Proposition 4.8]{GoHa17}.

We let $\mathcal {C}_T$ denote the class of separable II$_1$ factors $N$ such that the following conditions hold:
\begin{enumerate}
\item There exists a property (T) II$_1$ subfactor $M \subset N$ with $\mathcal Z( M' \cap N ) = \mathbb C$,
\item\label{item:BCT} There exists a II$_1$ factor $P$ and a sequence of embeddings $\alpha_n: P \to N$ so that the family $(\alpha_n)_n$ has uniform weak spectral gap, and such that for every $x \in N$ we have $\| \mathbb E_{\alpha_n(P)' \cap N}(x ) - x \|_2 \to 0$. 
\end{enumerate}
The second condition shows that any factor in $\mathcal C_T$ has the McDuff property. Note that if $N$ is asymptotically commutative and contains a property (T) type II$_1$ subfactor $M \subset N$ with $\mathcal Z( M' \cap N) = \mathbb C$, then the last condition is satisfied by taking $P = M$. In particular, if $N$ contains a property (T) type II$_1$ subfactor with factorial relative commutant, then $N^{\overline \otimes \mathbb N}$ and $N^{\overline \otimes \mathbb N} \rtimes S_{\mathbb N}$ both belong to $\mathcal C_T$. Also, if $N$ is an arbitrary full II$_1$ factor (e.g., $N = L \mathbb F_2$) and $M$ is a property (T) II$_1$ factor (or even if $M$ just contains a property (T) subfactor with factorial relative commutant), then $N^{\overline \otimes \mathbb N} \ovt M$ is contained in $\mathcal C_T$.

We remark that (\ref{item:BCT}) implies that if $B \subset N^{\mathcal U}$ is any separable von Neumann subalgebra, then there exist automorphisms $\alpha_\iota \in \Aut(N)$ so that setting $\alpha = (\alpha_\iota)^{\mathcal U}$ we have 
\[
B \subset \alpha( P)' \cap N^{\mathcal U} = \prod_{\iota \to \mathcal U} ( \alpha_\iota(P)' \cap N).
\] 

Theorem~\ref{thm:pairwisenonconj} is easily seen to be a special case of the following theorem, which is the ultraproduct analog of \cite[Theorem 0.2]{IoPePo08}.

\begin{thm}\label{thm:bassserreee}
Suppose $N_0$, and $\tilde N_0$ are separable and have weakly coarse malleable deformations (with possibly $N_0 = \tilde N_0 = \mathbb C$), suppose $N_1, N_2, \ldots, N_k$, $\tilde N_1, \tilde N_2, \ldots, \tilde N_\ell \in \mathcal {C}_T$ with $k, \ell \geq 2$ and suppose for each $0 \leq i < k$ (resp.\ $0 \leq j < \ell$) we have a tracial amenable von Neumann algebra $A_i$ (resp.\ $\tilde A_i$) and trace-preserving embeddings of $A_i$ into $N_i$ and $N_{i + 1}$ (resp.\ $\tilde A_j$ into $\tilde N_j$ and $\tilde N_{j + 1}$). If 
\[
N_0 *_{A_0} N_1 *_{A_1} N_2 *_{A_2} \cdots *_{A_{k-1}} N_k \equiv \tilde N_0 *_{A_0} \tilde N_1 *_{\tilde A_1} \tilde N_2 *_{\tilde A_2} \cdots *_{\tilde A_{k-1}} \tilde N_\ell,
\] 
then $\ell = k$ and after a permutation of indices we have $N_i \equiv \tilde N_i$ for $1 \leq i \leq k$. 
\end{thm}
\begin{proof}
Set $N = N_0 *_{A_0} N_1 *_{A_1}  \cdots *_{A_{k-1}} N_k$ and $\tilde N = \tilde N_0 *_{A_0} \tilde N_1 *_{\tilde A_1} \cdots *_{\tilde A_{k-1}} \tilde N_\ell$. For $1 \leq i \leq k$ (resp.\ $1 \leq j \leq \ell$) we take $M_i \subset N_i$ (resp.\ $\tilde M_j \subset \tilde N_j$) to be a type II$_1$ property (T) subfactor with factorial relative commutant. We also take $P_i \subset N_i$ (resp.\ $\tilde P_j \subset \tilde N_j$), satisfying (\ref{item:BCT}). We suppose $\theta: N^{\mathcal U} \to \tilde N^{\mathcal V}$ is an isomorphism. We may assume that $\mathcal U$ and $\mathcal V$ are countably incomplete. 

After possibly interchanging $N$ and $\tilde N$, we may assume from  Theorem~\ref{thm:genspectralgapapprox} that whenever $1 \leq i \leq k$ and $\alpha_\iota: P_i \to N_i$ are embeddings such that $(\alpha_\iota)_\iota$ has uniform weak spectral gap, then letting $\alpha: P_i \to N_i^{\mathcal U}$ be the embedding defined by $\alpha(x) = (\alpha_\iota(x) )^{\mathcal U}$ we have that  $\theta ( \alpha(P_i))$ is matricially inapproximable with respect to $\tilde N$. For the same reason, we may also assume that $\theta(M_i)$ is matricially inapproximable with respect to $\tilde N$. 

By Theorem~\ref{thm:ultrafreeembedding}, for each $i \in \{ 1, 2, \ldots, k \}$ there exists some $j_i \in \{ 0, 1, 2, \ldots, \ell \}$ and $u_i \in \mathcal U(\tilde N^{\mathcal V})$ so that $u_i \theta( M_i ) u_i^* \subset \tilde N_{j_i}^{\mathcal V}$. Note that we cannot have $j_i = 0$ by Theorem~\ref{thm:weaklycoarsemalleable}.

If $B \subset N_i^{\mathcal U}$ is any separable von Neumann subalgebra that contains $M_i$, then from (\ref{item:BCT}) there exist $\alpha_\iota \in \Aut(N_1)$ so that setting $\alpha = (\alpha_\iota)^{\mathcal U}$ as above we have $B \subset \alpha(P_i)' \cap N_i^{\mathcal U}$. Since $M_i \subset B$, we then have from Proposition~\ref{prop:relcommutant} that 
\[
u_i \theta(  \alpha(P_i)  ) u_i^* \subset u_i \theta( M_i' \cap N^{\mathcal U}) u_i^* \subset  \tilde N_{j_i}^{\mathcal V}.
\]
A second application of Proposition~\ref{prop:relcommutant} then shows us that
\[
u_i \theta (B ) u_i^* \subset u_i \theta ( \alpha(P_i)' \cap N^{\mathcal U}) \subset \tilde N_{j_i}^{\mathcal V}.
\]
Since $B \subset N_i^{\mathcal U}$ was an arbitrary separable von Neumann subalgebra containing $M_i$, it then follows that 
\[
 \theta( N_i^{\mathcal U} ) \subset  u_i^* \tilde N_{j_i}^{\mathcal V} u_i.
\]

We now fix some $j_0 \in \{ 1, 2, \ldots, \ell \}$ so that $j_0 \not= j_i$ for some $i \in \{ 1, 2, \ldots, k \}$. If $\beta_\kappa \in \Aut(\tilde N_{j_0})$ and we set $\beta = (\beta_\kappa)^{\mathcal V}$, then by Proposition~\ref{prop:ultrabackandforth} we have that $\theta^{-1}( \beta(\tilde P_{j_0} ) )$ is matricially inapproximable with respect to $N$. For the same reason, we see that $\theta^{-1}( \tilde M_{j_0} )$ is also matricially inapproximable with respect to $N$. Thus, from the same argument in the previous paragraph we may find $i \in \{ 1, 2, \ldots, k \}$ and $v \in \mathcal U( N^{\mathcal U})$ so that $\theta^{-1}(\tilde N_{j_0}^{\mathcal U}) \subset v^* N_i^{\mathcal U} v$. We then have 
\[
\tilde N_{j_0}^{\mathcal U} \subset \theta(v^*) \theta(N_i^{\mathcal U}) \theta(v) \subset \theta(v^*) u_i^* \tilde N_{j_i}^{\mathcal V} u_i \theta(v)
\]
and so by Remark~\ref{rem:normalizer} we have $j_0 = j_i$ and $u_i \theta(v) \in \mathcal U(\tilde N_{j_0}^{\mathcal U}) $. In particular, this also shows that $j_0$ is in the range of the map $\{ 1, 2, \ldots, k \} \ni i \mapsto j_i$, and so we can apply the same result for any $j \not= j_0$ as well. We have thus concluded that for each $j \in \{ 1, 2, \ldots, \ell \}$ the von Neumann subalgebra $\tilde N_j^{\mathcal V}$ is conjugate in $\tilde N^{\mathcal V}$ to a von Neumann subalgebra of the form $\theta( N_i^{\mathcal U} )$ for $i \not= 0$. Because of Remark~\ref{rem:normalizer}, there can be at most one such subalgebra $\theta(N_i^{\mathcal U})$, and for the same reason we see that the map $\{ 1, 2, \ldots, k \} \ni i \mapsto j_i$ must be injective so that $k = \ell$ and after permutation of indices we have each $\tilde N_i^{\mathcal V}$ is conjugate to $\theta(N_i^{\mathcal U})$ so that $N_i \equiv \tilde N_i$ for each $1 \leq i \leq k$. 
\end{proof}

\subsection{Elementary equivalence and finite-index subfactors}

It is an open problem whether or not there exists a II$_1$ factor $M$ such that $M \not\equiv M^t$ for some $t \not= 1$ (see, e.g.,  \cite[Question 4.9]{GoHa23}). Here we show how to adapt the proof of Theorem~\ref{thm:bassserreee} to answer a related problem by giving an example of a II$_1$ factor having a non-elementarily equivalent finite-index subfactor. As the proof of this result follows from the same ideas as above, we will favor brevity, leaving the reader to verify some of the details. 

\begin{thm}\label{thm:finotee}
If $N \in \mathcal C_T$, then the factors $N * N$ and $(N * N) \rtimes \mathbb Z/2\mathbb Z$ are not elementarily equivalent, where $\mathbb Z/2 \mathbb Z$ acts on $N * N$ as the free flip.
\end{thm}
\begin{proof}
We denote the first copy of $N$ in $N * N$ by $N_1$ and the second copy by $N_2$. We also let $\tilde N_1$ and $\tilde N_2$ denote the corresponding copies in $(N * N) \rtimes \mathbb Z/2\mathbb Z$. We suppose, by way of contradiction, that we have an isomorphism 
\[
\theta: (N_1 * N_2)^{\mathcal U} \to ( (\tilde N_1 * \tilde N_2) \rtimes \mathbb Z/2\mathbb Z)^{\mathcal V} \cong  ( \tilde N_1 * \tilde N_2)^{\mathcal V} \rtimes \mathbb Z/2\mathbb Z.
\] 
We let $M_i \subset N_i$ (resp.\ $\tilde M_i \subset \tilde N_i$) be property (T) subfactors with factorial relative commutant and take $P_i \subset N_i$ (resp.\ $\tilde P_i \subset \tilde N_i$) $i \in \{ 1, 2 \}$, satisfying condition (\ref{item:BCT}) from the definition of the class $\mathcal C_T$.

We assume that whenever $i \in \{ 1, 2 \}$ and $\alpha_\iota: P_i \to N_i$ are embeddings such that $(\alpha_\iota)_\iota$ has uniform weak spectral gap, then letting $\alpha: P_i \to N_i^{\mathcal U}$ be the embedding defined by $\alpha(x) = (\alpha_\iota(x) )^{\mathcal U}$ we have that  $\theta ( \alpha(P_i))$ is matricially inapproximable with respect to $\tilde N$. We also assume that $\theta(M_i)$ is matricially inapproximable with respect to $\tilde N$. If this were not the case, then by Theorem~\ref{thm:genspectralgapapprox} we could exchange the roles of $(N_1 * N_2)^{\mathcal U}$ and $( \tilde N_1 * \tilde N_2)^{\mathcal V} \rtimes \mathbb Z/2\mathbb Z$ and proceed similarly. 

We now show how the conclusion of Theorem~\ref{thm:ultrafreeembedding} is still valid in this setting. Indeed, as before, we may consider the free product deformation $\alpha_t * \alpha_t \in \Aut( ( \tilde N_1 * L\mathbb Z) * (\tilde N_2 * L \mathbb Z) )$ from \cite[Section 2.2]{IoPePo08}. Since this deformation commutes with the free flip, we extend it to a deformation on $\tilde \alpha_t \in \Aut ( ( (\tilde N_1 * L \mathbb Z) * (\tilde N_2 * L \mathbb Z) ) \rtimes \mathbb Z/2\mathbb Z )$ so that $\tilde \alpha_t$ is the identity on the generator of $\mathbb Z/ 2\mathbb Z$. Just as in the case for free products, we see that the orthogonal complement of the trivial $( \tilde N_1 * \tilde N_2 ) \rtimes \mathbb Z/ 2 \mathbb Z$-bimodule in $L^2 ( (  ( \tilde N_1 * L \mathbb Z) * (\tilde N_2 * L \mathbb Z) )  \rtimes \mathbb Z/2\mathbb Z )$ is contained in a direct sum of the coarse bimodule, and so the argument in Proposition~\ref{prop:relcommutant} shows that 
\[
\theta(M_i)' \cap  ( ((\tilde N_1 * L \mathbb Z) * (\tilde N_2 * L \mathbb Z)  ) \rtimes \mathbb Z/2\mathbb Z)^{\mathcal V} = \theta( M_i' \cap (N_1 * N_2)^{\mathcal U})
\] 
is a factor. 

By Theorem~\ref{thm:propTultra} we then conclude that we have an inclusion of $\theta( M_i )$-bimodules 
\[
L^2M_i \subset L^2 ( ((\tilde N_1 * L \mathbb Z) * (\tilde N_2 * L \mathbb Z)  \rtimes \mathbb Z/2\mathbb Z)^{\mathcal V} )_{\alpha_1},
\] 
where the right action is twisted by the flip automorphism on $( (\tilde N_1 * L \mathbb Z) * (\tilde N_2 * L \mathbb Z)  \rtimes \mathbb Z/2\mathbb Z)^{\mathcal V} $.

We now note that, just as in Proposition~\ref{prop:afpwkcontain}, The $( \tilde N_1 * \tilde N_2)  \rtimes \mathbb Z/2\mathbb Z$ Hilbert bimodule $L^2  (( (\tilde N_1 * L \mathbb Z) * (\tilde N_2 * L \mathbb Z) )  \rtimes \mathbb Z/2\mathbb Z )_{\alpha_1}$ decomposes as a direct sum of the bimodule $L^2( \langle ( \tilde N_1 * \tilde N_2 ) \rtimes \mathbb Z/2\mathbb Z, \tilde N_1 \rangle )$ and a bimodule that is weakly contained in the coarse bimodule. (Note that since $\tilde N_1$ and $\tilde N_2$ are conjugate, we only have one direct summand of the basic construction here). 

Applying Theorem~\ref{thm:amenimplieshyperfinite}, it then follows that we have $\theta(M_i) \prec_{(( \tilde N_1 * \tilde N_2) \rtimes \mathbb Z/2\mathbb Z)^{\mathcal V}} \tilde N_1$. Proposition~\ref{prop:relcommutant} remains valid in this setting and so, applying the argument in Theorem~\ref{thm:afpintertwining} just as above, we see that there are unitaries $v_i \in \mathcal U( (( \tilde N_1 * \tilde N_2) \rtimes \mathbb Z/2\mathbb Z)^{\mathcal V} )$ so that $\theta(N_i^{\mathcal U} ) \subset v_i \tilde N_1^{\mathcal V} v_i^*$. 

We now note that $\theta^{-1}(\tilde M_1)$ cannot be matricially approximable with respect to $N_1 * N_2$, because if $A_\iota \subset N_1 * N_2$ were finite-dimensional such that $\theta^{-1}(\tilde M_1) \subset \prod_{\iota \to \mathcal U} A_\iota$, then just as above we would have 
\[ 
\theta^{-1}( \tilde N_1^{\mathcal V} ) \subset \theta^{-1}( ( \tilde M_1' \cap  ( (\tilde N_1 * \tilde N_2) \rtimes \mathbb Z/2\mathbb Z)^{\mathcal V})' \cap ( (\tilde N_1 * \tilde N_2) \rtimes \mathbb Z/2\mathbb Z)^{\mathcal V}) \subset \prod_{\iota \to \mathcal U} A_\iota.
\] 
But since $\theta(N_i^{\mathcal U} ) \subset v_i \tilde N_1^{\mathcal V} v_i^*$, this would then give a contradiction by showing that $N_1$ is amenable. 

We may then apply the argument exactly as in Theorem~\ref{thm:bassserreee} to show that there exists a unitary $u \in \mathcal U( (N_1 * N_2)^{\mathcal U} )$ so that $\theta^{-1}(\tilde N_1^{\mathcal V} ) \subset u N_1^{\mathcal U} u^*$. But this then shows that $u \theta^{-1}(v_2) N_2^{\mathcal U} \theta^{-1}(v_2)^* u^* \subset N_1^{\mathcal U}$, which contradicts the fact that $N_2$ is not amenable. 
\end{proof}

\subsection{Ruination of elementarity by free products}  

The idea of the following lemma is contained in \cite[Proposition 4.13, Proposition 5.4]{GoJeKEPi25}. It is a variation on the same theme as in Proposition~\ref{prop:irredpseudocpt}.

\begin{lem}\label{lem:Tfactorialrel}
For any II$_1$ factor $N$ there exists a separable II$_1$ factor $P$ such that $P \equiv N$ and such that $P$ contains a property (T) II$_1$ subfactor having factorial relative commutant. 
\end{lem}
\begin{proof}
Suppose $\mathcal U$ is a countably incomplete ultrafilter, and take a property (T) II$_1$ factor $M \subset N^{\mathcal U}$ with factorial relative commutant. Such an inclusion $M \subset N^{\mathcal U}$ exists because we can take a decomposition $N \cong \mathbb M_{n_\iota}(\mathbb C) \otimes (M_{n_\iota}(\mathbb C)' \cap N)$ and consider the representations $\pi_{n_\iota}: SL_3(\mathbb Z) \to \mathcal U(n_\iota)$ from Proposition~\ref{prop:irredpseudocpt}. The ultraproduct representation $\pi: SL_3(\mathbb Z) \to \prod_{\iota \to \mathcal U} \mathbb M_{n_\iota}(\mathbb C) \subset N^{\mathcal U}$ then satisfies $\pi(SL_3(\mathbb Z))' \cap N^{\mathcal U} \cong \prod_{\iota \to \mathcal U} (M_{n_\iota}(\mathbb C)' \cap N)$.  

By the Downward L\"owenheim-Skolem Theorem (see \cite[Theorem 2.3]{FaHaSh14}) we may find a separable II$_1$ factor $P \subset N^{\mathcal U}$ so that $M \subset P$ and such that the inclusion $P \subset N^{\mathcal U}$ is elementary. There are then ultrafilters $\mathcal V$, $\mathcal V'$ such that the embedding $P \subset N^{\mathcal U}$ extends to an isomorphism $P^{\mathcal V} \cong (N^{\mathcal U})^{\mathcal V'}$. Since $M$ has (T) and $M' \cap N^{\mathcal U}$ is a factor, we have that $M' \cap (N^{\mathcal U})^{\mathcal V'}$ is also a factor, and hence so is $M' \cap P$. 
\end{proof}

We remark that the conclusion of the previous lemma can alternatively be described using model theoretic sentences as in \cite[Proposition 5.12]{GoJeKEPi25}. We also remark that the previous proof shows that the property (T) subfactor can be chosen to be any property (T) factor that has an irreducible embedding into $\prod_{\iota \to \mathcal U} \mathbb M_{n_\iota}(\mathbb C)$, which is conjectured by Ioana to also be any property (T) factor that embeds into $\prod_{\iota \to \mathcal U} \mathbb M_{n_\iota}(\mathbb C)$.

The following result shows that taking free products does not preserve first order theories. This answers \cite[Question 3.8]{GoPi25}.

\begin{thm}\label{thm:freeproductee}
Let $M$ be a property (T) II$_1$ factor and set $N = M^{\overline \otimes \mathbb N}$. Let $P$ be a II$_1$ factor so that $P \equiv R$ and $P$ contains a property (T) II$_1$ subfactor $M_0 \subset P$ such that $\mathcal Z(M_0' \cap P) = \mathbb C$. Then $N * N * R \not\equiv N * N * P$. 
\end{thm}
\begin{proof}
To make notation easier, let us label the copies of $N$ in $N * N * R$ as $N_1$ and $N_2$ and let us label the copies of $N$ in $N * N * P$ as $\tilde N_1$ and $\tilde N_2$. Suppose $\theta: (N_1 * N_2 * R)^{\mathcal U} \to (\tilde N_1 * \tilde N_2 * P)^{\mathcal V}$ were an isomorphism. Essentially the same proof as in Theorem~\ref{thm:bassserreee} then shows that, after permutation of indices, there exist $u_1, u_2 \in \mathcal U( (\tilde N_1 * \tilde N_2 * P )^{\mathcal V})$ so that $\theta( N_i^{\mathcal U} ) = u_i \tilde N_i^{\mathcal V} u_i^*$.

Since $M_0 \subset P$ is a property (T) subfactor with factorial relative commutant, Proposition~\ref{prop:ultrabackandforth} and Theorem~\ref{thm:ultrafreeembedding} show that there exists $i \in \{ 1, 2 \}$ and $v \in \mathcal U( (N_1 * N_2 * R)^{\mathcal U})$ so that $v \theta^{-1}(M_0) v^* \subset N_i^{\mathcal U}$. But then $M_0 \subset \theta(v^*) \theta(N_i^{\mathcal U}) \theta(v) \subset \theta(v^*) u_i \tilde N_i^{\mathcal V} u_i^* \theta(v)$, which would show that $M_0$ is amenable, giving a contradiction. 
\end{proof}

\subsection{A Continuum of complete theories of full factors}\label{sec:continuum}

We recall the construction of McDuff's factors from \cite{DiLa69, Mc69a, Mc69, BoChIo17} (see also \cite{GoHa17, GoHaTo18}). We first fix a countable group $\Gamma$. For $i \in \mathbb N$, we let $\Gamma_i$ be isomorphic copies of $\Gamma$ (we identify $\Gamma$ with the first copy $\Gamma_1)$ and let $\Lambda_i$ be an isomorphic copy of $\mathbb Z$. We let $\tilde \Gamma = \oplus_{i \in \mathbb N} \Gamma$ and we let $S_{\mathbb N}$, the group of finite permutations on the natural number, act naturally on $\tilde \Gamma$ by shifting the indices. We set 
\[
T_0(\Gamma) = \langle \tilde \Gamma, (\Lambda_i)_{i \in \mathbb N} \mid [\Gamma_i, \Lambda_j] = 0 {\rm  \ for \ } i \geq j \rangle
\] 
and 
\[
T_1(\Gamma) = \langle \tilde \Gamma \rtimes S_{\mathbb N}, (\Lambda_i)_{i \in \mathbb N} \mid [\Gamma_i, \Lambda_j] = 0 {\rm \ for \ } i \geq j \rangle.
\]

If we have an inclusion of groups $\Sigma < \Sigma'$, then this gives rise to a canonical inclusion $T_\alpha(\Sigma) < T_\alpha(\Sigma')$. Thus, any sequence $\alpha \in 2^{\leq \omega}$ gives rise to a sequence of inclusions
\[
\Gamma < T_{\alpha_1}(\Gamma) < T_{\alpha_1} \circ T_{\alpha_2}(\Gamma ) < T_{\alpha_1} \circ T_{\alpha_2} \circ T_{\alpha_3}(\Gamma) \cdots.
\]
The group $T_\alpha(\Gamma)$ is defined to be the inductive limit of these inclusions. Starting with any group $\Gamma$ and any sequence $\alpha \in 2^{\mathbb N}$, the group von Neumann algebra $L( T_\alpha(\Gamma))$ will be a II$_1$ factor with the McDuff property.  

Boutonnet, Chifan, and Ioana showed in \cite{BoChIo17} that, starting with any countable group $\Gamma$, these factors are pairwise non-elementarily equivalent. Their proof has a simple extension that is compatible with our class $\mathcal C_T$. To see this we recall their notion of a residual net of a von Neumann algebra $\mathcal M$, which consists of a decreasing net $(A_j)_{j \in J}$ of von Neumann subalgebras of $\mathcal M$ such that 
\begin{itemize}
\item if $Q \subset \mathcal M$ is a separable von Neumann subalgebra, there exists $j \in J$ so that $A_j \subset Q' \cap \mathcal M$,
\item if $j \in J$, there exists a separable von Neumann subalgebra $Q \subset \mathcal M$ so that $Q' \cap \mathcal M \subset A_j$. 
\end{itemize}

\begin{lem}[{\cite[Lemma 3.3]{BoChIo17}}]\label{lem:bochiotensor}
Let $\Gamma$ be a countable group, $N$ a separable full factor and $\alpha \in \{ 0, 1 \}$. For $n \geq 1$, let $M = L(T_\alpha(\Gamma) )$, and $P_n = L(\tilde \Gamma_n)$, where $\tilde \Gamma_n = \oplus_{i \geq n} \Gamma_i < T_\alpha(\Gamma)$. Suppose $\mathcal U$ is a countably incomplete ultrafilter on $I$. Suppose for each $\iota \in I$ we have $t_\iota \geq 1$ an integer and set $\mathcal M = \prod_{\iota \to \mathcal U} (M^{\overline \otimes t_\iota} \overline \otimes N)$. Endow $J = \mathbb N^I$ with the partial order $(n_\iota)_\iota < (m_\iota)_\iota$ if and only if $n_s < m_s$ for $\iota \in I$. For $j = (n_\iota)_\iota \in J$ set $A_j = \prod_{\iota \to \mathcal U} P_{n_\iota}^{\overline \otimes t_\iota}$. Then $(A_j)_{j \in J}$ is a residual net of $\mathcal M$.
\end{lem}
\begin{proof}
If $Q \subset \mathcal M$ is a separable von Neumann subalgebra, then it was already established explicitly in \cite[Corollary 2.5]{BoChIo17} that there exists $j \in J$ with $A_j \subset Q' \cap \mathcal M$, hence we only need to show the second defining characteristic. Given $j \in J$, it is shown in \cite[Corollary 2.12]{BoChIo17} that there exists a separable von Neumann subalgebra $Q_0 \subset \prod_{\iota \to \mathcal U} M^{\overline \otimes t_\iota}$ with the property that $Q_0' \cap \prod_{\iota \to \mathcal U} (M^{\overline \otimes t_\iota} \overline \otimes N) \subset \prod_{\iota \to \mathcal U} (P_{n_\iota}^{\overline \otimes t_\iota} \overline \otimes N)$. If we let $Q$ be the von Neumann algebra generated by $Q_0$ and $N$, then as $N$ is full we have 
\[
Q' \cap  \prod_{\iota \to \mathcal U} (M^{\overline \otimes t_\iota} \overline \otimes N) = Q_0' \cap \prod_{\iota \to \mathcal U} M^{\overline \otimes t_\iota} \subset \prod_{\iota \to \mathcal U} P_{n_\iota}^{\overline \otimes t_\iota}. \qedhere
\]
\end{proof}

A von Neumann algebra $\mathcal M$ has a property $\mathcal P$ at depth 1 if for some (and hence all) residual net $(A_j)_j$ of $\mathcal M$ we have that for $j_1 \in I$ there exists $j_2 > j_1$ such that for all $j_3 > j_2$ there exists $j_4 > j_3$ so that the inclusion $A_{j_3}' \cap A_{j_2} \subset A_{j_4}' \cap A_{j_1}$ contains an intermediate von Neumann subalgebra with property $\mathcal P$. 

Because of Lemma~\ref{lem:bochiotensor} we see that if $\Gamma$ is a countable group, $M = L(T_\alpha(\Gamma))$, $\alpha \in \{ 0, 1 \}$ and $N$ is a separable full factor, then $(M \overline \otimes N)^{\mathcal U}$ has a property $\mathcal P$ at depth $1$ if and only if $M^{\mathcal U}$ has property $\mathcal P$ at depth $1$. An immediate consequence of \cite{BoChIo17} is then the following result where if $\alpha = (\alpha_1, \alpha_2, \ldots)$, then we denote by $1 \hat{ \ } \alpha$ the sequence $(1, \alpha_1, \alpha_2, \ldots)$.

\begin{prop}\label{prop:bochioele}
Let $\Gamma$ be a countable group and suppose $N$ is a separable full II$_1$ factor, then the von Neumann algebras $\{ L( T_{1 \hat{ \ } \alpha}(\Gamma) ) \overline \otimes N \}_{\alpha \in 2^{\mathbb N} }$ are pairwise non-elementarily equivalent.
\end{prop}

We remark, that going through the proof in \cite{BoChIo17} it is possible to show that the previous proposition holds with $\alpha$ in place of $1 \hat{ \ } \alpha$, but for the sake of brevity we have chosen this more direct path.

\begin{lem}\label{lem:mcduffC}
Let $\Gamma$ be a non-trivial i.c.c. property (T) group and let $N$ be a separable II$_1$ factor with property (T), then for $\alpha \in 2^{\mathbb N}$ the von Neumann algebra $L( T_\alpha(\Gamma) ) \overline \otimes N$ is in class $\mathcal C_T$. 
\end{lem}
\begin{proof}
The relative commutant of $N$ in $L( T_\alpha(\Gamma) ) \overline \otimes N$ is $L( T_\alpha(\Gamma) )$, which is a factor.  Also, if we write $\alpha = (\alpha_1, \alpha_2, \alpha_3, \ldots)$ and $\alpha' = (\alpha_2, \alpha_3, \ldots)$, then using the decomposition $T_\alpha(\Gamma) = T_{\alpha_1}( T_{\alpha'}(\Gamma) )$ we have embeddings $\beta_i: L\Gamma \to L\Gamma_i \subset L( T_{\alpha} (\Gamma) )$, and $x \in L( T_\alpha(\Gamma))$, then we have $\| \mathbb E_{\beta_i(L\Gamma)' \cap L( T_\alpha(\Gamma))}(x ) - x \|_2 \to 0$. Since $\Gamma$ has (T), the sequence $(\beta_i)_i$ has uniform weak spectral gap.
\end{proof}

Theorem~\ref{thm:uncountablemcduff} now follows from Lemma~\ref{lem:mcduffC}, together with Theorem~\ref{thm:bassserreee} and Proposition~\ref{prop:bochioele}.

\bibliographystyle{amsalpha}
\bibliography{ref}

@article {AD95,
    AUTHOR = {Anantharaman-Delaroche, C.},
     TITLE = {Amenable correspondences and approximation properties for von
              {N}eumann algebras},
   JOURNAL = {Pacific J. Math.},
  FJOURNAL = {Pacific Journal of Mathematics},
    VOLUME = {171},
      YEAR = {1995},
    NUMBER = {2},
     PAGES = {309--341},
      ISSN = {0030-8730,1945-5844},
       URL = {http://projecteuclid.org/euclid.pjm/1102368918},
}

@article {Sa75,
    AUTHOR = {Sakai, Sh\^oichir\^o},
     TITLE = {Automorphisms and tensor products of operator algebras},
   JOURNAL = {Amer. J. Math.},
  FJOURNAL = {American Journal of Mathematics},
    VOLUME = {97},
      YEAR = {1975},
    NUMBER = {4},
     PAGES = {889--896},
      ISSN = {0002-9327,1080-6377},
       DOI = {10.2307/2373678},
       URL = {https://doi-org.proxy.lib.uwaterloo.ca/10.2307/2373678},
}

@article {JeKE26,
    AUTHOR = {Jekel, David and Kunnawalkam Elayavalli, Srivatsav},
     TITLE = {Upgraded free independence phenomena for random unitaries},
   JOURNAL = {Trans. Amer. Math. Soc. Ser. B},
  FJOURNAL = {Transactions of the American Mathematical Society. Series B},
    VOLUME = {13},
      YEAR = {2026},
     PAGES = {1--29},
      ISSN = {2330-0000},
}

@article {Vo90,
    AUTHOR = {Voiculescu, Dan},
     TITLE = {Property {$T$} and approximation of operators},
   JOURNAL = {Bull. London Math. Soc.},
  FJOURNAL = {The Bulletin of the London Mathematical Society},
    VOLUME = {22},
      YEAR = {1990},
    NUMBER = {1},
     PAGES = {25--30},
      ISSN = {0024-6093,1469-2120},
       DOI = {10.1112/blms/22.1.25},
       URL = {https://doi-org.proxy.lib.uwaterloo.ca/10.1112/blms/22.1.25},
}

@article {GoHaTo18,
    AUTHOR = {Goldbring, Isaac and Hart, Bradd and Towsner, Henry},
     TITLE = {Explicit sentences distinguishing {M}c{D}uff's {$\rm II_1$}
              factors},
   JOURNAL = {Israel J. Math.},
  FJOURNAL = {Israel Journal of Mathematics},
    VOLUME = {227},
      YEAR = {2018},
    NUMBER = {1},
     PAGES = {365--377},
      ISSN = {0021-2172,1565-8511},
       DOI = {10.1007/s11856-018-1735-8},
       URL = {https://doi-org.proxy.lib.uwaterloo.ca/10.1007/s11856-018-1735-8},
}

@article {Sa68,
    AUTHOR = {Sakai, Sh\^oichir\^o},
     TITLE = {Asymptotically abelian {${\rm II}\sb{1}$}-factors},
   JOURNAL = {Publ. Res. Inst. Math. Sci. Ser. A},
  FJOURNAL = {Publ. Res. Inst. Math. Sci. Ser. A},
    VOLUME = {4},
      YEAR = {1968/69},
     PAGES = {299--307},
       DOI = {10.2977/prims/1195194878},
       URL = {https://doi-org.proxy.lib.uwaterloo.ca/10.2977/prims/1195194878},
}

@article {DiLa69,
    AUTHOR = {Dixmier, J. and Lance, E. C.},
     TITLE = {Deux nouveaux facteurs de type {${\rm II}\sb{1}$}},
   JOURNAL = {Invent. Math.},
  FJOURNAL = {Inventiones Mathematicae},
    VOLUME = {7},
      YEAR = {1969},
     PAGES = {226--234},
      ISSN = {0020-9910,1432-1297},
       DOI = {10.1007/BF01404307},
       URL = {https://doi-org.proxy.lib.uwaterloo.ca/10.1007/BF01404307},
}

@article {FaHaRoTi17,
    AUTHOR = {Farah, Ilijas and Hart, Bradd and R{\o}rdam, Mikael and
              Tikuisis, Aaron},
     TITLE = {Relative commutants of strongly self-absorbing {$\rm
              C^*$}-algebras},
   JOURNAL = {Selecta Math. (N.S.)},
  FJOURNAL = {Selecta Mathematica. New Series},
    VOLUME = {23},
      YEAR = {2017},
    NUMBER = {1},
     PAGES = {363--387},
      ISSN = {1022-1824,1420-9020},
       DOI = {10.1007/s00029-016-0237-y},
       URL = {https://doi-org.proxy.lib.uwaterloo.ca/10.1007/s00029-016-0237-y},
}

@article {Mc69a,
    AUTHOR = {McDuff, Dusa},
     TITLE = {A countable infinity of {$\Pi \sb{1}$} factors},
   JOURNAL = {Ann. of Math. (2)},
  FJOURNAL = {Annals of Mathematics. Second Series},
    VOLUME = {90},
      YEAR = {1969},
     PAGES = {361--371},
      ISSN = {0003-486X},
       DOI = {10.2307/1970729},
       URL = {https://doi.org/10.2307/1970729},
}

@article {Mc69,
    AUTHOR = {McDuff, Dusa},
     TITLE = {Uncountably many {${\rm II}\sb{1}$} factors},
   JOURNAL = {Ann. of Math. (2)},
  FJOURNAL = {Annals of Mathematics. Second Series},
    VOLUME = {90},
      YEAR = {1969},
     PAGES = {372--377},
      ISSN = {0003-486X},
       DOI = {10.2307/1970730},
       URL = {https://doi.org/10.2307/1970730},
}

@article {GoHa17,
    AUTHOR = {Goldbring, Isaac and Hart, Bradd},
     TITLE = {On the theories of {M}c{D}uff's {$\rm II_1$} factors},
   JOURNAL = {Int. Math. Res. Not. IMRN},
  FJOURNAL = {International Mathematics Research Notices. IMRN},
      YEAR = {2017},
    NUMBER = {18},
     PAGES = {5609--5628},
      ISSN = {1073-7928,1687-0247},
       DOI = {10.1093/imrn/rnw180},
       URL = {https://doi-org.proxy.lib.uwaterloo.ca/10.1093/imrn/rnw180},
}

@article {ZM69,
    AUTHOR = {Zeller-Meier, G.},
     TITLE = {Deux nouveaux facteurs de type {${\rm II}\sb{1}$}},
   JOURNAL = {Invent. Math.},
  FJOURNAL = {Inventiones Mathematicae},
    VOLUME = {7},
      YEAR = {1969},
     PAGES = {235--242},
      ISSN = {0020-9910,1432-1297},
       DOI = {10.1007/BF01404308},
       URL = {https://doi-org.proxy.lib.uwaterloo.ca/10.1007/BF01404308},
}

@article {Po12,
    AUTHOR = {Popa, Sorin},
     TITLE = {On the classification of inductive limits of {II{$_{1}$}}
              factors with spectral gap},
   JOURNAL = {Trans. Amer. Math. Soc.},
  FJOURNAL = {Transactions of the American Mathematical Society},
    VOLUME = {364},
      YEAR = {2012},
    NUMBER = {6},
     PAGES = {2987--3000},
      ISSN = {0002-9947,1088-6850},
       DOI = {10.1090/S0002-9947-2012-05389-X},
       URL = {https://doi-org.proxy.lib.uwaterloo.ca/10.1090/S0002-9947-2012-05389-X},
}

@article {Po08,
    AUTHOR = {Popa, Sorin},
     TITLE = {On the superrigidity of malleable actions with spectral gap},
   JOURNAL = {J. Amer. Math. Soc.},
  FJOURNAL = {Journal of the American Mathematical Society},
    VOLUME = {21},
      YEAR = {2008},
    NUMBER = {4},
     PAGES = {981--1000},
      ISSN = {0894-0347,1088-6834},
       DOI = {10.1090/S0894-0347-07-00578-4},
       URL = {https://doi-org.proxy.lib.uwaterloo.ca/10.1090/S0894-0347-07-00578-4},
}

@article {Po95,
    AUTHOR = {Popa, Sorin},
     TITLE = {Free-independent sequences in type {${\rm II}_1$} factors and
              related problems},
      NOTE = {Recent advances in operator algebras (Orl\'eans, 1992)},
   JOURNAL = {Ast\'erisque},
  FJOURNAL = {Ast\'erisque},
    NUMBER = {232},
      YEAR = {1995},
     PAGES = {187--202},
      ISSN = {0303-1179,2492-5926},
}

@article {Po14,
    AUTHOR = {Popa, Sorin},
     TITLE = {Independence properties in subalgebras of ultraproduct {$\rm
              II_1$} factors},
   JOURNAL = {J. Funct. Anal.},
  FJOURNAL = {Journal of Functional Analysis},
    VOLUME = {266},
      YEAR = {2014},
    NUMBER = {9},
     PAGES = {5818--5846},
      ISSN = {0022-1236,1096-0783},
       DOI = {10.1016/j.jfa.2014.02.004},
       URL = {https://doi-org.proxy.lib.uwaterloo.ca/10.1016/j.jfa.2014.02.004},
}

@article{HaJeKE25, 
title={PROPERTY ({T}) AND STRONG 1-BOUNDEDNESS FOR VON NEUMANN ALGEBRAS}, 
DOI={10.1017/S1474748024000446}, 
journal={Journal of the Institute of Mathematics of Jussieu}, 
author={Hayes, Ben and Jekel, David and Kunnawalkam Elayavalli, Srivatsav}, 
year={2025}, 
pages={1--34},
}

@article {EfLa77,
    AUTHOR = {Effros, Edward G. and Lance, E. Christopher},
     TITLE = {Tensor products of operator algebras},
   JOURNAL = {Adv. Math.},
  FJOURNAL = {Advances in Mathematics},
    VOLUME = {25},
      YEAR = {1977},
    NUMBER = {1},
     PAGES = {1--34},
      ISSN = {0001-8708,1090-2082},
       DOI = {10.1016/0001-8708(77)90085-8},
       URL = {https://doi-org.proxy.lib.uwaterloo.ca/10.1016/0001-8708(77)90085-8},
}

@article {Ha18,
    AUTHOR = {Hayes, Ben},
     TITLE = {1-bounded entropy and regularity problems in von {N}eumann
              algebras},
   JOURNAL = {Int. Math. Res. Not. IMRN},
  FJOURNAL = {International Mathematics Research Notices. IMRN},
      YEAR = {2018},
    NUMBER = {1},
     PAGES = {57--137},
      ISSN = {1073-7928,1687-0247},
       DOI = {10.1093/imrn/rnw237},
       URL = {https://doi-org.proxy.lib.uwaterloo.ca/10.1093/imrn/rnw237},
}

@article {Po06,
    AUTHOR = {Popa, Sorin},
     TITLE = {Some rigidity results for non-commutative {B}ernoulli shifts},
   JOURNAL = {J. Funct. Anal.},
  FJOURNAL = {Journal of Functional Analysis},
    VOLUME = {230},
      YEAR = {2006},
    NUMBER = {2},
     PAGES = {273--328},
      ISSN = {0022-1236,1096-0783},
       DOI = {10.1016/j.jfa.2005.06.017},
       URL = {https://doi-org.proxy.lib.uwaterloo.ca/10.1016/j.jfa.2005.06.017},
}

@article {KE23,
    AUTHOR = {Kunnawalkam Elayavalli, Srivatsav},
     TITLE = {Remarks on the diagonal embedding and strong 1-boundedness},
   JOURNAL = {Doc. Math.},
  FJOURNAL = {Documenta Mathematica},
    VOLUME = {28},
      YEAR = {2023},
    NUMBER = {3},
     PAGES = {671--681},
      ISSN = {1431-0635,1431-0643},
       DOI = {10.4171/dm/918},
       URL = {https://doi-org.proxy.lib.uwaterloo.ca/10.4171/dm/918},
}

@article {KEP25,
    AUTHOR = {Kunnawalkam Elayavalli, Srivatsav and Patchell, Gregory},
     TITLE = {Sequential commutation in tracial von {N}eumann algebras},
   JOURNAL = {J. Funct. Anal.},
  FJOURNAL = {Journal of Functional Analysis},
    VOLUME = {288},
      YEAR = {2025},
    NUMBER = {4},
     PAGES = {Paper No. 110719, 28},
      ISSN = {0022-1236,1096-0783},
}

@article {VN42,
    AUTHOR = {von Neumann, John},
     TITLE = {Approximative properties of matrices of high finite order},
   JOURNAL = {Portugal. Math.},
  FJOURNAL = {Portugaliae Mathematica},
    VOLUME = {3},
      YEAR = {1942},
     PAGES = {1--62},
      ISSN = {0032-5155,1662-2758},
 }

@inproceedings {Va10,
    AUTHOR = {Vaes, Stefaan},
     TITLE = {Rigidity for von {N}eumann algebras and their invariants},
 BOOKTITLE = {Proceedings of the {I}nternational {C}ongress of
              {M}athematicians. {V}olume {III}},
     PAGES = {1624--1650},
 PUBLISHER = {Hindustan Book Agency, New Delhi},
      YEAR = {2010},
      ISBN = {978-81-85931-08-3; 978-981-4324-33-5; 981-4324-33-7},
}

@inproceedings {Io18,
    AUTHOR = {Ioana, Adrian},
     TITLE = {Rigidity for von {N}eumann algebras},
 BOOKTITLE = {Proceedings of the {I}nternational {C}ongress of
              {M}athematicians---{R}io de {J}aneiro 2018. {V}ol. {III}.
              {I}nvited lectures},
     PAGES = {1639--1672},
 PUBLISHER = {World Sci. Publ., Hackensack, NJ},
      YEAR = {2018},
      ISBN = {978-981-3272-92-7; 978-981-3272-87-3},
}

@article {NiPoSa07,
    AUTHOR = {Nicoara, Remus and Popa, Sorin and Sasyk, Roman},
     TITLE = {On {${\rm II}_1$} factors arising from 2-cocycles of
              {$w$}-rigid groups},
   JOURNAL = {J. Funct. Anal.},
  FJOURNAL = {Journal of Functional Analysis},
    VOLUME = {242},
      YEAR = {2007},
    NUMBER = {1},
     PAGES = {230--246},
      ISSN = {0022-1236,1096-0783},
       DOI = {10.1016/j.jfa.2006.05.015},
       URL = {https://doi-org.proxy.lib.uwaterloo.ca/10.1016/j.jfa.2006.05.015},
}

@incollection {Po07a,
    AUTHOR = {Popa, Sorin},
     TITLE = {Deformation and rigidity for group actions and von {N}eumann
              algebras},
 BOOKTITLE = {International {C}ongress of {M}athematicians. {V}ol. {I}},
     PAGES = {445--477},
 PUBLISHER = {Eur. Math. Soc., Z\"urich},
      YEAR = {2007},
      ISBN = {978-3-03719-022-7},
       DOI = {10.4171/022-1/18},
       URL = {https://doi-org.proxy.lib.uwaterloo.ca/10.4171/022-1/18},
}

@article {CoHa89,
    AUTHOR = {Cowling, Michael and Haagerup, Uffe},
     TITLE = {Completely bounded multipliers of the {F}ourier algebra of a
              simple {L}ie group of real rank one},
   JOURNAL = {Invent. Math.},
  FJOURNAL = {Inventiones Mathematicae},
    VOLUME = {96},
      YEAR = {1989},
    NUMBER = {3},
     PAGES = {507--549},
      ISSN = {0020-9910,1432-1297},
       DOI = {10.1007/BF01393695},
       URL = {https://doi-org.proxy.lib.uwaterloo.ca/10.1007/BF01393695},
}

@article {Ch83,
    AUTHOR = {Choda, Marie},
     TITLE = {Group factors of the {H}aagerup type},
   JOURNAL = {Proc. Japan Acad. Ser. A Math. Sci.},
  FJOURNAL = {Japan Academy. Proceedings. Series A. Mathematical Sciences},
    VOLUME = {59},
      YEAR = {1983},
    NUMBER = {5},
     PAGES = {174--177},
      ISSN = {0386-2194},
       URL = {http://projecteuclid.org.proxy.lib.uwaterloo.ca/euclid.pja/1195515589},
}

@book {ChCoJoJuVa01,
    AUTHOR = {Cherix, Pierre-Alain and Cowling, Michael and Jolissaint, Paul
              and Julg, Pierre and Valette, Alain},
     TITLE = {Groups with the {H}aagerup property},
    SERIES = {Progress in Mathematics},
    VOLUME = {197},
      NOTE = {Gromov's a-T-menability},
 PUBLISHER = {Birkh\"auser Verlag, Basel},
      YEAR = {2001},
     PAGES = {viii+126},
      ISBN = {3-7643-6598-6},
       DOI = {10.1007/978-3-0348-8237-8},
       URL = {https://doi-org.proxy.lib.uwaterloo.ca/10.1007/978-3-0348-8237-8},
}

@book {BeHaVa08,
    AUTHOR = {Bekka, Bachir and de la Harpe, Pierre and Valette, Alain},
     TITLE = {Kazhdan's property ({T})},
    SERIES = {New Mathematical Monographs},
    VOLUME = {11},
 PUBLISHER = {Cambridge University Press, Cambridge},
      YEAR = {2008},
     PAGES = {xiv+472},
      ISBN = {978-0-521-88720-5},
       DOI = {10.1017/CBO9780511542749},
       URL = {https://doi-org.proxy.lib.uwaterloo.ca/10.1017/CBO9780511542749},
}

@article {CoJo85,
    AUTHOR = {Connes, A. and Jones, V.},
     TITLE = {Property {$T$} for von {N}eumann algebras},
   JOURNAL = {Bull. London Math. Soc.},
  FJOURNAL = {The Bulletin of the London Mathematical Society},
    VOLUME = {17},
      YEAR = {1985},
    NUMBER = {1},
     PAGES = {57--62},
      ISSN = {0024-6093,1469-2120},
}

@article {Oz06,
    AUTHOR = {Ozawa, Narutaka},
     TITLE = {A {K}urosh-type theorem for type {$\rm II_1$} factors},
   JOURNAL = {Int. Math. Res. Not.},
  FJOURNAL = {International Mathematics Research Notices},
      YEAR = {2006},
     PAGES = {Art. ID 97560, 21},
      ISSN = {1073-7928,1687-0247},
       DOI = {10.1155/IMRN/2006/97560},
       URL = {https://doi-org.proxy.lib.uwaterloo.ca/10.1155/IMRN/2006/97560},
}

@article {GoPi25,
    AUTHOR = {Goldbring, Isaac and Pi, Jennifer},
     TITLE = {On the first-order free group factor elementary equivalence},
   JOURNAL = {J. Operator Theory},
  FJOURNAL = {Journal of Operator Theory},
    VOLUME = {94},
      YEAR = {2025},
    NUMBER = {1},
     PAGES = {129--150},
      ISSN = {0379-4024,1841-7744},
}

@article {FaHa11,
    AUTHOR = {Fang, Junsheng and Hadwin, Don},
     TITLE = {A note on the invariant subspace problem relative to a type
              {${\rm II}_1$} factor},
   JOURNAL = {Houston J. Math.},
  FJOURNAL = {Houston Journal of Mathematics},
    VOLUME = {37},
      YEAR = {2011},
    NUMBER = {3},
     PAGES = {879--893},
      ISSN = {0362-1588},
}

@article {IoTa24,
    AUTHOR = {Ioana, Adrian and Tan, Hui},
     TITLE = {Existential closedness and the structure of bimodules of {$\rm
              II_1$} factors},
   JOURNAL = {J. Funct. Anal.},
  FJOURNAL = {Journal of Functional Analysis},
    VOLUME = {286},
      YEAR = {2024},
    NUMBER = {4},
     PAGES = {Paper No. 110264, 31},
      ISSN = {0022-1236,1096-0783},
       DOI = {10.1016/j.jfa.2023.110264},
       URL = {https://doi-org.proxy.lib.uwaterloo.ca/10.1016/j.jfa.2023.110264},
}

@article {GoLo15,
    AUTHOR = {Goldbring, Isaac and Lopes, Vinicius Cif\'u},
     TITLE = {Pseudofinite and pseudocompact metric structures},
   JOURNAL = {Notre Dame J. Form. Log.},
  FJOURNAL = {Notre Dame Journal of Formal Logic},
    VOLUME = {56},
      YEAR = {2015},
    NUMBER = {3},
     PAGES = {493--510},
      ISSN = {0029-4527,1939-0726},
       DOI = {10.1215/00294527-3132833},
       URL = {https://doi-org.proxy.lib.uwaterloo.ca/10.1215/00294527-3132833},
}

@article {FaHaSh13,
    AUTHOR = {Farah, Ilijas and Hart, Bradd and Sherman, David},
     TITLE = {Model theory of operator algebras {I}: stability},
   JOURNAL = {Bull. Lond. Math. Soc.},
  FJOURNAL = {Bulletin of the London Mathematical Society},
    VOLUME = {45},
      YEAR = {2013},
    NUMBER = {4},
     PAGES = {825--838},
      ISSN = {0024-6093,1469-2120},
       DOI = {10.1112/blms/bdt014},
       URL = {https://doi-org.proxy.lib.uwaterloo.ca/10.1112/blms/bdt014},
}

@article {FaHaSh14a,
    AUTHOR = {Farah, Ilijas and Hart, Bradd and Sherman, David},
     TITLE = {Model theory of operator algebras {II}: model theory},
   JOURNAL = {Israel J. Math.},
  FJOURNAL = {Israel Journal of Mathematics},
    VOLUME = {201},
      YEAR = {2014},
    NUMBER = {1},
     PAGES = {477--505},
      ISSN = {0021-2172,1565-8511},
       DOI = {10.1007/s11856-014-1046-7},
       URL = {https://doi-org.proxy.lib.uwaterloo.ca/10.1007/s11856-014-1046-7},
}

@article {FaHaSh14,
    AUTHOR = {Farah, Ilijas and Hart, Bradd and Sherman, David},
     TITLE = {Model theory of operator algebras {III}: elementary
              equivalence and {$\rm II_1$} factors},
   JOURNAL = {Bull. Lond. Math. Soc.},
  FJOURNAL = {Bulletin of the London Mathematical Society},
    VOLUME = {46},
      YEAR = {2014},
    NUMBER = {3},
     PAGES = {609--628},
      ISSN = {0024-6093,1469-2120},
       DOI = {10.1112/blms/bdu012},
       URL = {https://doi-org.proxy.lib.uwaterloo.ca/10.1112/blms/bdu012},
}

@misc{KE25,
Author = { Kunnawalkam Elayavalli, Srivatsav},
Title = {50 Open Problems: Ultraproduct {II$_1$} factors},
Year = {2025},
howpublished = {arXiv:2511.20377},
}

@article {AtGoKE22,
    AUTHOR = {Atkinson, Scott and Goldbring, Isaac and Kunnawalkam
              Elayavalli, Srivatsav},
     TITLE = {Factorial relative commutants and the generalized {J}ung
              property for {$\rm II_1$} factors},
   JOURNAL = {Adv. Math.},
  FJOURNAL = {Advances in Mathematics},
    VOLUME = {396},
      YEAR = {2022},
     PAGES = {Paper No. 108107, 53},
      ISSN = {0001-8708,1090-2082},
       DOI = {10.1016/j.aim.2021.108107},
       URL = {https://doi-org.proxy.lib.uwaterloo.ca/10.1016/j.aim.2021.108107},
}

@article {ADHa90,
    AUTHOR = {Anantharaman-Delaroche, C. and Havet, J.-F.},
     TITLE = {On approximate factorizations of completely positive maps},
   JOURNAL = {J. Funct. Anal.},
  FJOURNAL = {Journal of Functional Analysis},
    VOLUME = {90},
      YEAR = {1990},
    NUMBER = {2},
     PAGES = {411--428},
      ISSN = {0022-1236,1096-0783},
        DOI = {10.1016/0022-1236(90)90090-8},
       URL = {https://doi.org/10.1016/0022-1236(90)90090-8},
}

@article {AtKE21,
    AUTHOR = {Atkinson, Scott and Kunnawalkam Elayavalli, Srivatsav},
     TITLE = {On ultraproduct embeddings and amenability for tracial von
              {N}eumann algebras},
   JOURNAL = {Int. Math. Res. Not. IMRN},
  FJOURNAL = {International Mathematics Research Notices. IMRN},
      YEAR = {2021},
    NUMBER = {4},
     PAGES = {2882--2918},
      ISSN = {1073-7928,1687-0247},
       DOI = {10.1093/imrn/rnaa257},
       URL = {https://doi.org/10.1093/imrn/rnaa257},
}

@article {Ha85,
    AUTHOR = {Haagerup, Uffe},
     TITLE = {A new proof of the equivalence of injectivity and
              hyperfiniteness for factors on a separable {H}ilbert space},
   JOURNAL = {J. Funct. Anal.},
  FJOURNAL = {Journal of Functional Analysis},
    VOLUME = {62},
      YEAR = {1985},
    NUMBER = {2},
     PAGES = {160--201},
      ISSN = {0022-1236},
       DOI = {10.1016/0022-1236(85)90002-3},
       URL = {https://doi.org/10.1016/0022-1236(85)90002-3},
}

@article {FaGoHaSh16,
    AUTHOR = {Farah, Ilijas and Goldbring, Isaac and Hart, Bradd and
              Sherman, David},
     TITLE = {Existentially closed {${\rm II}_1$} factors},
   JOURNAL = {Fund. Math.},
  FJOURNAL = {Fundamenta Mathematicae},
    VOLUME = {233},
      YEAR = {2016},
    NUMBER = {2},
     PAGES = {173--196},
      ISSN = {0016-2736,1730-6329},
       DOI = {10.4064/fm126-12-2015},
       URL = {https://doi-org.proxy.lib.uwaterloo.ca/10.4064/fm126-12-2015},
}

@article {Go23,
    AUTHOR = {Goldbring, Isaac},
     TITLE = {Non-embeddable {$\rm II_1$} factors resembling the hyperfinite
              {$\rm II_1$} factor},
   JOURNAL = {J. Noncommut. Geom.},
  FJOURNAL = {Journal of Noncommutative Geometry},
    VOLUME = {17},
      YEAR = {2023},
    NUMBER = {1},
     PAGES = {233--239},
      ISSN = {1661-6952,1661-6960},
       DOI = {10.4171/jncg/474},
       URL = {https://doi-org.proxy.lib.uwaterloo.ca/10.4171/jncg/474},
}

@book {Go23book,
     TITLE = {Model theory of operator algebras},
    SERIES = {De Gruyter Series in Logic and its Applications},
    VOLUME = {11},
    EDITOR = {Goldbring, Isaac},
 PUBLISHER = {De Gruyter, Berlin},
      YEAR = {[2023] \copyright 2023},
     PAGES = {ix+484},
      ISBN = {978-3-11-076821-3; 978-3-11-076828-2; 978-3-11-076833-6},
}

@article {Go21,
    AUTHOR = {Goldbring, Isaac},
     TITLE = {Enforceable operator algebras},
   JOURNAL = {J. Inst. Math. Jussieu},
  FJOURNAL = {Journal of the Institute of Mathematics of Jussieu. JIMJ.
              Journal de l'Institut de Math\'ematiques de Jussieu},
    VOLUME = {20},
      YEAR = {2021},
    NUMBER = {1},
     PAGES = {31--63},
      ISSN = {1474-7480,1475-3030},
       DOI = {10.1017/S1474748019000112},
       URL = {https://doi-org.proxy.lib.uwaterloo.ca/10.1017/S1474748019000112},
}

@article {Vo94,
    AUTHOR = {Voiculescu, Dan},
     TITLE = {The analogues of entropy and of {F}isher's information measure
              in free probability theory. {II}},
   JOURNAL = {Invent. Math.},
  FJOURNAL = {Inventiones Mathematicae},
    VOLUME = {118},
      YEAR = {1994},
    NUMBER = {3},
     PAGES = {411--440},
      ISSN = {0020-9910,1432-1297},
       DOI = {10.1007/BF01231539},
       URL = {https://doi-org.proxy.lib.uwaterloo.ca/10.1007/BF01231539},
}

@incollection {Vo85,
    AUTHOR = {Voiculescu, Dan},
     TITLE = {Symmetries of some reduced free product {$C^\ast$}-algebras},
 BOOKTITLE = {Operator algebras and their connections with topology and
              ergodic theory ({B}u\c steni, 1983)},
    SERIES = {Lecture Notes in Math.},
    VOLUME = {1132},
     PAGES = {556--588},
 PUBLISHER = {Springer, Berlin},
      YEAR = {1985},
      ISBN = {3-540-15643-7},
       DOI = {10.1007/BFb0074909},
       URL = {https://doi-org.proxy.lib.uwaterloo.ca/10.1007/BFb0074909},
}

@misc{Po86,
author={Popa, Sorin}, 
title = {Correspondences}, 
Year = {1986},
howpublished = {Lecture Notes, INCREST, Romania, 1986, unpublished. Available at: https://www.math.ucla.edu/{\textasciitilde}popa/popa-correspondences.pdf },
}

@article {Po93,
    AUTHOR = {Popa, Sorin},
     TITLE = {Markov traces on universal {J}ones algebras and subfactors of
              finite index},
   JOURNAL = {Invent. Math.},
  FJOURNAL = {Inventiones Mathematicae},
    VOLUME = {111},
      YEAR = {1993},
    NUMBER = {2},
     PAGES = {375--405},
      ISSN = {0020-9910,1432-1297},
       DOI = {10.1007/BF01231293},
       URL = {https://doi-org.proxy.lib.uwaterloo.ca/10.1007/BF01231293},
}

@article {Ta23,
    AUTHOR = {Tan, Hui},
     TITLE = {Spectral gap characterizations of property ({T}) for {${\rm
              II}_1$} factors},
   JOURNAL = {Int. Math. Res. Not. IMRN},
  FJOURNAL = {International Mathematics Research Notices. IMRN},
      YEAR = {2023},
    NUMBER = {19},
     PAGES = {16994--17020},
      ISSN = {1073-7928,1687-0247},
       DOI = {10.1093/imrn/rnad109},
       URL = {https://doi-org.proxy.lib.uwaterloo.ca/10.1093/imrn/rnad109},
}

@article {GoHa24,
    AUTHOR = {Goldbring, Isaac and Hart, Bradd},
     TITLE = {The universal theory of the hyperfinite {$\rm II_1$} factor is
              not computable},
   JOURNAL = {Bull. Symb. Log.},
  FJOURNAL = {The Bulletin of Symbolic Logic},
    VOLUME = {30},
      YEAR = {2024},
    NUMBER = {2},
     PAGES = {181--198},
      ISSN = {1079-8986,1943-5894},
       DOI = {10.1017/bsl.2024.7},
       URL = {https://doi-org.proxy.lib.uwaterloo.ca/10.1017/bsl.2024.7},
}

@article {Oz04,
    AUTHOR = {Ozawa, Narutaka},
     TITLE = {There is no separable universal {$\rm II_1$}-factor},
   JOURNAL = {Proc. Amer. Math. Soc.},
  FJOURNAL = {Proceedings of the American Mathematical Society},
    VOLUME = {132},
      YEAR = {2004},
    NUMBER = {2},
     PAGES = {487--490},
      ISSN = {0002-9939,1088-6826},
        DOI = {10.1090/S0002-9939-03-07127-2},
       URL = {https://doi-org.proxy.lib.uwaterloo.ca/10.1090/S0002-9939-03-07127-2},
}

@article {Wa75,
    AUTHOR = {Wang, P. S.},
     TITLE = {On isolated points in the dual spaces of locally compact
              groups},
   JOURNAL = {Math. Ann.},
  FJOURNAL = {Mathematische Annalen},
    VOLUME = {218},
      YEAR = {1975},
    NUMBER = {1},
     PAGES = {19--34},
      ISSN = {0025-5831,1432-1807},
       DOI = {10.1007/BF01350065},
       URL = {https://doi-org.proxy.lib.uwaterloo.ca/10.1007/BF01350065},
}

@article {Wa91,
    AUTHOR = {Wassermann, Simon},
     TITLE = {{$C^*$}-algebras associated with groups with {K}azhdan's
              property {$T$}},
   JOURNAL = {Ann. of Math. (2)},
  FJOURNAL = {Annals of Mathematics. Second Series},
    VOLUME = {134},
      YEAR = {1991},
    NUMBER = {2},
     PAGES = {423--431},
      ISSN = {0003-486X,1939-8980},
       DOI = {10.2307/2944351},
       URL = {https://doi-org.proxy.lib.uwaterloo.ca/10.2307/2944351},
}

@article {OzPo04,
    AUTHOR = {Ozawa, Narutaka and Popa, Sorin},
     TITLE = {Some prime factorization results for type {${\rm II}_1$}
              factors},
   JOURNAL = {Invent. Math.},
  FJOURNAL = {Inventiones Mathematicae},
    VOLUME = {156},
      YEAR = {2004},
    NUMBER = {2},
     PAGES = {223--234},
      ISSN = {0020-9910,1432-1297},
       DOI = {10.1007/s00222-003-0338-z},
       URL = {https://doi-org.proxy.lib.uwaterloo.ca/10.1007/s00222-003-0338-z},
}

@article {Br11,
    AUTHOR = {Brown, Nathanial P.},
     TITLE = {Topological dynamical systems associated to {${\rm
              II}_1$}-factors},
      NOTE = {With an appendix by Narutaka Ozawa},
   JOURNAL = {Adv. Math.},
  FJOURNAL = {Advances in Mathematics},
    VOLUME = {227},
      YEAR = {2011},
    NUMBER = {4},
     PAGES = {1665--1699},
      ISSN = {0001-8708,1090-2082},
       DOI = {10.1016/j.aim.2011.04.003},
       URL = {https://doi-org.proxy.lib.uwaterloo.ca/10.1016/j.aim.2011.04.003},
}

@book {BrOz08,
    AUTHOR = {Brown, Nathanial P. and Ozawa, Narutaka},
     TITLE = {{$C^*$}-algebras and finite-dimensional approximations},
    SERIES = {Graduate Studies in Mathematics},
    VOLUME = {88},
 PUBLISHER = {American Mathematical Society, Providence, RI},
      YEAR = {2008},
     PAGES = {xvi+509},
      ISBN = {978-0-8218-4381-9; 0-8218-4381-8},
       DOI = {10.1090/gsm/088},
       URL = {https://doi-org.proxy.lib.uwaterloo.ca/10.1090/gsm/088},
}

@article {BoChIo17,
    AUTHOR = {Boutonnet, R\'emi and Chifan, Ionu\c{t} and Ioana, Adrian},
     TITLE = {I{I{$_1$}} factors with nonisomorphic ultrapowers},
   JOURNAL = {Duke Math. J.},
  FJOURNAL = {Duke Mathematical Journal},
    VOLUME = {166},
      YEAR = {2017},
    NUMBER = {11},
     PAGES = {2023--2051},
      ISSN = {0012-7094,1547-7398},
       DOI = {10.1215/00127094-0000017X},
       URL = {https://doi-org.proxy.lib.uwaterloo.ca/10.1215/00127094-0000017X},
}

@article {Go20,
    AUTHOR = {Goldbring, Isaac},
     TITLE = {On {P}opa's factorial commutant embedding problem},
   JOURNAL = {Proc. Amer. Math. Soc.},
  FJOURNAL = {Proceedings of the American Mathematical Society},
    VOLUME = {148},
      YEAR = {2020},
    NUMBER = {11},
     PAGES = {5007--5012},
      ISSN = {0002-9939,1088-6826},
       DOI = {10.1090/proc/15141},
       URL = {https://doi-org.proxy.lib.uwaterloo.ca/10.1090/proc/15141},
}

@article {Go22,
    AUTHOR = {Goldbring, Isaac},
     TITLE = {The {C}onnes embedding problem: a guided tour},
   JOURNAL = {Bull. Amer. Math. Soc. (N.S.)},
  FJOURNAL = {American Mathematical Society. Bulletin. New Series},
    VOLUME = {59},
      YEAR = {2022},
    NUMBER = {4},
     PAGES = {503--560},
      ISSN = {0273-0979,1088-9485},
       DOI = {10.1090/bull/1768},
       URL = {https://doi-org.proxy.lib.uwaterloo.ca/10.1090/bull/1768},
}

@article {Ju07a,
    AUTHOR = {Jung, Kenley},
     TITLE = {Strongly 1-bounded von {N}eumann algebras},
   JOURNAL = {Geom. Funct. Anal.},
  FJOURNAL = {Geometric and Functional Analysis},
    VOLUME = {17},
      YEAR = {2007},
    NUMBER = {4},
     PAGES = {1180--1200},
      ISSN = {1016-443X,1420-8970},
       DOI = {10.1007/s00039-007-0624-9},
       URL = {https://doi-org.proxy.lib.uwaterloo.ca/10.1007/s00039-007-0624-9},
}

@article {Ju07,
    AUTHOR = {Jung, Kenley},
     TITLE = {Amenability, tubularity, and embeddings into {$
              R^\omega$}},
   JOURNAL = {Math. Ann.},
  FJOURNAL = {Mathematische Annalen},
    VOLUME = {338},
      YEAR = {2007},
    NUMBER = {1},
     PAGES = {241--248},
      ISSN = {0025-5831,1432-1807},
       DOI = {10.1007/s00208-006-0074-y},
       URL = {https://doi.org/10.1007/s00208-006-0074-y},
}

@misc{AIM,
title={AimPL: Classification of group von Neumann algebras, available at http://aimpl.org/groupvonneumann},
year={2018},
}

@article {JuSh07,
    AUTHOR = {Jung, Kenley and Shlyakhtenko, Dimitri},
     TITLE = {Any generating set of an arbitrary property {T} von {N}eumann
              algebra has free entropy dimension {$\le1$}},
   JOURNAL = {J. Noncommut. Geom.},
  FJOURNAL = {Journal of Noncommutative Geometry},
    VOLUME = {1},
      YEAR = {2007},
    NUMBER = {2},
     PAGES = {271--279},
      ISSN = {1661-6952,1661-6960},
       DOI = {10.4171/JNCG/7},
       URL = {https://doi-org.proxy.lib.uwaterloo.ca/10.4171/JNCG/7},
}

@article {Be07,
    AUTHOR = {Bekka, Bachir},
     TITLE = {Operator-algebraic superridigity for {${\rm SL}_n(\mathbb Z)$},
              {$n\geq 3$}},
   JOURNAL = {Invent. Math.},
  FJOURNAL = {Inventiones Mathematicae},
    VOLUME = {169},
      YEAR = {2007},
    NUMBER = {2},
     PAGES = {401--425},
      ISSN = {0020-9910,1432-1297},
       DOI = {10.1007/s00222-007-0050-5},
       URL = {https://doi-org.proxy.lib.uwaterloo.ca/10.1007/s00222-007-0050-5},
}

@incollection {GoHa23,
    AUTHOR = {Goldbring, Isaac and Hart, Bradd},
     TITLE = {A survey on the model theory of tracial von {N}eumann
              algebras},
 BOOKTITLE = {Model theory of operator algebras},
    SERIES = {De Gruyter Ser. Log. Appl.},
    VOLUME = {11},
     PAGES = {133--157},
 PUBLISHER = {De Gruyter, Berlin},
      YEAR = {[2023] \copyright 2023},
      ISBN = {978-3-11-076821-3; 978-3-11-076828-2; 978-3-11-076833-6},
}

@incollection {Go23b,
    AUTHOR = {Goldbring, Isaac},
     TITLE = {Spectral gap and definability},
 BOOKTITLE = {Beyond first order model theory. {V}ol. {II}},
     PAGES = {103--137},
 PUBLISHER = {CRC Press, Boca Raton, FL},
      YEAR = {2023},
      ISBN = {978-0-367-20826-4; 978-1-032-51601-1; 978-0-429-26363-7},
}

@article {GoJeKEPi25,
    AUTHOR = {Goldbring, Isaac and Jekel, David and Kunnawalkam Elayavalli,
              Srivatsav and Pi, Jennifer},
     TITLE = {Uniformly super {M}c{D}uff {${\rm II}_1$} factors},
   JOURNAL = {Math. Ann.},
  FJOURNAL = {Mathematische Annalen},
    VOLUME = {391},
      YEAR = {2025},
    NUMBER = {2},
     PAGES = {2757--2781},
      ISSN = {0025-5831,1432-1807},
}

@article {Po07,
    AUTHOR = {Popa, Sorin},
     TITLE = {On {O}zawa's property for free group factors},
   JOURNAL = {Int. Math. Res. Not. IMRN},
  FJOURNAL = {International Mathematics Research Notices. IMRN},
      YEAR = {2007},
    NUMBER = {11},
     PAGES = {Art. ID rnm036, 10},
      ISSN = {1073-7928,1687-0247},
}

@article {Co80,
    AUTHOR = {Connes, A.},
     TITLE = {A factor of type {${\rm II}\sb{1}$}\ with countable
              fundamental group},
   JOURNAL = {J. Operator Theory},
  FJOURNAL = {Journal of Operator Theory},
    VOLUME = {4},
      YEAR = {1980},
    NUMBER = {1},
     PAGES = {151--153},
      ISSN = {0379-4024},
}

@article {Co76,
    AUTHOR = {Connes, A.},
     TITLE = {Classification of injective factors. {C}ases {$II\sb{1},$}
              {$II\sb{\infty },$} {$III\sb{\lambda },$} {$\lambda \not=1$}},
   JOURNAL = {Ann. of Math. (2)},
  FJOURNAL = {Annals of Mathematics. Second Series},
    VOLUME = {104},
      YEAR = {1976},
    NUMBER = {1},
     PAGES = {73--115},
      ISSN = {0003-486X},
       DOI = {10.2307/1971057},
       URL = {https://doi.org/10.2307/1971057},
}

@article {ChIoKE23,
    AUTHOR = {Chifan, Ionu{\c t{}} and Ioana, Adrian and Kunnawalkam Elayavalli, Srivatsav},
     TITLE = {An exotic {$\rm II_1$} factor without property gamma},
   JOURNAL = {Geom. Funct. Anal.},
  FJOURNAL = {Geometric and Functional Analysis},
    VOLUME = {33},
      YEAR = {2023},
    NUMBER = {5},
     PAGES = {1243--1265},
      ISSN = {1016-443X,1420-8970},
       DOI = {10.1007/s00039-023-00649-4},
       URL = {https://doi.org/10.1007/s00039-023-00649-4},
}

@article {HoIo24,
    AUTHOR = {Houdayer, Cyril and Ioana, Adrian},
     TITLE = {Asymptotic freeness in tracial ultraproducts},
   JOURNAL = {Forum Math. Sigma},
  FJOURNAL = {Forum of Mathematics. Sigma},
    VOLUME = {12},
      YEAR = {2024},
     PAGES = {Paper No. e88, 22},
      ISSN = {2050-5094},
       DOI = {10.1017/fms.2024.93},
       URL = {https://doi.org/10.1017/fms.2024.93},
}

@article {IoPePo08,
    AUTHOR = {Ioana, Adrian and Peterson, Jesse and Popa, Sorin},
     TITLE = {Amalgamated free products of weakly rigid factors and
              calculation of their symmetry groups},
   JOURNAL = {Acta Math.},
  FJOURNAL = {Acta Mathematica},
    VOLUME = {200},
      YEAR = {2008},
    NUMBER = {1},
     PAGES = {85--153},
      ISSN = {0001-5962,1871-2509},
       DOI = {10.1007/s11511-008-0024-5},
       URL = {https://doi.org/10.1007/s11511-008-0024-5},
}

@article {Pe09,
    AUTHOR = {Peterson, Jesse},
     TITLE = {{$L^2$}-rigidity in von {N}eumann algebras},
   JOURNAL = {Invent. Math.},
  FJOURNAL = {Inventiones Mathematicae},
    VOLUME = {175},
      YEAR = {2009},
    NUMBER = {2},
     PAGES = {417--433},
      ISSN = {0020-9910,1432-1297},
       DOI = {10.1007/s00222-008-0154-6},
       URL = {https://doi-org.proxy.lib.uwaterloo.ca/10.1007/s00222-008-0154-6},
}

@article {ChSi13,
    AUTHOR = {Chifan, Ionut and Sinclair, Thomas},
     TITLE = {On the structural theory of {${\rm II}_1$} factors of
              negatively curved groups},
   JOURNAL = {Ann. Sci. \'Ec. Norm. Sup\'er. (4)},
  FJOURNAL = {Annales Scientifiques de l'\'Ecole Normale Sup\'erieure.
              Quatri\`eme S\'erie},
    VOLUME = {46},
      YEAR = {2013},
    NUMBER = {1},
     PAGES = {1--33},
      ISSN = {0012-9593,1873-2151},
       DOI = {10.24033/asens.2183},
       URL = {https://doi-org.proxy.lib.uwaterloo.ca/10.24033/asens.2183},
}

@misc{JiNaViWrYu20,
  title         = {{$MIP^*=RE$}},
  author        = {Ji, Zhengfeng and Natarajan, Anand and Vidick, Thomas and
                   Wright, John and Yuen, Henry},
  year          =  {2020},
  howpublished = {ar{X}iv:2001.04383},
}

@misc{GaKEPa25,
Author = {David Gao and Srivatsav Kunnawalkam Elayavalli and Gregory Patchell},
Title = {3-handle construction on II$_1$ factors},
Year = {2025},
howpublished = {ar{X}iv:2504.02003},
}

@misc{GaJeKEPa26,
Author = {David Gao and David Jekel and Srivatsav Kunnawalkam Elayavalli and Gregory Patchell},
Title = {Exotic full factors via weakly coarse bimodules},
Year = {2026},
howpublished = {ar{X}iv:2602.00930},
}

@misc{GoHaSi19,
      title={Correspondences, Ultraproducts and Model Theory}, 
      author={Isaac Goldbring and Bradd Hart and Thomas Sinclair},
      year={2019},
      howpublished={ar{X}iv:1809.00049},
      archivePrefix={arXiv},
      primaryClass={math.LO},
      url={https://arxiv.org/abs/1809.00049}, 
}

@article {Po06b,
    AUTHOR = {Popa, Sorin},
     TITLE = {On a class of type {${\rm II}_1$} factors with {B}etti numbers
              invariants},
   JOURNAL = {Ann. of Math. (2)},
  FJOURNAL = {Annals of Mathematics. Second Series},
    VOLUME = {163},
      YEAR = {2006},
    NUMBER = {3},
     PAGES = {809--899},
      ISSN = {0003-486X,1939-8980},
       DOI = {10.4007/annals.2006.163.809},
       URL = {https://doi-org.proxy.lib.uwaterloo.ca/10.4007/annals.2006.163.809},
}

@article {Po83,
    AUTHOR = {Popa, Sorin},
     TITLE = {Orthogonal pairs of {$\ast $}-subalgebras in finite von
              {N}eumann algebras},
   JOURNAL = {J. Operator Theory},
  FJOURNAL = {Journal of Operator Theory},
    VOLUME = {9},
      YEAR = {1983},
    NUMBER = {2},
     PAGES = {253--268},
      ISSN = {0379-4024},
}

\end{document}